\documentclass[reqno,11pt]{amsart}

\usepackage[all]{xy}
\usepackage{amsthm, amssymb, amsmath, amscd}
\usepackage{braket}
\usepackage{tikz}
\usepackage{threeparttable, mathtools, multirow, makecell}
\usepackage{longtable}
\usepackage[justification=centering]{caption}
\usepackage{tabularx}
\usepackage{geometry}
\usepackage{mathrsfs}
\usepackage[colorlinks=true,citecolor=blue,linkcolor=blue]{hyperref}
\usepackage{array}

\def\a{\alpha}
\def\b{\beta}
\def\c{\gamma}
\def\d{\delta}

\def\l{\lambda}
\def\s{\sigma}

\def\NN{{\mathbb N}}
\def\PP{{\mathbb P}}

\def\ZZ{{\mathbb Z}}

\def\cal{\mathcal}

\def\cA{{\cal A}}
\def\cB{{\cal B}}

\def\cD{{\cal D}}

\def\cF{{\cal F}}

\def\cM{{\cal M}}

\def\cO{{\cal O}}

\def\fm{{\mathfrak m}}

\def\Aut{\operatorname{Aut}}

\def\dim{\operatorname{dim}}

\def\Ext{\operatorname {Ext}}

\def\GL{\operatorname {GL}}

\def\gcd{\operatorname{gcd}}
\def\GKdim{\operatorname{GKdim}}
\def\GL{\operatorname{GL}}
\def\gldim{\operatorname{gldim}}

\def\grmod{\operatorname{grmod}}
\def\GrMod{\operatorname{GrMod}}

\def\id{\operatorname {id}}

\def\ker{\operatorname {ker}}
\def\kdim{\operatorname{Kdim}}

\def\Max{\operatorname{Max}}

\def\mod{\operatorname{mod}}
\def\Mod{\operatorname{Mod}}

\def\Proj{\operatorname{Proj}}

\def\Spec{\operatorname {Spec}}

\def\tors{\operatorname{tors}}

\def\tails{\operatorname{tails}}

\def\uExt{\operatorname{\underline{Ext}}}

\def\uExt{\operatorname{\underline{Ext}}}

\def\uCM{\underline {\operatorname{CM}}}

\def\<{\langle}
\def\>{\rangle}

\def\sD{\mathscr D}
\def\sT{\mathscr T}
\def\sH{\mathscr H}

\def\Projn{\operatorname {Proj_{nc}}}
\def\Specn{\operatorname {Spec_{nc}}}

\def\rnum#1{\expandafter{\romannumeral #1}}
\def\Rnum#1{\uppercase\expandafter{\romannumeral #1}}

\theoremstyle{plain} 
\newtheorem{theorem}{Theorem}[section]
\newtheorem{corollary}[theorem]{Corollary}
\newtheorem{lemma}[theorem]{Lemma}
\newtheorem{proposition}[theorem]{Proposition}

\theoremstyle{definition}
\newtheorem{definition}[theorem]{Definition}
\newtheorem{example}[theorem]{Example}

\theoremstyle{remark}

\newtheorem{remark}[theorem]{Remark}

\numberwithin{equation}{section}

\begin{document}


\title{Classification of noncommutative central conics} 

\author{Haigang Hu, Izuru Mori, Wenchao Wu}

\address{\parbox{16.19cm}{Hu: School of Mathematical Sciences, University of Science and Technology of China, 
Hefei, Anhui 230026, CHINA}}
\email{huhaigang@ustc.edu.cn; huhaigang\_phy@163.com}

\address{\parbox{16.19cm}{Mori: Department of Mathematics, Faculty of Science, Shizuoka University, 
Shizuoka 422-8529, JAPAN}}
\email{mori.izuru@shizuoka.ac.jp}


\address{\parbox{16.19cm}{Wu: School of Mathematical Sciences, University of Science and Technology of China, 
Hefei, Anhui 230026, CHINA\\
and\\
Department of Information Science and Technology, Shizuoka University, 
 Shizuoka 422-8529, JAPAN}}
\email{wuwch20@mail.ustc.edu.cn}



\thanks{The first author was supported by the NSF of China (Grant Nos.\,12371042 and 12501054). The second author was supported by JSPS Grant-in-Aid for Scientific Research (C)  Grant Number JP25K06917. The third author was supported by China Scholarship Council (CSC)}


\begin{abstract}  
Classification of noncommutative quadric hypersurfaces is one of the major projects in noncommutative algebraic geometry. 
In recent years, we are dedicated to complete the classification of noncommutative central conics. 
To achieve this goal, we and other authors develop some theories to study and classify some classes of noncommutative quadric hypersurfaces 
in a series of papers \cite{H,HMM,HM,HMTW,HWY, MU3, MU}. 
Finally, in this paper,  we completely classify noncommutative central conics by developing the general theory of homogenization and dehomogenization for noncommutative algebras and by previous results.   As a main result, we show that there are bijections among the following sets of objects (i) the set of isomorphism classes of $4$-dimensional Frobenius algebras, (ii) the set of isomorphism classes of noncommutative affine pencils of conics, and (iii) the set of isomorphism classes of noncommutative central conics.
\end{abstract} 

\maketitle

\tableofcontents


\section{Introduction}

Throughout this paper, let $k$ be an algebraically closed field of characteristic $0$.
All algebras and vector spaces are over $k$ unless otherwise mentioned.  

Noncommutative algebraic geometry introduced by Artin, Schelter, Tate, Van den Bergh, et al, is a vibrant field that extends classical geometric concepts to noncommutative settings, offering new insights and tools for understanding algebraic structures \cite{AS, ATV, AZ, Rog}. A central theme in this area is the study of noncommutative analogues of classical geometric objects, such as quantum polynomial algebras, noncommutative quadratic complete intersections, etc \cite{AS, HM, KKZ, SV}. 

In recent years, we are dedicated to complete the classification of noncommutative central conics. 
In fact, we classify noncommutative central conics in Calabi-Yau quantum projective planes \cite {H, HMM}, develop some theory about Clifford quadratic complete intersections \cite{HM}, classify 4-dimensional Frobenius algebras \cite{HMTW}, develop the theory of noncommutative affine pencils of conics \cite{HMTW}, show that every noncommutative conic satisfies (G1) condition \cite{HWY}, etc.  

A noncommutative (Calabi-Yau) projective plane is defined by the noncommutative projective scheme $\Projn S$ associated to a 3-dimensional (Calabi-Yau) quantum polynomial algebra $S$, and a noncommutative (Calabi-Yau) (central) conic is defined by the noncommutative projective scheme $\Projn S/(f)$ where $S$ is a 3-dimensional (Calabi-Yau) quantum polynomial algebra and $f\in S_2$ is a homogeneous normal (central) element of degree 2.  Although there are infinitely many isomorphism classes of noncommutative Calabi-Yau projective planes, it is rather surprising that we showed  in our previous paper \cite{HMM} that there are only 9 isomorphism classes of noncommutative Calabi-Yau central conics. 
In order to finish the classification of noncommutative central conics, we have to drop the Calabi-Yau condition. However, it turns out to be very difficult to classify the non-Calabi-Yau cases. Fortunately, after the previous work about Clifford quadratic complete intersections, noncommutative affine pencils of conics, point varieties, etc, we are getting closer to our goal. 
Finally, in this paper, by developing the techniques of noncommutative homogenization and dehomogenization, and by using previous results, 
we are able to completely classify the noncommutative central conics. 

We recall some background about noncommutative algebraic geometry. A quantum polynomial algebra defined below is a noncommutative analogue of a polynomial algebra in noncommutative algebraic geometry.

\begin{definition}[\cite{AS}] \label{def.qpa}
A noetherian connected graded algebra $S$ is called an $n$-dimensional quantum polynomial algebra if 
\begin{enumerate}
\item{} $\gldim S=n$, 
\item{} $\uExt^i_S(k, S)=\begin{cases}  k(n) & \textnormal {if } i=n, \\
0 & \textnormal {if } i\neq n, \end{cases}$ and
\item{} $H_S(t)=(1-t)^{-n}$.
\end{enumerate} 
\end{definition}

The noncommutative projective scheme associate to an $n$-dimensional quantum polynomial algebra in the sense of Artin-Zhang \cite{AZ} is regarded as a noncommutative analogue of $\PP^{n-1}$ in noncommutative algebraic geometry.  It
plays the role of the ambient space where noncommutative hypersurfaces are embedded into.

In this paper, we mainly study the following noncommutative objects. 

\begin{definition}[Definition \ref{defn-nci}, Section \ref{subsec-srns}]
$$
\cA_{n,m} = \left\{ 
\begin{tabular}{l} 
$A = S/(h_1, \dots, h_m)$ where
$S$ is an $n$-dimensional quantum polynomial algebra, \\
and $(h_1, \dots, h_m)$ is a homogeneous  regular normal sequence of degree $2$ in $S$
\end{tabular}
\right\} / \cong,
$$
and
$$
\cA^\vee_{n,m} = \left\{ 
\begin{tabular}{l} 
$R = S/(f_1, \dots, f_m)$ where
$S$ is an $n$-dimensional quantum polynomial algebra, \\
and $(f_1, \dots, f_m)$ is a strongly regular normal sequence of degree $2$ in $S$
\end{tabular}
\right\} / \cong.
$$
\end{definition}

We call a graded algebra $A\in \cA_{3, 1}$ the {\it homogeneous coordinate ring of a noncommutative conic}, and $\Projn A$ {\it a noncommutative conic}.  Following \cite{HMTW},  we call an algebra $R\in \cA_{2, 2}^\vee$ the {\it coordinate ring of a noncommutative affine pencil of conics}, and $\Specn R$ {\it a noncommutative affine pencil of conics}.

For a noncommutative quadric hypersurface $A = S/(f) \in \cA_{n,1}$, there is a unique regular normal element $f^! \in A^!_2$ such that $A^!/(f^!) = S^!$, where $(-)^!$ is the operation by taking quadratic dual. We define $C(A) : = A^![(f^!)^{-1}]_0$. Smith and Van den Bergh showed that Cohen-Macaulay representation $\uCM^{\ZZ}(A)$ is equivalent to $\cD^b(\mod C(A))$ (\cite[Proposition 5.2]{SV}, see also \cite[Lemma 4.13]{MU}). We give a proof of the well-known result that $C(A)$ is a self-injective algebra, and that $C(A)$ is a 4-dimensional Frobenius algebra when $A\in \cA_{3, 1}$ (Theorem \ref{thm.37}). 

We use techniques of noncommutative homogenization and dehomogenization in this paper. 
One subtle point about these techniques is that a homogenization of an algebra
$R$ depends on the choice of defining relations of $R$ (so it is not uniquely determined by $R$), 
and the dehomogenization of a graded algebra $A$ depends on a choice of a regular normal (central) element of degree $1$ in $A$ (so it is not uniquely determined by $A$). In fact, even in the commutative case, they are not inverse operations to each other, and in the noncommutative case, these operations are  even more complicated. (Many examples violating our naive intuition are given in this paper and in our previous paper \cite{HMTW}.)  But for the purpose of this paper, we partly solve this problem by introducing the notion of strongly regular normal sequence.  We hope this part of the paper can be a useful reference for these techniques.

We show that the homogenization $\sH^z(R)$ of a noncommutative affine pencil of conics $R = S/(f_1,f_2)\in \cA_{2, 2}^\vee$
is independent of the choice of $S$ and a strongly regular normal sequence $(f_1,f_2)$ in $S$, and $\sH^z(R)\in \cA_{3, 2}$ 
(Theorem \ref{thm.csfg}). 
We also show that the dehomogenization $\sD_{w}(A)$ of $A \in \cA_{n,n-1}$ is independent of the choice of the regular normal element $w$ (Lemma \ref{lem.awwd}). 

Let
\begin{align*}
\cA^c_{3,1} & := \{A = S/(f) \in \cA_{3,1} \mid f \in Z(S)_2 \} / \cong, \\
\cA^z_{3,1} & := \{A \in \cA^c_{3,1} \mid 
\exists
\textnormal { a regular central element $z\in Z(A^!)_1$}
 \} / \cong.
\end{align*}
We call a graded algebra $A\in \cA_{3, 1}^c$ the {\it homogeneous coordinate ring of a noncommutative central conic}.
The crucial part of this paper is to show that there is a bijection between $\cA_{3, 1}^z$ and $\cA_{2, 2}^\vee$ where techniques of noncommutative homogenization and dehomogenization play essential roles.  
We prove that the composition
$$
\xymatrix{
\cA^z_{3,1} \ar[r]^-{(-)^!} & \cA_{3,2} \ar[r]^-{\sD_z} & \cA^\vee_{2,2}
}
$$
is a bijection having the inverse (Theorem \ref{thm.acb})
$$
\xymatrix{
\cA^\vee_{2,2} \ar[r]^-{\sH^z} &\cA_{3,2} \ar[r]^-{(-)^!} &\cA^z_{3,1}.
}
$$

Combining the techniques of noncommutative homogenization and dehomogenization together with the explicit calculation results of regular central elements of degree $2$ of $3$-dimensional quantum polynomial algebras and regular normal elements of degree $1$ in the dual algebras of $A \in \cA^c_{3,1}$ in Section \ref{sec-clafi},
we have the following result.  

\begin{theorem} [Theorem \ref{con.main}] \label{intro-con.main}
For algebras $A, A' \in \cA^c_{3, 1}$, the following are equivalent: 
\begin{enumerate}
\item{} $\GrMod A\cong \GrMod A'$. 
\item{} $\Projn A\cong \Projn A'$. 
\item{} $\GrMod A^!\cong \GrMod {A'}^!$. 
\item{} $\Projn A^!\cong \Projn {A'}^!$. 
\item{} $\uCM^{\ZZ}(A)\cong \uCM^{\ZZ}(A')$.  
\item{} $\cD^b(\mod C(A))\cong \cD^b(\mod C(A'))$. 
\item{} $C(A)\cong C(A')$. 
\end{enumerate} 
\end{theorem} 

The main result of this paper is the following bijections. 

\begin{theorem}[{Corollary \ref{cor-main.bij}}] \label{intro-main.bij}
    There are bijections among the following sets:
\begin{itemize}
    \item [(i)] the set of isomorphism classes of 4-dimensional Frobenius algebras.
    \item [(ii)] the set of isomorphism classes of noncommutative affine pencils of conics.
    \item [(iii)] the set of isomorphism classes of noncommutative central conics.
\end{itemize} 
\end{theorem}

In \cite{HMTW}, Takeda and the authors give complete classifications of 4-dimensional Frobenius algebras and noncommutative affine pencils of conics, so in this paper, we can give a complete classification of noncommutative central conics by the above bijections.  

\begin{theorem}[{Theorem \ref{thm-iso-conics}}] Every noncommutative central conic is isomorphic to the noncommutative projective scheme associated to exactly one of graded algebras in Table \ref{intro.tab.alg}, where $\Projn \mathcal{S}_\lambda/(x^2) \cong \Projn \mathcal{S}_{\lambda'}/(x^2)$ if and only if $\l' = \l^{\pm1}$. 

\begin{center}
\begin{table}[ht]
\begin{threeparttable}
\centering
\caption{$A \in A^c_{3,1}/\sim$.} \label{intro.tab.alg}
\begin{tabular}{|c|}
\hline
{\rm Commutative}  \\ \hline 
$k[x,y,z]/(x^2)$, $k[x,y,z]/(x^2 + y^2)$, $k[x,y,z]/(x^2 + y^2 + z^2)$ \\ \hline \hline
{\rm Not-commutative}  \\ \hline 
$\mathcal{W}/(x^2)$, $\mathcal{S}_\lambda/(x^2)$, $\mathcal{S}/(x^2+y^2)$, $\mathcal{S}/(x^2+y^2+z^2)$, 
$\mathcal{N}/(x^2)$, $\mathcal{N}/(x^2 + y^2 - 4 z^2)$, \\ $\mathcal{N}/(3x^2 + 3y^2 + 4z^2)$ \\\hline
\end{tabular}
\begin{tablenotes}
\linespread{1}
\item\hspace*{-\fontdimen2\font} 
$\mathcal{S} := k \< x,y,z \> /(yz+zy,zx + xz, xy+yx)$, \\
$\mathcal{S}_\l := k \< x,y,z \> /(yz-\l zy,zx - xz, xy - yx)$, $\l \neq 0,1$,\\
$\mathcal{W} := k \< x,y,z \> /(yz-zy-y^2,zx - xz, xy-yx)$, \\
$\mathcal{N} := k \< x,y,z \> /(yz+zy+x^2,zx + xz+y^2, xy+yx)$. 
\end{tablenotes}
\end{threeparttable}
\end{table}
\end{center}
\end{theorem}

\noindent{\bf Acknowledgments.}
We would like to thank Saya Ogawa, Ryoma Suzuki, Ryota Takahashi and Koki Takeda for the helpful discussion and for helping our computations.

\section{Preliminaries} 
 
\subsection{Quasi-schemes} 

\begin{definition} 
A {\it quasi-scheme} $X=(\mod X, \cO_X)$ over $k$ consists of a $k$-linear abelian category $\mod X$ and an object $\cO_X\in \mod X$.  We say that two quasi-schemes $X$ and $Y$ are isomorphic, denoted by $X\cong Y$, if there exists an equivalence functor $F:\mod X\to \mod Y$ such that $F(\cO_X)\cong \cO_Y$.   
\end{definition} 

If $X$ is a noetherian scheme over $k$, then we view $X$ as a quasi-scheme $X=(\mod X, \cO_X)$ over $k$ where $\mod X$ is the category of coherent sheaves on $X$, and $\cO_X\in \mod X$ is the structure sheaf on $X$.  There are two typical examples of a quasi-scheme in noncommutative algebraic geometry, namely, a noncommutative affine scheme and a noncommutative projective scheme.

For an algebra $R$, we denote by $\Mod R$ the category of right $R$-modules, and by $\mod R$ the full subcategory consisting of finitely generated modules. If $R$ is right noetherian, the {\it noncommutative affine scheme} associated to $R$ is defined by $\Specn R = (\mod R, R)$. The following easy lemma is crucial in this paper. 

\begin{lemma} [{\cite[Lemma 5.2]{HMM}}]  \label{lem.spnc} For right noetherian algebras $R$ and $R'$,  $\Specn R\cong \Specn R'$ as quasi-schemes if and only if $R\cong R'$ as algebras.  
\end{lemma}

In this paper, a graded algebra $A$ means a $\mathbb{Z}$-graded algebra $A = \bigoplus_{i \in \ZZ} A_i$.  A connected graded algebra $A$ is an $\NN$-graded algebra $A = \bigoplus_{i \in \NN} A_i$ such that $A_0 = k$.   
When we write a graded algebra $A$ of the form $k \<x_1, \dots, x_n \>/I$,  we implicitly assume $\deg x_i = 1$ so that $A$ is a connected graded algebra finitely generated in degree $1$.
For a graded algebra $A$, we denote by $\GrMod A$ the category of graded right $A$-modules, where a morphism is a right $A$-module homomorphism preserving degrees, and by $\grmod A$ the full subcategory of $\GrMod A$ consisting of finitely generated graded right $A$-modules. 
For a right noetherian connected graded algebra $A$, the {\it noncommutative projective scheme} associated to $A$ is defined by $\Projn  A : = (\tails A, \cA)$, where $\tors A$ is the full subcategory of $\grmod A$ consisting of finite dimensional modules over $k$,  $\tails A: = \grmod A / \tors A$ is the quotient category, and $\cA$ is the image of $A\in \grmod A$ in $\tails A$.

Let $A$ be a graded algebra. 
For $M \in \GrMod A$ and $j\in \ZZ$, we define the shift $M(j) \in \GrMod A$ by $M(j)_i : = M_{j+i}$ for $i \in \mathbb{Z}$.  
For $M, N \in \GrMod A$, we write $\Ext^i_A(M, N):=\Ext^{i}_{\GrMod A}(M,N)$
and 
$$
\uExt^i_A (M,N): = \bigoplus_{j \in \mathbb{Z}} \Ext^i_A (M, N(j)).
$$
We say that $M\in \GrMod A$ is locally finite if $\dim _kM_i<\infty$ for all $i\in \ZZ$, and in this case, we define the Hilbert series of $M$ by  
$
H_M(t) : = \sum_{i \in \mathbb{Z}} (\dim_{k} M_i) t^i \in \mathbb{Z}[[t,t^{-1}]].
$
The Gelfand--Kirillov dimension (GK-dimension for short) of $A$ is defined to be 
\[
\mathrm{GKdim} A := \inf\{d \in \mathbb{R}^{+} \mid \dim_{k} A_{i} \leq c i^{d - 1} \text{ for some constant } c > 0, i \gg 0\}.
\]

A quantum polynomial algebra defined below is a noncommutative analogue of a commutative polynomial algebra in noncommutative algebraic geometry.  

\begin{definition} \label{def.qpa}
A noetherian connected graded algebra $S$ is called an $n$-dimensional quantum polynomial algebra if 
\begin{enumerate}
\item{} $\gldim S=n$, 
\item{} $\uExt^i_S(k, S)=\begin{cases}  k(n) & \textnormal {if } i=n, \\
0 & \textnormal {if } i\neq n, \end{cases}$ and
\item{} $H_S(t)=(1-t)^{-n}$.
\end{enumerate} 
\end{definition}

If $S$ is an $n$-dimensional quantum polynomial algebra, then $\Projn S$ is regarded as a noncommutative analogue of $\PP^{n-1}$ in noncommutative algebraic geometry. 

\begin{example} \label{exm-2dimqpa}
A graded algebra $A$ is a $2$-dimensional quantum polynomial algebra if and only if 
$$
A \cong k \langle x,y \rangle / (ax^2 +bxy+cyx+dy^2), 
$$
where $a,b,c,d \in k$ such that $ad \neq bc$, if and only if $A$ is isomorphic to one of the following 
$$
k_{\l}[x,y] : = k \langle x,y \rangle / (xy - \l yx), \; k_J[x,y] := k \langle x,y \rangle /(x^2 - xy + yx),
$$ 
where $0 \neq \l \in k$. It is known that $k_{\l}[x, y]\cong k_{\l'}[x, y]$ as graded algebras if and only if $\l'=\l^{\pm 1}$.
\end{example}

The following lemma is well-known. 

\begin{lemma}\label{lem.Ore} 
If $S$ is an $n$-dimensional quantum polynomial algebra, then $S[z]$ is an $n+1$-dimensional quantum polynomial algebra. 
\end{lemma}  

\subsection{Regular and normal elements}
\begin{definition} Let $R$ be an algebra and $f\in R$.
\begin{enumerate}
\item{} We say that $f$ is right (left) regular if, for every $g\in R$,  $gf=0$ ($fg=0$) implies $g=0$.   
We say that $f$ is regular if it is both right and left regular. 
\item{} We say that $f$ is normal if $Rf=fR$.
\end{enumerate} 
\end{definition} 

Note that if $f\in R$ is a normal element, then $Rf=fR=(f)\lhd R$ is a two-sided ideal.  We denote by $RN(R)$ ($RZ(R)$) the set of all regular normal (central) elements which are not units.  If $A$ is a graded algebra, then we denote by $RN(A)_d$ ($RZ(A)_d$) the set of all homogeneous regular normal (central) elements of degree $d$ which are not units. 

In this paper, it is essential to find all regular normal elements for a given algebra.  It is not easy to determine if a given element is regular or normal in general.  Here are some criteria.  

\begin{lemma} \label{lem.reno}
Let $R$ be an algebra. 
\begin{enumerate}
\item{} 
If there exists a surjective map $\nu: R \to R$ such that $fw=w\nu (f)$ for $f\in R$, then $w\in R$ is a normal element.  In particular, if $R = k\<x_1, \dots, x_n\>/I$ and 
if there exists $(a_{ij})\in \GL_n(k)$ such that $x_iw=w\left(\sum_{j=1}^na_{ij}x_j\right)$, then $w$ is normal.
\item{}  $w\in R$ is a regular normal element if and only if there exists a unique $\nu\in \Aut R$ such that $fw=w\nu(f)$ for $f\in R$.  We call $\nu$ the {\it normalizing automorphism} of $w$. 
\end{enumerate}
\end{lemma}

A regular normal element $w\in R$ is central if and only if $\nu=\id$.  
For a  graded algebra $A$ and a homogeneous regular normal element $w\in A$, $\nu\in \Aut^{\ZZ}A $ is a graded algebra automorphism.

\begin{definition}  \label{defn-regseq}
Let $R$ be an algebra, and $F = (f_1, \dots, f_m)$ a sequence in $R$.
\begin{enumerate}
\item[(1)] We say $F$ is {\it regular}  if $\bar f_i\in R/(f_1, \dots, f_{i-1})$ is regular for every $i=1, \dots, m$. 
\item[(2)] We say $F$ is {\it normal (resp. central)} if $\bar f_i\in R/(f_1, \dots, f_{i-1})$ is normal (resp. central) for every $i=1, \dots, m$.  
\end{enumerate} 
In addition, let $R$ be a graded algebra.
\begin{enumerate}
\item[(3)] We say $F$ is of degree $d$, denoted by $\deg F = d$, if $\deg f_i = d$  for every $i=1, \dots, m$.
\item[(4)] We say $F$ is {\it homogeneous} if $f_i$ is homogeneous for every $i=1, \dots, m$.
\end{enumerate} 
We denote by $I_F\lhd R$ the (homogeneous) ideal of $R$ generated by $\{f_1, \dots, f_m\}$,
and by $(F, g) : = (f_1, \dots, f_m, g)$ the sequence in $R$ obtained by adding one more element $g \in R$. 
\end{definition}

For an algebra $R$, a sequence  $F = (f_1,\dots,f_m)$ in $R$ can be regarded as a sequence in $R[z]$ by the canonical inclusion $R \to R[z]$.  The following lemma is obvious.

\begin{lemma} \label{lem-suci}
Keep the notations as above. The following are equivalent: 
\begin{enumerate}
\item{} $F$ is regular (resp. normal or central) in $R$. 
\item{} $F$ is regular (resp. normal or central) in $R[z]$. 
\item{} $(F,z)$ is regular (resp. normal or central) in $R[z]$. 
\end{enumerate}
\end{lemma}

\begin{lemma}[{\cite[Lemma 3.10]{HMTW}}] \label{lem.Tak}  
Let $A$ be a locally finite $\NN$-graded algebra, and $F = (f_1, \dots, f_m)$ a homogeneous normal sequence.
Then $F$ is regular if and only if 
$$
H_{A/I_F}(t)=(1-t^{d_1})\cdots (1-t^{d_m})H_A(t)
$$
where $d_i = \deg f_i$ for $i = 1, \dots, m$. 
\end{lemma}

\begin{remark} \label{rem.52} 
For $f, g\in k[x_1, \dots, x_n]\setminus k$, $\gcd(f, g)=1$ if and only if  $(f, g)$ is a regular sequence, so $(f, g)$ is a regular sequence if and only if $(g, f)$ is a regular sequence (\cite[Remark 3.8 (1)]{HMTW}). However, unlike \cite[Lemma 2.28]{HM}, a (homogeneous) regular normal sequence is not preserved by permutations for a noncommutative (graded) algebra (see \cite[Example 3.21 (2)]{HMTW}). 
\end{remark}

\subsection{Noncommutative complete intersections}

\begin{definition} A noetherian connected graded algebra $A$ is called an {\it AS-Gorenstein algebra of dimension $n$ and of Gorenstein parameter $\ell$} if 
\begin{enumerate}
\item{} $\id_AA=\id_{A^o}A=n$, and 
\item{} $\uExt^i_A(k, A)\cong \uExt^i_{A^o}(k, A)\cong \begin{cases} k(\ell) & \text{ if } i= n \\ 0 & \text { if }  i\neq n.\end{cases}$
\end{enumerate}
\end{definition} 

\begin{definition}
A finite dimensional algebra $R$ is called a {\it Frobenius} algebra if there is a bilinear form 
$$
(-,-): R \times R \to k
$$ 
satisfying the following conditions:
\begin{itemize}
\item[(1)] Associative: $(ab,c) = (a,bc)$ for all $a,b,c \in R$. 
\item[(2)] Nondegenerate: for any $0\neq a \in R$, there are $b, b' \in R$ such that $(a,b) \neq 0$ and $(b',a) \neq 0$. 
\end{itemize}
\end{definition}

\begin{example} \label{ex-4dimsjfro} It is known that every finite dimensional basic self-injective algebra is Frobenius (\cite[Proposition IV 3.9]{SY}).  Since $M_2(k)$ is the only non-basic algebra up to dimension 4 
and $M_2(k)$ is Frobenius, every self-injective algebra up to dimension 4 is Frobenius.  
\end{example} 

\begin{lemma} [{\cite[Remark 3.11]{MM}}]  \label{lem.FAS} 
A connected graded algebra $A$ is an AS-Gorenstein algebra of dimension $0$ if and only if $A$ is a graded Frobenius algebra. 
\end{lemma}

\begin{lemma}[{\cite[Propositions 3.4 and 3.5]{Le}}] \label{lem.Ree} 
Let $A$ be a connected graded algebra and $f \in RN(A)_d$.
\begin{enumerate}
\item{} $A$ is noetherian if and only if $A/(f)$ is noetherian.  
\item{}  {\rm (Ree's lemma)} $A$ is an AS-Gorenstein algebra of dimension $n$ and of Gorenstein parameter $\ell$ if and only if $A/(f)$ is an AS-Gorenstein algebra of dimension $n-1$ and of Gorenstein parameter $\ell-d$.  
\end{enumerate} 
\end{lemma} 

By the above lemma, noncommutative complete intersections defined below are AS-Gorenstein. 

\begin{definition} \label{defn-nci}
A graded algebra $A$ is called a {\it noncommutative complete intersection} if there is a quantum polynomial algebra $S$, and 
a homogeneous regular normal sequence $H = (h_1,\dots, h_m)$  in $S$ such that $A = S/I_H$.
In addition, if $\deg H = 2$, $A$ is called a {\it noncommutative quadratic complete intersection}. 
\end{definition} 

Define
\begin{gather*}
\cA_{n, m} : = \{\text{noncommutative quadratic complete intersections } A = S/(h_1, \dots, h_m) \mid \gldim S = n\} / \cong, \\
\cB_{n, m} : = \{\text{quadratic complete intersections } A = k[x_1, \dots, x_n]/(h_1, \dots, h_m)\} / \cong.
\end{gather*}
We remark that $\cA_{n,0}$ is the set of isomorphism classes of $n$-dimensional quantum polynomial algebras, and $\cB_{n,0} = \{k[x_1, \dots, x_n] \}$.
We call $A$ the {\it homogeneous coordinate ring of a noncommutative conic} 
if $A\in \cA_{3, 1}$.
We also define 
$$
\cA_{n, 1}^c:=\{S/(f) \in \cA_{n,1} \mid  f\in Z(S)_2\}/\cong.
$$ 
We call $A$ the {\it homogeneous coordinate ring of a noncommutative central conic} if $A\in \cA_{3, 1}^c$.  


\section{Noncommutative homogenization and dehomogenization}

In this section, we introduce and study several operations on a sequence $F=(f_1, \dots, f_n)$ in an  algebra $R$. 

\subsection{Noncommutative dehomogenization}

Let $A$ be a graded algebra finitely generated in degree 1.  Then $0\neq z\in Z(A)_1$ if and only if $A$ has a presentation $A=k\<x_1, \dots, x_n\>[z]/I$ for some homogeneous ideal $I\lhd k\<x_1, \dots, x_n\>[z]$.

\begin{definition}[Dehomogenization] \label{defn-deh}
\begin{itemize}
\item[(1)] For $f\in k\<x_1, \dots, x_n\>[z]$,  we define 
$$
f_z:=f(x_1, \dots, x_n, 1)\in k\<x_1, \dots, x_n\>.
$$  
\item[(2)] For a sequence $F = (f_1, \dots, f_m)$ in $k\<x_1, \dots, x_n\>[z]$, we define
$$
F_z = ((f_1)_z, \dots, (f_m)_z), 
$$
which is a sequence in $k\<x_1, \dots, x_n\>$. 
\end{itemize} 
\end{definition} 

For a graded algebra $A=k\<x_1, \dots, x_n\>[z]/I_F$, a naive definition of a dehomogenization of $A$ should be $k\<x_1, \dots, x_n\>/ I_{F_z}$.  However, it is unclear if this definition is independent of a choice of $F$, so we will find an alternate definition.

\begin{definition} Let $A$ be a graded algebra and $w\in RN(A)_d$
with the normalizing automorphism $\nu$.  We define a new algebra 
$A[w^{-1}]_0:=\{fw^{-i}\mid f\in A_{di}, i\in \NN\}$
with the addition and the multiplication 
$$fw^{-i}+gw^{-j}=(fw^j+gw^i)w^{-i-j}, \qquad (fw^{-i})(gw^{-j})=f\nu^i(g)w^{-i-j}$$
for $f\in A_{di}, g\in A_{dj}$. 
\end{definition} 

The following lemma is well-known for commutative algebras.  

\begin{lemma} \label{lem.zin} If $A$ is a graded algebra and $z\in RZ(A)_d$,  
then 
$$
A[z^{-1}]_0\cong A^{(d)}/(z-1)$$ 
as algebras where $A^{(d)}:=\oplus _{i\in \ZZ}A_{di}$ is the $d$-th Veronese algebra. 
\end{lemma} 

\begin{proof} 
Since $z\in A^{(d)}/(z-1)$ is a unit, the natural surjection $A^{(d)}\to A^{(d)}/(z-1)$ induces a ring homomorphism 
$$
\phi:A[z^{-1}]_0=A^{(d)}[z^{-1}]_0\to A^{(d)}[z^{-1}]\to A^{(d)}/(z-1); \; fz^{-i}\mapsto \bar f
$$ 
where $f\in (A^{(d)})_i=A_{di}$.  On the other hand, the ring homomorphism $A^{(d)}\to A^{(d)}[z^{-1}]_0; \; f\mapsto fz^{-i}$ where $f\in (A^{(d)})_i=A_{di}$ induces a ring homomorphism $\psi :A^{(d)}/(z-1)\to A^{(d)}[z^{-1}]_0=A[z^{-1}]_0$ since $z-1\mapsto zz^{-1}-1=0$.  It is easy to check that $\phi$ and $\psi$ are inverses to each other. 
\end{proof} 

\begin{example} 
The above lemma does not hold if $z$ is normal but not central. 
For example, if $A=k_{\l}[x, y]$, then $y\in A_1$ is a regular normal element.  Since $A[y^{-1}]_0\cong k[x]$ for every $\l$ while 
$A/(y-1)\cong \begin{cases} k[x]  & \text{ if } \l =1, \\ 
k & \text{ if }  \l \neq 1, 
\end{cases}$
$A[y^{-1}]_0\cong A/(y-1)$ if and only if $y\in A_1$ is central.  
\end{example}

\begin{lemma} \label{lemq.06}  
If $F = (f_1, \dots, f_m)$ is a homogeneous sequence in $k\<x_1, \dots, x_n\>[z]$, and $A = k\<x_1, \dots, x_n\>[z]/I_F$, then
$$
A/(z-1) \cong  k\<x_1, \dots, x_n\>/ I_{F_z}
$$ 
as algebras.  In particular, if $z\in A_1$ is a regular element,  then $ k\<x_1, \dots, x_n\>/ I_{F_z} \cong A[z^{-1}]_0$. 
\end{lemma} 

\begin{proof} 
It is easy to see that the isomorphism 
$$
k\<x_1, \dots, x_n\>[z]/(z-1)\to k\<x_1, \dots, x_n\>; \; f\to f_z
$$ 
induces an isomorphism 
$$
A/(z-1) = k\<x_1, \dots, x_n\>[z]/(f_1, \dots, f_m, z-1)\to   k\<x_1, \dots, x_n\>/((f_1)_z, \dots, (f_m)_z).
$$  
The last claim follows from Lemma \ref{lem.zin}.   
\end{proof} 

Inspired by Lemma \ref{lemq.06}, we define the dehomogenization of a graded algebra as follows.

\begin{definition}[Dehomogenization  for graded algebras] \label{defn-dehomnor}
Let $A$ be a graded algebra, and $w\in RN(A)_1$.
We define the {\it dehomogenization} of $A$ with respect to $w$ by $\sD_w(A) : = A[w^{-1}]_0$. 
\end{definition}

\begin{remark} 
If $h=x_ix_j-x_jx_i\in k\<x_1, \dots, x_n\>[z]$, then $h_z=h$ in $k\<x_1, \dots, x_n\>$, so, for a commutative graded algebra $B=k[x_1, \dots, x_n, z]/(f_1, \dots, f_m)$,
$$
B/(z-1)\cong k[x_1, \dots, x_n]/((f_1)_z, \dots, (f_m)_z)
$$ 
by Lemma \ref{lemq.06}, which is a usual definition of the dehomogenization for commutative graded algebras. 
In particular, if $z\in B_1$ is regular, then 
 $$
\sD_z(B)\cong k[x_1, \dots, x_n]/((f_1)_z, \dots, (f_m)_z)
$$ 
by  Lemma \ref{lem.zin}.
\end{remark}

\begin{remark} \label{rem-idnor}
Later in Lemma \ref{lem.633}, we show that for a graded algebra $A$ and $w \in RN(A)_1$, 
$$
\sD_w(A) = A[w^{-1}]_0 \cong A^\nu/(w-1)
$$
where $\nu$ is the normalizing automorphism of $A$ and $A^\nu$ is the twisted algebra discussed in subsection \ref{subsec-twalg}, so we will identify $\sD_w(A)$ with $A^\nu/(w-1)$. 
\end{remark}

Let $A$ be a connected graded algebra and $f\in RN(A)_d$. 
The canonical injection 
$
A \to A[f^{-1}]
$
induces the composition of  functors
$$
F: \GrMod A \to \GrMod A[f^{-1}] \to \Mod A[f^{-1}]_0; \; M \mapsto M [f^{-1}] \mapsto M [f^{-1}]_0,
$$
which restricts to the functor $F: \grmod A \to \mod A[f^{-1}]_0$.

\begin{lemma}[{\cite[Proposition 4.6]{MU}}] \label{lem-pitails}
Let $A$ be a right noetherian graded algebra finitely generated in degree $1$ over $k$ and $f\in RN(A)_d$.  
If $\dim_k A/(f) < \infty$, then the functor 
$F: \grmod A \to \mod A[f^{-1}]_0$ defined above induces 
an isomorphism $\Projn A\to \Specn A[f^{-1}]_0$ as quasi-schemes. 
\end{lemma}

In general, the dehomogenization $\sD_w(A)$ depends on the choice of $w\in RN(A)_1$. However, we have the following result. 

\begin{lemma} \label{lem.awwd} 
Let $A, A'$ be right noetherian connected graded algebras finitely generated in degree $1$ over $k$ such that $\GKdim A=\GKdim A'=1$.  If $\GrMod A\cong \GrMod A'$ 
and $w\in RN(A)_1, w'\in RN(A')_1$, then $\sD_w(A)\cong \sD_{w'}(A')$ as algebras.  
In particular, $\sD_w( A)$ is independent of the choice of  $w\in RN(A)_1$ up to isomorphism.   
\end{lemma}  

\begin{proof} 
Since $\GKdim A=\GKdim A'=1$ and $w\in RN(A)_1, w'\in RN(A')_1$,  we have $\dim_k A/(w) < \infty$ and $\dim_k A'/(w') < \infty$, so
$$\Specn\sD_w(A)\cong \Projn A\cong \Projn A'\cong \Specn\sD_{w'}(A')$$ 
as quasi-schemes by Lemma \ref{lem-pitails} and \cite[Theorem 1.4]{Z}. By Lemma \ref{lem.spnc}, $\sD_w(A)\cong \sD_{w'}(A')$ as algebras.  
\end{proof} 

By the above result, we may omit the subscript $w$ from the notation $\sD_w(A)$ when $\GKdim A = 1$. 

\begin{definition}[Localization]
Let $A$ be a graded algebra, and $w\in RN(A)_1$.
For a homogeneous sequence $F = (f_1, \dots, f_m)$ in $A$ where $\deg f_i = d_i$ for $1\leq i \leq m$, we define
$$
F[w^{-1}]_0 : = (f_1w^{-d_1}, \dots, f_mw^{-d_m}),
$$
which is a sequence in $\sD_w(A)=A[w^{-1}]_0$.
\end{definition}

\begin{lemma} \label{lem.lrn} 
Let $A$ be a graded algebra.  If $(f, z)$ is a homogeneous regular normal sequence in $A$ such that  $f \in RN(A)_d$ and $z\in RZ(A)_1$,
then $fz^{-d}\in RN(\sD_z(A))$
and $\sD_{\bar z}(A/(f))\cong \sD_z(A)/(fz^{-d})$.
\end{lemma} 

\begin{proof} Let $\nu\in \Aut^{\ZZ}A$ be the normalizing automorphism of $f\in A$.  Since $z\in Z(A)_1$, $\nu (z)=z$.   For every $r\in A_i$, $rz^{-i}fz^{-d}=rfz^{-i-d}=f\nu(r)z^{-i-d}=fz^{-d}\nu(r)z^{-i}$, $\nu\in \Aut^{\ZZ}A$ induces 
$\bar\nu\in \Aut A[z^{-1}]_0$ by $\bar \nu(rz^{-i}):=\nu(r)z^{-i}$, so  $fz^{-d}\in A[z^{-1}]_0$ is normal by Lemma \ref{lem.reno}. 
If $rz^{-i}fz^{-d}=rfz^{-i-d}=0$ in $A[z^{-1}]_0$, then $rf=0$ in $A$.  Since $f\in A$ is right regular, $r=0$  in $A$, so $rz^{-i}=0$ in $A[z^{-1}]_0$, hence $fz^{-d}$ is right regular.  We can similarly show that $fz^{-d}$ is left regular. If  $fz^{-d}$ is a unit in $A[z^{-d}]_0$, then $fgz^{-d-e}=fz^{-d}gz^{-e}=1$ for some $g\in A_e$, so $fg=z^{d+e}$ in $A$.  Since $f, z$ is a regular sequence, this is impossible, so $fz^{-d}$ is not a unit in $A[z^{-d}]_0$.   

Define a map $\phi:A[z^{-1}]_0\to (A/(f))[\bar z^{-1}]_0; \; rz^{-i}\mapsto \bar r\bar z^{-i}$.  Since $z$ is central, 
\begin{align*}
\phi(r_1z^{-i_1}+r_2z^{-i_2}) & =\phi((r_1z^{i_2}+r_2z^{i_1})z^{-i_1-i_2})=\overline {r_1z^{i_2}+r_2z^{i_1}}\bar z^{-i_1-i_2} \\
& =(\bar r_1\bar z^{i_2}+\bar r_2\bar z^{i_1})\bar z^{-i_1-i_2} =\bar r_1\bar z^{-i_1}+\bar r_2\bar z^{-i_2}=\phi(r_1z^{-i_1})+\phi(r_2z^{-i_2}), \\
\phi(r_1z^{-i_1}r_2z^{-i_2}) & =\phi(r_1r_2z^{-i_1-i_2})=\overline {r_1r_2}\bar z^{-i_1-i_2} \\
& =(\bar r_1\bar r_2)\bar z^{-i_1-i_2} =\bar r_1\bar z^{-i_1}\bar r_2\bar z^{-i_2}=\phi(r_1z^{-i_1})\phi(r_2z^{-i_2}), 
\end{align*}
so $\phi$ is a surjective algebra homomorphism.   For $r\in A_i$, if $r\in (f)$, then $r=tf$ for some $t\in A_{i-d}$, so $rz^{-i}=tfz^{-i}=tz^{d-i}fz^{-d}\in (fz^{-d})$.  Conversely, since $fz^{-d}\in A[z^{-1}]_0$ is normal, if $rz^{-i}\in (fz^{-d})$, then $rz^{-i}=sz^{-j}fz^{-d}=sfz^{-j-d}$ for some $s\in A_j$, so $rz^{j+d-i}=sf\in (f)$.  Since $f, z\in A$ is a regular sequence, $r\in (f)$.  
Since $\bar z\in A/(f)$ is regular, 
$$\phi(rz^{-i})=\bar r\bar z^{-i}=0 \Leftrightarrow \bar r=0 \Leftrightarrow r\in (f) \Leftrightarrow rz^{-i}\in (fz^{-d}),$$
so $\ker \phi=(fz^{-d})\lhd A[z^{-1}]_0$,  hence $\phi$ induces an isomorphism  $A[z^{-1}]_0/(fz^{-d})\cong (A/(f))[\bar z^{-1}]_0$.   
\end{proof}   

\begin{lemma} \label{lem.zfr}
Let $A$ be a graded algebra, $F=(f_1, \dots, f_m)$ a homogeneous sequence of $A$, and $z\in Z(A)_1$.  
If $(F, z)$ is a homogeneous regular normal sequence in $A$, then $F[z^{-1}]_0$ is a regular normal sequence in $\sD_z(A)$.
\end{lemma} 

\begin{proof} 
Since $(f_1, z)$ is a homogeneous regular normal sequence in $A$ such that $z\in Z(A)_1$, $f_1 z^{-d_1}\in A[z^{-1}]_0$ is a regular normal element by Lemma \ref{lem.lrn}.  
Similarly, since $(\bar f_2, \bar z)$ is a homogeneous regular normal sequence in $A/(f_1)$  such that $\bar z\in Z(A/(f_1))_1$, $\bar f_2\bar z^{-d_2}\in  A/(f_1)[\bar z^{-1}]_0 \cong A[z^{-1}]_0/(f_1z^{-d_1})$ is a regular normal element by Lemma \ref{lem.lrn}.  Repeat this process. 
\end{proof} 

\subsection{Noncommutative homogenization}

Comparing with dehomogenization, the homogenization is subtler. We first introduce the homogenization for free algebras.

\begin{definition}[Homogenization for free algebras]
\begin{itemize}
\item[(1)] For $f\in k\<x_1, \dots, x_n\>$ such that $\deg f=d$, we define 
$$
f^z:=f(x_1z^{-1}, \dots, x_nz^{-1})z^d\in k\<x_1, \dots, x_n\>[z]_d. 
$$  
\item[(2)] For  a sequence $F = (f_1, \dots, f_m)$ in  $k\<x_1, \dots, x_n\>$, we define 
$$
F^z:=((f_1)^z, \dots, (f_m)^z),
$$ 
which is a homogeneous sequence in $k\<x_1, \dots, x_n\>[z]$.
Define
$$
\sH^{\dagger, z}(F) := k\<x_1, \dots, x_n\>[z] / I_{F^z}.
$$
\end{itemize}
\end{definition}

If $f=\sum_{i=0}^df_i\in k\<x_1, \dots, x_n\>$ such that $\deg f=d$  
where $f_i\in k\<x_1, \dots, x_n\>_i$, then $f^z=\sum_{i=0}^df_iz^{d-i}\in k\<x_1, \dots, x_n\>[z]_d$.   

The element $z\in Z(\sH^{\dagger, z}(F))$ may not be regular in general as an example below.

\begin{example} \label{exam-znotreg}
Let $F= (x-xy, x^2 y)$ be a sequence in $k\langle x,y \rangle$.
Then 
$
\sH^{\dagger, z}(F) =  k \< x,y\> [z] /(xz-xy, x^2 y) . 
$
Since $x^2z =x^2y = 0$ in $\sH^{\dagger, z}(F)$, $z$ is not regular. 
\end{example}

\begin{lemma} \label{lemq.05} 
\begin{enumerate}
\item{} For $f\in k\<x_1, \dots, x_n\>$, $(f^z)_z=f$.
\item{} For $f\in k\<x_1, \dots, x_n\>[z]_d$, $(f_z)^z z^{d-\deg f_z}=f$.
\end{enumerate}
\end{lemma}

\begin{proof} (1) Since $z$ is central in $k\<x_1, \dots, x_n\>[z]$, $x_iz=zx_i$ so $z^{-1}x_i=x_iz^{-1}$.  For every monomial $f=x_{i_1}\cdots x_{i_d}\in k\<x_1, \dots, x_n\>_d$, 
$$
f^z=(x_{i_1}z^{-1}\cdots x_{i_d}z^{-1})z^d=(x_{i_1}\cdots x_{i_d}z^{-d})z^d=x_{i_1}\cdots x_{i_d}=f,
$$ 
so, for every homogeneous polynomial $f_i\in k\<x_1, \dots, x_n\>_i$, $(f_i)^z=f_i\in k\<x_1, \dots, x_n\>$.  For $f=\sum _{i=0}^df_i\in k\<x_1, \dots, x_n\>$ where $f_{i}\in k\<x_1, \dots, x_n\>_i$ such that $f_d\neq 0$, $(f^z)_z=(\sum_{i=0}^d(f_i)^z z^{d-i})_z=\sum _{i=0}^df_i=f$.  

(2) For $f=\sum _{i=0}^df_{i}z^{d-i}\in k\<x_1, \dots, x_n\>[z]_d$ where $f_{i}\in k\<x_1, \dots, x_n\>_i$, 
\begin{equation*}
(f_z)^z z^{d-\deg f_z}=\left (\sum _{i=0}^df_{i}\right )^z z^{d-\deg f_z}
=\sum_{i=0}^d (f_{i})^z z^{\deg f_z-i}z^{d-\deg f_z}=\sum _{i=0}^df_{i}z^{d-i}=f. \qedhere
\end{equation*}
\end{proof} 

\begin{proposition} \label{prop.05} 
\begin{enumerate}
\item{} If $F = (f_1, \dots, f_m)$ be a sequence in $k\<x_1, \dots, x_n\>$, then 
$$
\sH^{\dagger, z}(F) / (z-1)  \cong k\<x_1, \dots, x_n\> / I_F.
$$
\item{} Let $F = (f_1, \dots, f_m )$ be a homogeneous sequence in $k\<x_1, \dots, x_n\>[z]$, and $A = k\<x_1, \dots, x_n\>[z]/I_F$.
If $z\in RZ(A)_1$, then $\sH^{\dagger, z}(F_z) =  A$.
\end{enumerate}
\end{proposition}

\begin{proof} 
(1) Since $((f_i)^z)_z=f_i$ for every $i=1, \dots, m$ by Lemma \ref{lemq.05} (1),  we have 
$$
\sH^{\dagger, z}(F)/(z-1) \cong k\<x_1, \dots, x_n\>/I_{(F^z)_z}=k\<x_1, \dots, x_n\>/I_F
$$  
by Lemma \ref{lemq.06}. 

(2) Write $f_i = z^{r_i} f_i'$ where $r_i\geq 0$ and $f_i'\in k\<x_1, \dots, x_n\>[z]$ such that $z \nmid f_i'$ for each $i=1, \dots, m$. If $f_i' \notin I_F$, then $r_i>0$, and $z^{r_i} f_i'=f_i=0 $ in $A$ but $f_i'\neq 0$ in $A$, so $z$ is not regular, which is a contradiction.  It follows that  
$f_i' \in I_F$ for every $i = 1, \dots, m$, so  $I_F = I_{F'} \lhd k\<x_1, \dots, x_n\>[z]$ where $F' = (f_1', \dots, f_m')$.  Since  $(f_i)_z=(f_i')_z$, we have $F_z=F'_z$.  By Lemma \ref{lemq.05} (2),  $(F'_z)^z=F'$, so  
\begin{align*}
	\sH^{\dagger, z}(F_z) &=  k\<x_1, \dots, x_n\>[z]/I_{(F_z)^z} =  k\<x_1, \dots, x_n\>[z]/I_{(F'_z)^z}\\
	& =k\<x_1, \dots, x_n\>[z]/I_{F'}  =k\<x_1, \dots, x_n\>[z]/I_{F} = A.  \qedhere
\end{align*} 
\end{proof}


\begin{lemma} \label{lem.barz}
Let $S = k \langle x_1, \dots, x_n \rangle/I$ be a graded algebra, 
and $f, f' \in k \langle x_1, \dots, x_n \rangle$  with $\deg f = \deg f' =d$. 
Then $\overline{f^z} = \overline{(f')^z}$ in $S[z]$  if and only if $\bar{f} = \bar{f'}$ in $S$. 
\end{lemma}

\begin{proof}
If $U$ is a set of homogeneous generators for $I\lhd k\<x_1, \dots, x_n\>$, then  $S[z] = k\langle x_1, \dots, x_n\rangle[z]/(U)$. 
If $\overline{f^z - (f')^z}=\overline{f^z} - \overline{(f')^z} =0$ in $S[z]$, then 
$$
 f^z-(f')^z = \sum_{i=1}^{m}a_ir_ib_i\in (U)\lhd k \langle x_1, \dots, x_n \rangle[z], 
$$ 
where $r_i \in U$, and $a_i,b_i \in k \langle x_1, \dots, x_n \rangle[z]$ for all $i$. Since
$$
f - f' = (f^z)_z-((f')^z)_z
= ( f^z-(f')^z)_z=  \sum_{i=1}^{m}(a_i)_zr_i(b_i )_z\in I\lhd k \langle x_1, \dots, x_n \rangle
$$ 
by Lemma \ref{lemq.05} (1),  $\bar{f} - \bar{f'} = \overline{f - f'} = 0$ in $S$.

Conversely, if $\overline{f - f'} =\bar{f} - \bar{f'} =  0$ in $S$, then 
$$
f - f' =\sum_{i=0}^{d_1} (f_i - f'_i) \in I\lhd k\<x_1, \dots, x_n\>
$$ 
where $d_1=\deg (f-f')$.  Since $I$ is a homogeneous ideal, $f_i - f'_i \in I$ for every $i=0, \dots, d_1$,
so 
$$f^z - (f')^z  = z^{d-d_1}(f - f')^z=z^{d-d_1}\sum_{i=0}^{d_1} z^{d_i-i}(f_i - f'_i)=\sum_{i=0}^{d_1} z^{d-i}(f_i - f'_i) \in (U)\lhd k\langle x_1, \dots, x_n\rangle[z],$$  
hence $\overline{f^z} - \overline{(f')^z} = \overline{f^z - (f')^z}= 0$ in $S[z]$. 
\end{proof}
 
Let $A=S/I_F$ be an (ungraded) algebra where $S$ is a graded algebra and $F$ is a (nonhomogeneous) sequence in $S$.  We want to define a homogenization of $A$

\begin{definition}[Homogenization for algebras]
Let $S = k \langle x_1, \dots, x_n \rangle/I$ be a graded algebra.
\begin{enumerate}
\item{} For $f \in S$, define
$$
f^{S,z} : = \overline{g^z} \in S[z]
$$
where $g \in  k \langle x_1, \dots, x_n \rangle$ such that $\deg g = \deg f$ and $\bar{g} = f$ in $S$. 
\item{} For a sequence $F = (f_1, \dots, f_m)$ in $S$, define
$$
F^{S,z} : = (f_1^{S,z}, \dots, f_m^{S,z}),
$$
which is a homogeneous sequence in $S[z]$. We define a graded algebra 
$$
\sH^z(S, F) := S[z]/I_{F^{S,z}}.
$$
\end{enumerate}
\end{definition}

If there is no potential confusion, we write $f^z$, $F^z$ and $\sH^z(F)$ instead of $f^{S,z}$, $F^{S,z}$ and $\sH^z(S,F)$ respectively. 
Unfortunately, the graded algebra $\sH^z(S, F)$ depends on the choice of $S$ and $F$ such that $A=S/I_F$, which is a main problem to apply the technique of homogenization.  We will partly solve this problem for the purpose of this paper in Subsection 3.4. 

\begin{proposition} \label{prop.05n} 
For a sequence $F = (f_1, \dots, f_m)$ in a graded algebra $S = k \langle x_1, \dots, x_n \rangle/I_{H}$ where $H$ is a homogeneous sequence in $k\<x_1, \dots, x_n\>$, $\sH^z(S, F)/(z-1) \cong S/I_{F}$.
\end{proposition} 

\begin{proof}  
 If  $G=(g_1, \dots, g_m)$ is a sequence in $k\<x_1, \dots, x_n\>$ such that $\deg g_i = \deg f_i$ and $\bar g_i = f_i$ in $S$ for every $i$, then
$$
\sH^z(S, F)/(z-1) \cong \sH^{\dagger, z}(H, G)/(z-1)  \cong k\<x_1, \dots, x_n\>/I_{(H, G)} \cong (k\<x_1, \dots, x_n\>/I_H)/I_F = S/I_F
$$
by Proposition \ref{prop.05} (1). 
\end{proof} 

\begin{corollary} \label{cor.DHc}
Let  $R = k[x_1, \dots, x_n]$, $F$ a sequence in $R[t]$, and $A = R[t]/I_F$. 
\begin{enumerate}
\item{} $\sH^z(R[t], F)/(z-1) \cong A$. 
\item{} If $F$ is homogeneous and $t \in A_1$ is regular, then $\sH^t(R, F_t) \cong A $.      
\end{enumerate}
\end{corollary}

\begin{proof}
(1) is a special case of Proposition \ref{prop.05n}. 

(2) The proof is almost same as the proof of Proposition \ref{prop.05} (2).
\end{proof}

\subsection{Wild homogenization} 

There is another simple-minded way to homogenize algebras.  

\begin{definition}[Wild homogenization] \label{def.Tga} 
\begin{enumerate}
\item{} For $f \in k \langle x_1, \dots, x_n \rangle$ such that $\deg f = d$, define
$$
f^\vee : = f_d \in k \langle x_1, \dots, x_n \rangle. 
$$
\item{} For a sequence $F = (f_1, \dots, f_m)$ in $k \langle x_1, \dots, x_n \rangle$, define
$$
F^\vee : = (f_1^\vee, \dots, f_m^\vee),
$$
which is a homogeneous sequence in $k \langle x_1, \dots, x_n \rangle$. Define
$$
\sT^\dagger(F) : = k \langle x_1, \dots, x_n \rangle / I_{F^\vee}. 
$$
\end{enumerate}
\end{definition}

\begin{lemma}
Let $S = k \langle x_1, \dots, x_n \rangle/I$ be a graded algebra.
For $f, f' \in k \langle x_1, \dots, x_n \rangle$ such that $f^\vee, f'^\vee \notin I$. 
If $\bar{f} = \bar{f'}$ in $S$, then $\overline{f^\vee} = \overline{(f')^\vee}$ in $S$.  
\end{lemma}

\begin{proof}
Let $\deg f= d$, $\deg f' = d'$. By assumption, $f^\vee = f_d, (f')^\vee = f_{d'} \notin I$. 
If $\bar{f} = \bar{f'}$ in $S$, then $d=d'$, and there is $h \in I_d$ such that $f^\vee =f_d = f_{d'} + h=(f')^{\vee}+h$.
\end{proof}

\begin{definition}[Wild homogenization for algebras]
Let $S = k \langle x_1, \dots, x_n \rangle /I$ be a graded algebra. 
\begin{enumerate}
\item{} For $0 \neq f \in S$, define
$$
f^{S,\vee} : = \overline{g_d} \in S,
$$
where $g\in k\<x_1, \dots, x_n\>$ such that $\bar g=f$ in $S$ and $d=\deg f$.
\item{} For a sequence $F = (f_1, \dots, f_m)$ in $S$, define
$$
F^{S,\vee} : = (f_1^{S,\vee}, \dots, f_m^{S,\vee}),
$$
which is a homogeneous sequence in $S$. We define a graded algebra
$$
\sT(S, F) = S/I_{F^{S,\vee}}. 
$$
\end{enumerate}
\end{definition}

Again, the graded algebra $\sT(S, F)$ depends on the choice of $S$ and $F$ such that $A=S/I_F$. 
If there is no potential confusion, we write $f^\vee$, $F^\vee$ and $\sT(F)$ instead of $f^{S,\vee}$, $F^{S,\vee}$ and $\sT(S,F)$ respectively.

The following lemma is useful.

\begin{lemma} \label{lem-tchz}
\begin{enumerate}
\item{} For a sequence $F$ in $k\langle x_1, \dots, x_n \rangle$,
$
\sH^{\dagger,z}(F)/(z) \cong \sT^\dagger(F).
$
\item{} For a sequence $F$ in a graded algebra $S = k \langle x_1, \dots, x_n \rangle /I$,
$
\sH^z(S,F)/(z) \cong \sT(S,F). 
$
\end{enumerate}
\end{lemma}

\begin{proof}
For (1). Write $F = (f_1,\dots,f_m)$, then
\begin{align*}
\sH^{\dagger,z}(F)/(z) =& k \langle x_1, \dots, x_n \rangle [z] /(f_1^z, \dots, f_m^z,z) \\
= & k \langle x_1, \dots, x_n \rangle [z] /(f_1^\vee, \dots, f_m^\vee,z) \\
\cong & k \langle x_1, \dots, x_n \rangle /(f_1^\vee, \dots, f_m^\vee) = \sT^\dagger(F). 
\end{align*}
The proof of (2) is similar. 
\end{proof}

For a sequence $F$, checking if $F^\vee$ is regular is easier than checking if $F^z$ is regular in practice. 
Thus the following lemma is useful. 

\begin{lemma} \label{lem.612} 
Let $F$ be a sequence in a graded algebra $S = k \langle x_1, \dots, x_n \rangle /I$.
If $F^\vee$ is normal in $S$ and $F^z$ is normal in $S[z]$, then the following are equivalent: 
\begin{enumerate}
\item{} $F^\vee$ is regular in $S$. 
\item{} $(F^\vee,z)$ is regular in $S[z]$. 
\item{} $(F^z,z)$ is regular in $S[z]$.
\end{enumerate}
Moreover, if one of the above holds, then $F^z$ is regular in $S[z]$. 
\end{lemma} 

\begin{proof}
$ (1) \Leftrightarrow  (2)$ by Lemma \ref{lem-suci}.

$ (1) \Leftrightarrow  (3)$. By Lemma \ref{lem-tchz} (2), $S/I_{F^\vee} \cong S[z]/I_{(F^z,z)}$, so $F^\vee$ is regular in $S$ if and only if $(F^z,z)$ is regular in $S[z]$ by Lemma \ref{lem.Tak}. 

The moreover part is clear. \end{proof}

In general, the converse of the last claim in Lemma \ref{lem.612} does not hold. 

\begin{example}
Let $F = (x^2-y, x^2+y)$ be a sequence in $k[x,y]$. Then $F^z = (x^2-yz, x^2+yz)$ is regular in $k[x,y,z]$, but $F^\vee = (x^2, x^2)$ is not regular in $k[x,y]$.
\end{example}

 \subsection{$st$-equivalences}
 
It is rather inconvenient that the (wild) homogenization $\sH^z(S,F)$ ($\sT(S, F)$) depends on the choice of $F$ even in the commutative case. 

\begin{example} 
Let $F = (x^2-y, x^2+y)$, and $F' = (x^2, y)$ be sequences in $k[x,y]$.  
Then $k[x, y]/I_F=k[x, y]/I_{F'}$, but 
$\sH^z(F) = k[x, y, z]/(x^2-yz, x^2+yz)\not \cong k[x, y, z]/(x^2, y) = \sH^z(F')$ 
and $\sT(F) = k[x, y, z]/(x^2)\not \cong k[x, y, z]/(x^2, y) = \sT(F')$. 
\end{example} 

\begin{example} There is a (counter-)example even in the case that $\deg F=\deg F'=2$. 
For example, let $F = (xy-y, xy+y)$, and $F' = (y^2-y, y^2+y)$ be sequences in $k[x,y]$.  
Then $k[x, y]/I_F=k[x, y]/(y)=k[x, y]/I_{F'}$, but 
$
\sH^z(F) = k[x, y, z]/(xy, yz)\not \cong k[x, y, z]/(y^2, yz) = \sH^z(F')
$ and $
\sT(F) = k[x, y, z]/(xy)\not \cong k[x, y, z]/(y^2) = \sT(F')
$.   
\end{example}  

We now introduce some conditions so that $\sH^z(S,F)$ is independent on the choice of $F$. 

\begin{definition} [{\cite[Definition 3.1]{HMTW}}]
Let $S$ be a graded algebra, and $F = (f_1, \dots, f_m), F' = (f_1', \dots, f_m') $ sequences in $S$.
\begin{enumerate}
\item{} We write $F \sim _s F'$ if $f'_j=\sum _{i=1}^m\a_{ij}f_i$ 
for some $(\a_{ij})\in \GL_m(k)$. 
\item{} We write $F \sim _t F'$ if $f'_i=\phi(f_i)$ for some $\phi\in \Aut^{\ZZ}S$.   
\item{} We write $F \sim _{st} F'$ if there exists a sequence $F''$ in $S$ 
such that  $F \sim_t F'' \sim _s F'.$ 
\end{enumerate}
\end{definition} 

\begin{remark}  Let $S$ be a graded algebra, and $F, F'$ sequences in $S$.  If $F\sim_{st}F'$, then $S/I_F\cong S/I_{F'}$, but the converse does not hold in general.  
However, let $F, F'$ be (linearly independent) homogeneous sequences of degree $d$ in $k[x_1, \dots, x_n]$ (or in $k\<x_1, \dots, x_n\>$).  Then $F\sim _{st}F'$
if and only if  
$k[x_1, \dots, x_n]/I_F\cong k[x_1, \dots, x_n]/I_{F'}$ (or $k\<x_1, \dots, x_n\>/I_F\cong k\<x_1, \dots, x_n\>/I_{F'}$) as graded algebras (see {\cite[Remark 3.3]{HMTW}}). 
\end{remark} 


\begin{lemma}[{\cite[Lemma 3.4]{HMTW}}]
For every fixed $m\in\mathbb{N}$, $\sim_s$, $\sim_t$ and $\sim_{st}$ are equivalence relations on the set of sequences 
$(f_1,\cdots, f_m)$ in $S$.
\end{lemma}

\begin{lemma} \label{lem.st} 
Let $F = (f_1, \dots, f_m)$ and $F' = (f_1', \dots, f_m') $ be sequences of degree $d$ in a graded algebra $S = k \langle x_1, \dots, x_n \rangle/I$.
\begin{enumerate} 
\item{} If $F \sim _s F'$, then $\sH^z(F) = \sH^z(F')$.  
\item{} If $F \sim _t F'$, then $\sH^z(F) \cong \sH^z(F')$ as graded algebras. 
\item{} If $F \sim _{st} F'$,  then $\sH^z(F) \cong \sH^z(F')$ as graded algebras. 
\end{enumerate}
\end{lemma} 

\begin{proof} 
(1) Let $f'_j=\sum _{i=1}^m\a_{ij}f_i$ 
for $(\a_{ij})\in \GL_m(k)$.  Since 
$\deg f_i=\deg f'_j =d$, we have $(f'_j)^z=\sum _{i=1}^m\a_{ij}(f_i)^z$. 

(2) Let $\phi \in \Aut^{\ZZ} S$ such that $f'_i=\phi(f_i)$. Then there is $(\a_{ij})\in \GL_n(k)=\Aut^{\ZZ}(k\<x_1, \dots, x_n\>)$ such that $\phi(x_i)=\sum_{j=1}^n\a_{ij}x_j$. Let $\tilde{\phi} \in \Aut^{\ZZ} S[z]$ be defined by $\tilde \phi(x_i) = \phi(x_i)$ and $\tilde\phi(z)=z$. 
For $f \in S$, 
\begin{align*}
\phi(f)^z & =f\left (\sum_{j=1}^n\a_{1j}(x_jz^{-1}), \dots, \sum_{j=1}^n\a_{nj}(x_jz^{-1})\right )z^{\deg \phi(f)} \\
& =f\left (\left (\sum_{j=1}^n\a_{1j}x_j\right )z^{-1}, \dots, \left (\sum_{j=1}^n\a_{nj}x_j\right )z^{-1}\right )z^{\deg f} = \tilde \phi (f^z).
\end{align*}
It implies that $\tilde{\phi}: \sH^z(F) \to \sH^z(F')$ is an isomorphism.  

(3) This follows from (1) and (2).  
\end{proof}

\begin{remark} A similar results hold for a wild homogenization.  Let $F, F'$ be sequences of degree $d$ in a graded algebra $S$.  If $F\sim _{st}F'$, then $F^{\vee}\sim_{st}{F'}^{\vee}$ in $S$, so $\sT(F)\cong \sT(F')$. 
\end{remark}

The converse of Lemma \ref{lem.st} (3) does not hold even for commutative algebras (see \cite[Example 5.12]{HMTW}) .  

\subsection{Strongly regular normal sequences} \label{subsec-srns}

\begin{definition} 
Let $F = (f_1, \dots, f_m)$ be a sequence in a graded algebra $S$.
\begin{enumerate}
\item{} $F$ is called {\it strongly regular} if both $F$ and $F^\vee$ are regular. 
\item{} $F$ is called {\it strongly normal (resp. strongly central)} if both $F$ and $F^\vee$ are normal (resp. central). 
\item{} $F$ is called {\it strongly regular normal (resp. strongly regular central)} if $F$ is strongly regular and strongly normal (resp. strongly central).
\end{enumerate}
\end{definition} 

\begin{remark} \label{rem.92} 
(1) For $f, g\in k[x_1, \dots, x_n]$, if $(f^\vee, g^\vee)$ is a regular sequence, then $\gcd (f^\vee, g^\vee)=1$, so $\gcd(f, g)=1$, hence $(f, g)$ is a (strongly) regular sequence by Remark \ref{rem.52}.  

(2) If $S$ is not commutative and $f\in S$, then  
it is possible that $f^\vee$ is normal but $f$ is not normal.
\end{remark} 

We denote by 
\begin{gather*}
{\cA}_{n, m}^\vee:=\{S/I_F\mid F  = (f_1, \dots, f_m) \textnormal{ is a strongly regular normal sequence of degree $2$ in } S \in \cA_{n,0}\}/\cong,  \\
{\cB}_{n, m}^\vee:=\{S/I_F\mid F = (f_1, \dots, f_m \textnormal{ is a strongly regular sequence of degree $2$ in } S \in \cB_{n,0}\}/\cong.
\end{gather*}
Following \cite{HMTW}, we say that $R$ is {\it (the homogeneous coordinate ring of) a noncommutative affine pencil of conics} if $R \in \cA_{2, 2}^\vee$, and $B$ is {\it (the homogeneous coordinate ring of) an affine pencil of conics} if $B \in \cB_{2, 2}^\vee$.

In the sequel, it is important to know if the normalizing automorphism $\nu\in \Aut S$ of a regular normal element $f\in S$ is graded, that is, if $\nu\in \Aut^{\ZZ}S$.  If $f$ is homogeneous or central, then $\nu\in \Aut^{\ZZ}S$.    

\begin{lemma} \label{lem.gna} 
Let $S$ be a graded algebra, and $f\in S$ a regular normal element such that $f^\vee \in S$ is also a regular normal element.  Then the normalizing automorphism of $f$ is graded if and only if  it coincides with the normalizing automorphism of $f^\vee$. 
\end{lemma} 

\begin{proof} Let $\nu\in \Aut S$ be the normalizing automorphism of $f$ and  $\nu^\vee \in \Aut^{\ZZ} S$ the normalizing automorphism of $f^\vee$.  If $\nu\in \Aut^{\ZZ}(S)$, then $\nu (g)\in S_i$ for every $g\in S_i$, so  
$$
f^\vee\nu^\vee(g)=gf_d = (gf)_{d+i} = (f\nu(g))_{d+i} = f^\vee \nu(g),
$$
where $d = \deg f$.
Since $f^\vee$ is regular, $\nu^\vee(g)=\nu(g)$.  
\end{proof} 

\begin{lemma} \label{lem.afbg} 
Let $S$ be a connected graded domain, and $F = (f, g)$ a strongly regular normal sequence of degree $d$ such that the normalizing automorphism of $f$ is graded.  If $h\in I_F$ such that $\deg h=d$, then $h=\a f+\b g$ for some $\a, \b\in k$. 
\end{lemma} 

\begin{proof} Since $F$ is a normal sequence, $I_F=Sf+Sg$, so  $h=af+bg$ for some $a, b\in S$.  Take $a\in S$ of minimal degree such that $h=af+bg$ for some $b\in S$.  Suppose that $\deg a\geq 1$.  Since $S$ is a domain and $bg=af-h$, $\deg b=\deg a=:e$.  Write $f=f_d+f', g=g_d+g', a=a_e+a', b=b_e+b'$ so that $\deg f', \deg g'<d$ and $\deg a', \deg b'<e$.  
Since 
$$h=(a_e+a')(f_d+f')+(b_e+b')(g_d+g')=(a_ef_d+b_eg_d)+a_ef'+b_eg'+a'f+b'g,$$
$a_ef_d+b_eg_d=0$.  Since $F^\vee = (f_d, g_d)$ is a regular sequence and $(f_d)=Sf_d$, $b_e=cf_d$ for some $c\in S$.  Since $a_ef_d+cf_dg_d=(a_e+c\nu^{-1}(g_d))f_d=0$ and $f_d$ is regular, $a_e=-c\nu^{-1}(g_d)$ where $\nu \in \Aut^{\ZZ} S$ is the normalizing automorphism of $f_d$.  Since $\nu$ is also the normalizing automorphism of $f$ by Lemma \ref{lem.gna}, 
\begin{align*}
(c\nu^{-1}(g')+a')f+(-cf'+b')g 
& = c\nu^{-1}(g')f-cf'g+a'f+b'g \\
& = cfg'-cf'g+a'f+b'g \\
& = c(f_d+f')g'-cf'(g_d+g')+a'f+b'g  \\
& = cf_dg'-c(f-f_d)g_d+a'f+b'g \\
& = -c\nu^{-1}(g_d)(f-f_d)+cf_dg'+a'f+b'g \\
& = a_ef'+b_eg'+a'f+b'g = h.  
\end{align*}
Since $\deg c=e-d$,  $\deg (c\nu(g')+a')<e=\deg a$, so this contradicts the assumption that $\deg a$ is minimal.  It follows that $\deg a, \deg b\leq 0$, so $a, b\in k$.    
\end{proof} 

\begin{proposition} \label{prop.fgfg} 
Let $(f, g)$ and $(f', g')$ be strongly regular normal sequences of degree $d$ in a connected graded domain $S$ such that the normalizing automorphisms of $f$ and $f'$ are both graded. 
Then $S/(f, g)=S/(f', g')$ if and only if  $(f, g)\sim_s (f', g')$. 
\end{proposition}  

\begin{proof} If $S/(f, g)=S/(f', g')$, then $f'=\a f+\b g, g'=\c f+\d g$ and $f=\a' f'+\b' g', g=\c' f+\d' g'$ 
for some 
$\a, \b, \c, \d, \a', \b', \c', \d'\in k$ by Lemma \ref{lem.afbg}.  
Since $f, g$ are linearly independent, $\begin{pmatrix} \a & \b \\ \c & \d \end{pmatrix}^{-1} =\begin{pmatrix} \a' & \b' \\ \c' & \d' \end{pmatrix}$.  
\end{proof} 

\begin{corollary} \label{cor.SRN} 
Let $F = (f, g)$ and $F' = (f', g')$ be strongly regular sequences of degree $d$ in $R = k[x_1, \dots, x_n]$. 
If 
$
R/(f, g)=R/(f', g'),
$
then $\sH^z(F) =\sH^z(F')$. 
\end{corollary}  

\begin{proof} By Proposition \ref{prop.fgfg}, $F \sim_s F'$, 
so $\sH^z(F) = \sH^z(F')$ by Lemma \ref{lem.st} (1). 
\end{proof} 

\begin{lemma} \label{lem.bcn} 
If $B\in \cB_{n+1, m}$ and $z\in B_1$ is regular, then $\sD_z(B) \in \cB_{n, m}^\vee$.  
\end{lemma} 

\begin{proof}  If $H = (h_1, \dots , h_m)$ is a homogeneous regular sequence of degree 2 in $k[x_1, \dots, x_n, z]$ such that 
$$
B = k[x_1, \dots, x_n, z]/I_H\in \cB_{n+1, m},
$$
then the map
$$
 \sD_z(B) \to k[x_1, \dots, x_n]/I_{H_z}; \; fz^{-\deg f}\mapsto f_z
$$ 
is an isomorphism by Lemma \ref{lem.zin} and Lemma \ref{lemq.06}.  
Since $z\in B_1$ is regular,
$H_z$ is a regular sequence of degree $2$ in $k[x_1, \dots, x_n]$ by Lemma \ref{lem.zfr}. 
Since 
$\deg (h_i)_z=\deg h_i=2$ for every $i=1, \dots, m$, $(H_z)^z=H$ is a homogeneous regular sequence of degree $2$ in $k[x_1, \dots, x_n, z]$ by Lemma \ref{lemq.05} (2). 
Since $((H_z)^z, z)=(H, z)$ is a regular sequence, $(H_z)^\vee$ is a regular sequence of degree $2$ in $k[x_1, \dots, x_n]$ by  Lemma \ref{lem.612}, so
\begin{equation*}
\sD_z(B) \cong k[x_1, \dots, x_n]/I_{H_z}\in \cB_{n, m}^\vee. \qedhere
\end{equation*}
\end{proof}

\section{Noncommutative quadric hypersurfaces}

\subsection{Duality} 

Let $V$ be a finite dimensional vector space. A {\it quadratic algebra} $A = T(V)/(W)$ is a quotient algebra of the tensor algebra $T(V)$ where $W \subset V \otimes V$ is a subspace. The {\it quadratic dual} of $A$ is the quadratic algebra $A^!: = T(V^*)/(W^\perp)$ where $V^*$ is the dual vector space and $W^\perp \subset V^* \otimes V^*$ is the orthogonal complement of $W$.  

A locally finite connected graded algebra $A$ is called a {\it Koszul algebra} if the trivial module $k_A$ has a free resolution
$$
\xymatrix{
\cdots \ar[r] & P^n \ar[r] & \cdots \ar[r] & P^1 \ar[r] & P^0 \ar[r] & k_A \ar[r] & 0
}
$$
where $P^n$ is a graded free right $A$-module generated in degree $n$ for each $n \geq 0$. 

\begin{lemma} [{\cite[Theorem 1.2]{ST}}] \label{lem.kos} 
Let $A$ be a quadratic algebra. 
\begin{enumerate}
\item{} $A$ is Koszul if and only if $A^!$ is Koszul.   
\item{} If $A$ is Koszul, then $H_{A^!}(t)=1/H_A(-t)$.  
\item{} If $A$ is Koszul and $f\in A_2$ is a regular normal element, then $A/(f)$ is also Koszul.  
\end{enumerate} 
\end{lemma} 

\begin{lemma}[{\cite[Proposition 5.10]{Shi96}}]  \label{lem.smith} 
Let $A$ be a noetherian connected graded algebra. Then $A \in \cA_{n,0}$ if and only if $A^!$ is a Frobenius Koszul algebra such that $H_{A^!}(t)=(1+t)^n$.
\end{lemma}

\begin{lemma}[{\cite[Corollary 1.4]{ST}, \cite[Section 5]{HY}}] \label{lem.asf} 
Let $S\in \cA_{n, 0}$.   For every $f\in RN(S)_2$ so that $A=S/(f)\in \cA_{n, 1}$, there exists a unique (up to scalar) 
element $f^!\in RN(A^!)_2$ such that $S^!=A^!/(f^!)$.  Moreover, $f\in RZ(S)_2$ if and only if $f^!\in RZ(A^!)_2$.   
\end{lemma} 

\begin{lemma} \label{lem.Kono} 
If $A$ is a Koszul algebra and $F = (f_1, \dots, f_m)$ is a homogeneous regular normal sequence of degree $2$, then $A/I_F$ is a Koszul algebra, and there exists a homogeneous regular normal sequence $F^! := (f_m^!, \dots, f_1^!)$ of degree $2$ in $(A/I_F)^!$ 
such that $A^!=(A/I_F)^!/I_{F^!}$. 
\end{lemma} 

\begin{proof} The result holds for $m=1$ by Lemma \ref{lem.kos} (2) and Lemma \ref{lem.asf}, and the general case follows by induction.
\end{proof} 

Let $A$ be a locally finite $\NN$-graded algebra. 
If $H_A(t)=b(t)/(1-t)^d$ for some $b(t)\in \ZZ[t]$ such that $b(1)\neq 0$, then $\GKdim A=d$ by \cite[Corollary 2.2]{SZ}.

\begin{lemma} \label{lem.KAS}
Let $A=S/(h_1, \dots, h_m)\in \cA_{n, m}$.
\begin{enumerate}
\item{} $A$ is a Koszul AS-Gorenstein  algebra of dimension $n-m$ and of Gorenstein parameter $n-2m$ such that $\GKdim A=n-m$.  In particular, every $A\in \cA_{n, n}$ is a Frobenius Koszul algebra.  
\item{} $A^!$ is a  Koszul AS-Gorenstein algebra of dimension $m$ and of Gorenstein parameter $2m-n$ such that $\GKdim A^!=m$.
\end{enumerate} 
\end{lemma}
 
\begin{proof} By Lemma \ref{lem.kos}, $A, A^!$ are Koszul.   
By Lemma \ref{lem.Tak}, 
$$
H_A(t)=(1-t^2)^m/(1-t)^n=(1+t)^m/(1-t)^{n-m}, 
$$
so $\GKdim A=n-m$.  By Lemma \ref{lem.kos} (3), 
$H_{A^!}(t)=1/H_A(-t)=(1+t)^{n-m}/(1-t)^{m}$, so $\GKdim A^!=m$.

(1) Since $S$ is an AS-Gorenstein algebra of dimension $n$ and of Gorenstein parameter $n$, $A$ is an AS-Gorenstein algebra of dimension $n-m$ and of Gorenstein parameter $n-2m$ by Lemma \ref{lem.Ree}.   
The last statement follows from Lemma \ref{lem.FAS}.  

(2)   Since $S^!$ is an AS-Gorenstein algebra of dimension 0 and of Gorenstein parameter $-n$, and $S^!=A^!/(h_m^!, \dots, h_1^!)$, where 
$(h_m^!, \dots, h_1^!)$ is a homogeneous regular normal sequence of degree $2$ in $A^!$  by Lemma \ref{lem.Kono}, the result follows by Lemma \ref{lem.Ree} (1) again. 
\end{proof}

\begin{lemma} \label{lem.48} 
Let $A\in \cA_{n, m}$.   If $A^!$ is commutative, then $A^!\in \cB_{n, n-m}$. 
\end{lemma} 

\begin{proof} By the proof of Lemma \ref{lem.KAS} and Lemma \ref{lem.kos}, 
$$
H_{A^!}(t)=(1+t)^{n-m}/(1-t)^{m}=(1-t^2)^{n-m}/(1-t)^n.
$$  
If $A^!$ is commutative, then $A^!=k[x_1, \dots, x_n]/(h_1, \dots , h_{n-m})$ for some homogeneous regular sequence $(h_1, \dots, h_{n-m})$ of degree $2$ in $k[x_1, \dots, x_n]$ by Lemma \ref{lem.Tak},
so $A^!\in \cB_{n, n-m}$. 
\end{proof}  

\begin{remark}  Let $A\in \cA_{n, m}$.   It follows from the above lemma that $A^!$ is commutative if and only if $A\cong S/I_H$ where $S$ is a Clifford quantum polynomial algebra and $H$ is a homogeneous regular central sequence of degree 2 in $S$ by \cite{HM}.
\end{remark} 

\begin{lemma} \label{lem.du} 
Let $A = S/(h_1, \dots, h_m)\in \cA_{n, m}$.  Then $A^!\in \cA_{n, n-m}$ if and only if we can extend the sequence $(h_1, \dots, h_m)$ to a homogeneous regular normal sequence $(h_1, \dots,  h_m, h_{m+1}, \dots, h_n)$ of degree $2$ in $S$.  
In particular, the map
$
\cA_{n, n}\to \cA_{n, 0}; \; A\mapsto A^!
$
is well-defined.   
\end{lemma} 

\begin{proof} 
If $A^! = S'/(g_1, \dots, g_{n-m})\in \cA_{n, n-m}$, where $S' \in \cA_{n, 0}$ and $(g_1, \dots, g_{n-m})$ is a homogeneous regular normal sequence of degree $2$ in $S'$. 
Then $(g_{n-m}^!, \dots, g_{1}^!)$ is a homogeneous regular normal sequence of degree $2$ in $A$, 
so $(h_1, \dots, h_m, g_{n-m}^!, \dots, g_{1}^!)$ is a homogeneous regular normal sequence of degree $2$ in $S$    (such that $(S')^!=A/(g_{n-m}^!, \dots, g_{1}^!)\cong S/(h_1, \dots, h_m, g_{n-m}^!, \dots, g_{1}^!)$) by Lemma \ref{lem.Kono}. 

Conversely, if 
$(h_1, \dots,  h_m, h_{m+1}, \dots, h_n) $ is homogeneous regular normal sequence of degree $2$ in $S$,  
then 
$$E:=A/(h_{m+1}, \dots, h_n)=S/(h_1, \dots, h_m, h_{m+1}, \dots , h_n)\in \cA_{n, n}$$ 
is a Frobenius Koszul algebra by Lemma \ref{lem.KAS} (1).  Since 
$$
A^! / (h_m^!, \dots, h_1^!)=E^!/(h_n^!, \dots, h_{m+1}^!, h_m^!, \dots, h_1^!)
$$ 
is noetherian (artinian), $E^!$ is noetherian by Lemma \ref{lem.Ree} (1),  so $E^!\in \cA_{n, 0}$ by Lemma \ref{lem.smith},
and $A^!=E^!/(h_n^!, \dots, h_{m+1}^!)\in \cA_{n, n-m}$ by Lemma \ref{lem.Kono}.     
\end{proof}

\begin{lemma} Let $A\in \cA_{n+1, n}$.  If there exists $w\in RN(A)_1$, then $A^!\!\in \cA_{n+1, 1}$. 
\end{lemma} 

\begin{proof}  
If $A=S/(h_1, \dots, h_n)\in \cA_{n+1, n}$ and $w\in RN(A)_1$, then 
$(h_1, \dots, h_n, w^2)$ is a homogeneous regular normal sequence of degree $2$ in $S$, 
so $A^!\in \cA_{n+1, 1}$ by Lemma \ref{lem.du}.
\end{proof}

\begin{proposition} \label{prop.du}
For $m=0, 1, 2$, the map 
$
\cA_{2, m}\to \cA_{2, 2-m}; \; A\mapsto A^!
$
is a duality.  In particular, every $A\in \cA_{2, 2}$ is isomorphic to one of the following
$$
k_{\l}[x, y]/(x^2, y^2),\ k_{-1}[x, y]/(xy+x^2, y^2),
$$
where $0 \neq \l \in k$.   
\end{proposition}

\begin{proof} If $A\in \cA_{2, 0}$, we may assume $A=k_{\l}[x, y]$ where $0 \neq \l \in k$ or $A=k\<x, y\>/(xy-yx-x^2)$ (see Example \ref{exm-2dimqpa}).  In either case, $(x^2, y^2)$ is a homogeneous regular normal sequence of degree 2 in $A$, so $A^!\in \cA_{2, 2}$ by Lemma \ref{lem.du}.  It follows that $\cA_{2, 0}\to \cA_{2, 2}; \; A\to A^!$ is well-defined, so it is a duality by Lemma \ref{lem.du}.  

Every $A\in \cA_{2, 1}$ is isomorphic to 
$$k[x, y]/(x^2+y^2), \; k_{\l}[x, y]/(x^2), 
\; k_{-1}[x, y]/(x^2+y^2)$$
by \cite[Proposition 3.17]{HMTW} 
and their duals are 
$$k_{-1}[x, y]/(x^2-y^2), \; k_{-\l^{-1}}[x, y]/(y^2), 
\; k[x, y]/(x^2-y^2)\in \cA_{2, 1},$$  so $\cA_{2, 1}\to \cA_{2, 1}; \; A \mapsto A^!$ is a duality.  
\end{proof}

\begin{remark} It is unclear if $\cA_{n, 0}\to \cA_{n, n}; \; S\mapsto S^!$ is well-defined so that $\cA_{n, 0}\to \cA_{n, n}$ is a duality.  There are many $S\in \cA_{3, 0}$ such that $Z(S)_2=\{0\}$ (For example, the Type T$_3$ algebra in Table \ref{tab-cenele}),   so there may exist $S\in \cA_{3, 0}$ such that there is no normal element of degree 2.  If so, then $S^!\not \in \cA_{3, 3}$ by Lemma \ref{lem.du}, so $\cA_{3, 0}\to \cA_{3, 3}; \; S \mapsto S^!$ is not well-defined (not a duality). 
\end{remark} 

\begin{example} It is slightly more hopeful to expect that $\cA_{3, 1}\to \cA_{3, 2}; \; A \mapsto A^!$ is a duality.  If 
$$
S=k\<x, y, z\>/(yz-zy-x^2, zx-xz, xy-yx)\in \cA_{3, 0},$$ 
then $x^2\in S_2$ is normal (in fact, it is the unique homogeneous normal element of degree 2 up to scalar) so that $A:=S/(x^2)\in \cA_{3,1}$.  Although $(y^2, z^2)$ is not a normal sequence in $S$, $(x^2, y^2, z^2)$ is a regular normal sequence in $S$ , 
so $S^!\cong k_{-1}[x, y, z]/(yz+x^2, y^2, z^2)\in \cA_{3, 3}$ and $A^!=k_{-1}[x, y, z]/(y^2, z^2)\in \cA_{3, 2}$. 
\end{example}

\subsection{Twisted algebras} \label{subsec-twalg}

\begin{definition} 
[{\rm Zhang twist \cite{Z}}]
Let $A$ be a $\ZZ$-graded algebra.  A sequence $\theta=\{\theta_i\}_{i\in \ZZ}$ of graded linear automorphisms of $A$ is called a twisting system of $A$ if $\theta_i(a\theta_j(b))=\theta_i(a)\theta_{i+j}(b)$ for every $a\in A_j, b\in A$ and every $i, j\in \ZZ$.  
For a twisting system $\theta$ of $A$, 
we define a new graded algebra $A^{\theta}=A$ as a graded vector space with a new multiplication $a*b=a\theta_j(b)$ for $a\in A_j, b\in A$.  
\end{definition} 

For a graded algebra $A$ and $\s\in \Aut^{\ZZ}(A)$, $\{\s^i\}_{i\in \ZZ}$ is a twisting system of $A$, so we can define $A^{\s}:=A^{\{\s^i\}}$. 

\begin{theorem} 
[{\cite[Theorem 1.1]{Z}}] \label{thm.Z11} 
Let $A$ be a graded algebra.  For every twisting system $\theta$ of $A$,  
$\GrMod A^{\theta}\cong \GrMod A$. 
\end{theorem}

\begin{lemma} \label{lem.62} 
Let $A$
be a graded algebra, $w\in RN(A)_1$ and $\nu\in \Aut ^{\ZZ}A$ the normalizing automorphism of $w$.
For every $\s\in \Aut^{\ZZ}(A)$ such that $\s(w)=w$ (so that $\s$ induces $\s\in \Aut^{\ZZ}(A/(w)$), $w\in RN(A^{\s})_1$. 
In particular, $w\in RZ(A^{\nu})_1$. 
\end{lemma} 
 
\begin{proof} 
Let $f\in A_d$. Then $f*w=f\s^d(w)=fw$ and $w*f=w\s(f)$ implies that $w\in A^{\s}_1$ is regular.  On the other hand, $f*w=f\s^d(w)=fw=w\nu (f)=w*\s^{-1}\nu(f)$, so $w\in A^{\s}_1$ is a regular normal element with the normalizing automorphism $\s^{-1}\nu$.  
It follows that $w\in A^{\s}_1$ is central if and only if $\s=\nu$. 
\end{proof}

\begin{lemma} \label{lem.sf} 
Let $A$ be a graded algebra and $f\in A$ a homogeneous normal element.  If $\s\in \Aut^{\ZZ}A$ such that $\s(f)=\l f$ for some $0\neq \l\in k$, then $\bar \s\in \Aut^{\ZZ}(A/(f))$ such that $(A/(f))^{\bar \s}\cong A^{\s}/(f)$. 
\end{lemma} 

\begin{proof} 
Let $\phi: A^{\s} \to (A/(f))^{\bar\s}, a\mapsto \overline{a}$, then $\phi(a*b) = \overline{a\s^{i}b}= \overline{a}\, \overline{\s^i(b)} =  \overline{a}* \overline{b} =  \phi(a) * \phi(b)$ for every $a\in A_i,b\in A_j$, so $\phi$ induces an isomorphism $A^{\s}/(f)\to (A/(f))^{\bar \s}$.
\end{proof} 

\begin{lemma} \label{lem.ndu}
For a quadratic algebra $A$, there exists an isomorphism $\Aut^{\ZZ}A\to \Aut^{\ZZ}A^!; \; \s\mapsto \s^!$ of groups 
such that $(A^!)^{\s^!} \cong (A^{\s})^!$.
\end{lemma} 

\begin{proof}   For $\s\in \GL(V)$, $((\s^*\otimes \s^*)(\xi))(f)=\xi((\s\otimes \s)(f))$ for $f\in V\otimes V$ and $\xi\in V^*\otimes V^*$, so $\s\in \Aut ^{\ZZ}A$ if and only if $(\s\otimes \s)(W)=W$ if and only if  $(\s^*\otimes \s^*)(W^{\perp})=W^{\perp}$ if and only if $\s^*\in \Aut ^{\ZZ}A^!$, hence $\Aut^{\ZZ}A\to \Aut^{\ZZ}A^!; \; \s\mapsto \s^!:=(\s^*)^{-1}$ is an isomorphism of groups.  

Since $\xi\in 
((\s\otimes \id)(W))^{\perp}$ if and only if 
$((\s^*\otimes \id)\xi)(f)=\xi((\s\otimes \id)(f))=0$ for every $f\in W$ if and only if $(\s^*\otimes \id)\xi\in W^{\perp}$, $((\s\otimes \id)(W))^{\perp}=((\s^*)^{-1}\otimes \id)(W^{\perp})$, so 
\begin{align*}
(A^!)^{(\s^*)^{-1}} \cong & T(V^*)/(((\s^*)^{-1}\otimes \id)(W^{\perp})) \\
= &T(V^*)/(((\s\otimes \id)(W))^{\perp})\\
= &(T(V)/((\s\otimes \id)(W)))^! \cong (A^{\s})^!. \qedhere
\end{align*}
\end{proof} 

If $A$ is a graded algebra, $w \in RN(A)_1$ and $\s\in \Aut^{\ZZ}A$ such that $\s(w)=\l w$ for some $0\neq \l\in k$, then $\sD_w(A^{\s})$ is well-defined since $w \in RN(A^{\s})_1$ by Lemma \ref{lem.62}. 

\begin{lemma}  \label{lem.633} 
Let $A$ be a graded algebra and $w \in RN(A)_1$ with the normalizing automorphism $\nu\in \Aut ^{\ZZ}A$.  For every $\s\in \Aut^{\ZZ}A$ such that $\s(w)=w$ and $ \s\nu = \nu\s$, $\sD_w(A) \cong \sD_w(A^{\s})$ as algebras.  In particular, $\sD_w(A)\cong A^{\nu}/(w-1)$.
\end{lemma} 

\begin{proof} 
Define a map 
$$
\varphi: \sD_w(A) = A[w^{-1}]_0 \to \sD_w(A^\s) = A^\s[w^{-1}]_0; \;  fw^{-i} \mapsto f w^{-i} 
$$
for $f\in A_i=A^{\s}_i$.  Since the normalizing automorphism of $w$ in $A^\s$ is $\s^{-1} \nu$ by the proof in Lemma \ref{lem.62}, 
\begin{align*}
\varphi(fw^{-i})\cdot \varphi(gw^{-j}) & = (fw^{-i})* gw^{-j}=f*(\s^{-1}\nu)^i(g)w^{-i-j} =f\s^i(\s^{-1}\nu)^i(g)w^{-i-j}  \\
& =f\nu^i(g)w^{-i-j}=\varphi(f\nu^i(g)w^{-i-j})=\varphi(fw^{-i}\cdot gw^{-j})
\end{align*}
for $f\in A_i=A^{\s}_i, g\in A_j=A^{\s}_j$, so we can check that $\varphi$ is an isomorphism of algebras. 
By Lemma \ref{lem.62},  $w\in RZ(A^{\nu})_1$, so $\sD_w(A)\cong \sD_w(A^{\nu})\cong A^{\nu}/(w-1)$ by Lemma \ref{lem.zin}.
\end{proof}

Due to the above Lemma, we will identify $\sD_w(A)$ with $A^\nu/(w-1)$, which we already mentioned in Remark \ref{rem-idnor}.

The following theorem is a key result to classify noncommutative central conics. 

\begin{theorem} \label{thm.awz} For every $A\in \cA_{n, 1}$ with 
$w\in RN(A^!)_1$, there exists $A'\in \cA^c_{n, 1}$ with 
$z\in RZ((A')^!)_1$ such that $\GrMod A'\cong \GrMod A$. 
\end{theorem}

\begin{proof} 
Let $\nu\in \Aut ^{\ZZ}A^!$ be the normalizing automorphism of $w\in RN(A^!)_1$. 
Since $A^!$ is a Koszul AS-Gorenstein algebra of dimension 1 by Lemma \ref{lem.KAS}, 
$E:=A^!/(w^2)$ is a Koszul AS-Gorenstein algebra of dimension 0 by Lemma \ref{lem.Ree}. 
Since  $E^!/(f)  = A$ is noetherian where $f:=(w^2)^!\in RN(E^!)_2$ 
by Lemma \ref{lem.asf}, $E^!$ is noetherian by Lemma \ref{lem.Ree} (1).  Since $E$ is a Frobenius Koszul algebra by Lemma \ref{lem.FAS} and 
$$H_{E^!}(t)=1/H_E(-t)=1/(H_{A^!}(-t)(1-t^2))=H_A(t)/(1-t^2)=1/(1-t)^n,$$ 
$E^!\in \cA_{n, 0}$ by Lemma \ref{lem.smith}.  Since $\nu(w^2)=w^2$, $\nu\in \Aut^{\ZZ}A^!\cap \Aut^{\ZZ}E$, 
so $\nu^!\in \Aut^{\ZZ}A\cap \Aut^{\ZZ}E^!$ by Lemma \ref{lem.ndu} so that $\nu^!(f)=\l f$ for some $0\neq \l\in k$.  
On the other hand, $(A^{\nu^!})^! \cong (A^!)^{\nu}$ and 
$$
A^{\nu^!} = (E^!/(f))^{\nu^!} \cong (E^!)^{\nu^!}/(f) \cong (E^\nu)^! / (f) = ((A^!/(w^2))^\nu)^!/(f) \cong ((A^!)^\nu / (w^2))^! /(f)
$$
by Lemma \ref{lem.sf} and Lemma \ref{lem.ndu}.
Since $w\in RZ((A^!)^{\nu})_1$ by Lemma \ref{lem.62}, $f\in RZ((E^!)^{\nu^!})_2$ by Lemma \ref{lem.asf}.

Since $(E^!)^{\nu^!}\in \cA_{n, 0}$ by \cite[Theorem 1.3]{Z},  $A^{\nu^!} \in \cA^c_{n, 1}$.
By Theorem \ref{thm.Z11}, $\GrMod A^{\nu^!}\cong \GrMod A$.  
\end{proof}   

\begin{remark} 
 A key point of the above proof is to show that $A^{\nu^!}\in \cA_{n, 1}$. In general, $A\in \cA_{n, 1}$ and $\s\in \Aut^{\ZZ} A$ does not imply $A^{\s}\in \mathcal{A}_{n,1}$ unless $\s\in \Aut^{\ZZ}S$ as the following example shows.
\end{remark}   

\begin{example}[The algebra $F_4$ in Table \ref{tab-caf}]
	If $S = k\langle x,y,z \rangle/(xy+yx, yz - zy, zx-xz)\in \mathcal{A}_{3,0}$ and $A=S/(x^2)\in \cA_{3, 1}$, then it is easy to check that \[ \sigma = \begin{pmatrix}
		1 & 0  & 0 \\
		1 &  1  &  0\\
		0 & 0 &  1
	\end{pmatrix} \in \text{\rm Aut}^{\mathbb{Z}}A\setminus \Aut^{\ZZ}S. \]
	If $A^\sigma \in \mathcal{A}_{3,1}$, then $A^{\s}=S'/(f')$ for some $S'\in \cA_{3, 0}$ and $f'\in RN(S')_2$, so there exists ${f'}^!\in RN((A^{\s})^!)_2$  by Lemma \ref{lem.asf}.  Since $A^{\sigma} \cong  k\langle x,y,z \rangle/(xy+yx, yz - zy +xz, zx-xz, x^2)$,  and  $(A^\sigma)^!\cong k\<x, y, z\>/(xy-yx, yz + zy, zx+xz-yz, y^2, z^2)$, we can check that $RN((A^\sigma)^!)_2 = \emptyset$ by direct calculation, so $A^\sigma \notin \mathcal{A}_{3,1}$. 
\end{example}

\begin{remark} 
It is very nice if we can prove that $\cA_{3, 1}^c$ and $\cA_{3, 1}$ are the same up to graded Morita equivalence (see \cite{SV}).  Let $A=S/(f)\in \cA_{3, 1}$.  If there exists $\s\in \Aut ^{\ZZ}S$ such that $\s(f)=f$ and $\s^2$ is the normalizing automorphism of $f$, then $S^{\s}\in\cA_{3, 0}$ and $f\in Z(S^{\s})_2$, so $A^{\s}\cong S^{\s}/(f)\in \cA_{3, 1}^c$ such that $\GrMod A^{\s}\cong \GrMod A$ (see example below), but we do not know if we can always find such a $\s$. 
\end{remark}

\begin{example}
     Let $S=k\langle x,y,z \rangle/(xy+yx,yz-zy,zx-xz) =k_{-1}[x, y][z]\in \cA_{3,0}$, $f = xy\in RN(S)_2$, and $A=S/(f)$. Since the normalizing automorphism of $f$ is 
$$
\nu = \begin{pmatrix}
        -1 &  0  &  0 \\
        0 &  -1  &  0 \\
        0 &  0  &  1
\end{pmatrix}\neq \id,
$$ $A\in \cA_{3, 1}\setminus \cA_{3, 1}^c$.  
Since 
$$\sigma = \begin{pmatrix}
        \sqrt{-1} &  0  &  0 \\
        0 &  -\sqrt{-1}  &  0 \\
        0 &  0  &  1
\end{pmatrix}\in \Aut ^{\ZZ}S
$$ 
such that $\sigma^2 = \nu$ and $\sigma(f) = f$, we see that $S^{\sigma} = k\langle x,y,z \rangle/(xy-yx,yz-izy,zx-ixz)\in \cA_{3, 0}$ and $f = xy \in Z(S^{\sigma})_2$, so  $A^{\sigma}\cong S^{\sigma}/(f)\in \cA_{3, 1}^c$ such that $\GrMod A^{\s}\cong \GrMod A$.
\end{example}

\subsection{Noncommutative Bezout's theorem}

\begin{lemma} \label{lem.08} 
If $b(t)\in \ZZ[t]$ and $a(t)=\dfrac{b(t)}{1-t}=\sum_{i\in \NN}a_it^i\in \ZZ[[t]]$, then $a_i=b(1)$ for every $i\geq \deg b(t)$. 
\end{lemma} 

\begin{proof} If $\deg b(t)=d$ so that $b(t)=\sum_{i=0}^db_it^i$, then 
\begin{align*}
a(t) = &\dfrac{b(t)}{1-t}=\left(\sum_{i=0}^db_it^i\right)\left (\sum_{i\in \NN}t^i\right) \\
 = & b_0+(b_0+b_1)t+(b_0+b_1+b_2)t^2+\cdots +(b_0+b_1+\cdots +b_d)(t^d+t^{d+1}+\cdots ),
\end{align*}
so $a_i=b_0+b_1+\cdots +b_d=b(1)$ for every $i\geq d$.  
\end{proof}

\begin{lemma} \label{lem.09}
Let $A$ be a locally finite  $\NN$-graded algebra 
such that $H_A(t)=\dfrac{b(t)}{1-t}$ for some $b(t)\in \ZZ[t]$.  If $f\in RN(A)_d$, then $\GKdim A=1$ and $\dim _kA[f^{-1}]_0=b(1)$. 
\end{lemma} 

\begin{proof} Since $H_A(t)=\dfrac{b(t)}{1-t}$ for some $b(t)\in \ZZ[t]$, $\GKdim A\leq 1$.  Since $k[f]\subset A$ and $f\in A_d$ is regular, $\GKdim A\geq \GKdim k[f]=1$.  Since $f\in A_d$ is a regular normal element,
$$A_0\subset A_df^{-1}\subset A_{2d}f^{-2}\subset A_{3d}f^{-3}\subset \cdots \subset A[f^{-1}]_0,$$
and $A[f^{-1}]_0=\cup_{i\in \NN}A_{id}f^{-i},$ so 
$$\dim _kA[f^{-1}]_0=\lim _{i\to \infty}\dim _kA_{id}f^{-i}=\lim_{i\to \infty}\dim _kA_{id}=b(1)$$  
 by Lemma \ref{lem.08}.
\end{proof}

\begin{theorem}[Generalized Bezout's Theorem] \label{thm.12} 
Let $A=S/(h_1, \dots, h_{n-1})$ be a noncommutative complete intersection where $S \in \cA_{n,0}$.
If $f\in RN(A)_d$, then 
$\dim _kA[f^{-1}]_0=\prod _{i=1}^{n-1}\deg h_i$. 
\end{theorem} 

\begin{proof} If $\deg h_i=d_i$, then 
$$H_A(t)=(1-t^{d_1})\cdots (1-t^{d_{n-1}})H_S(t)=\dfrac{(1-t^{d_1})\cdots (1-t^{d_{n-1}})}{(1-t)^n}=\dfrac{\sum_{i=0}^{d_1-1}t^i\cdots \sum_{i=0}^{d_{n-1}-1}t^i}{1-t}=:\dfrac{b(t)}{1-t},$$
so $\dim_kA[f^{-1}]_0
=b(1)=d_1\cdots d_{n-1}$ by Lemma \ref{lem.09}.  
\end{proof}

\begin{theorem}[Affine Bezout's Theorem] \label{thm.abt} 
Let $F = (f_1,\dots,f_n)$ be a sequence in $R = k[x_1, \dots, x_n]$.   
If $F^\vee$ is regular  in $R$ (e.g., if $F$ is strongly regular in $R$), then  
$\dim_k R/I_F=\prod _{i=1}^{n}\deg f_i$. 
\end{theorem}  

\begin{proof}  
Assume $F^\vee$ is a regular sequence in $R$. 
By Lemma \ref{lem.612}, $(F^z,z)$ is regular in $R[z]$,
so 
$$
\dim _k R/I_F=\dim_k \sD_z(\sH^z(F))= \dim_k \sH^z(F) [z^{-1}]_0=\prod _{i=1}^{n}\deg (f_i)^z=\prod _{i=1}^{n}\deg f_i
$$
by Corollary \ref{cor.DHc} (2), Lemma \ref{lemq.06} and Theorem \ref{thm.12}.  
\end{proof} 

Combining with Remark \ref{rem.52}, we have the following result. 

\begin{corollary} \label{cor-abtxy}
If $f,g \in k[x,y]$ such that $\gcd(f^\vee, g^\vee) = 1$, then $\dim_k k[x,y]/(f,g) = \deg f \deg g$. 
\end{corollary}

\begin{lemma}[{\cite[Lemma 5.5]{HMTW}}] \label{lem.wbwd} 
If $B\in \cB_{n, n-1}$, then there exists $z\in RZ(B)_1$ so that $\sD(B)$ is well-defined.  
\end{lemma}

\begin{lemma} \label{lem.bzb} 
If 
$$
B=k[x_1, \dots , x_n]/(h_1, \dots, h_{n-1})
$$ 
is a complete intersection, then $\sD(B)$ is a $\prod _{i=1}^{n-1}\deg h_i$-dimensional Frobenius algebra. 
\end{lemma} 

\begin{proof} 
By Theorem \ref{thm.12}, $\dim_k\sD(B)=\prod _{i=1}^{n-1}\deg h_i$.  
Since $B$ is a complete intersection, $B$ is a graded Gorenstein ring.   For every $\frak m\in \Max \sD(B)$, there exists a homogeneous prime ideal $\frak n$ of $B$ such that $\sD(B)_{\frak m}\cong B_{(\frak n)}$, so $\sD(B)_{\frak m}$ is a local Gorenstein ring of $\kdim \sD(B)_{\frak m}=\kdim \sD(B)=0$, hence $\sD(B)_{\fm}$ is self-injective.  Since $\sD(B)$ is artinian, $\sD(B)\cong \prod _{\frak m\in \Max \sD(B)}\sD(B)_{\frak m}$ is Frobenius.
\end{proof}

\begin{remark} 
Let $B=k[x_1, \dots, x_n]/(h_1, \dots, h_{n-1})$ be a complete intersection.  Since 
there exists a natural isomorphism $\phi:X:=\Proj B\to \Spec \sD(B)=:Y$, $\cO_{X, p}\cong B_{(\frak m_p)}\cong \sD(B)_{\frak m_{\phi (p)}}$ for every $p\in X$, and $\sD(B)\cong \prod _{q\in Y}\sD(B)_{\frak m_q}$,   
$$\sum_{p\in X}\dim _k\cO_{X, p}=\sum_{p\in X}\dim _kB_{(\frak m_p)}=\sum_{q\in Y}\dim _k\sD(B)_{\frak m_q}=\dim _k\sD(B)=\prod _{i=1}^{n-1}\deg h_i$$ 
by Theorem \ref{thm.12}. 
In particular, if $B=k[x, y, z]/(f, g)$, $C=\Proj k[x, y, z]/(f), D=\Proj k[x, y, z]/(g)$, then, for $p\in \PP^2$, $\cO_{\PP^2, p}/(f, g)\cong \begin{cases} \cO_{X, p} & p\in X \\
0 & p\not\in X,\end{cases}$ so 
$$C\cdot D:=\sum_{p\in \PP^2}\dim _k\cO_{\PP^2, p}/(f, g)=\sum_{p\in X}\dim _k\cO_{X, p}=\deg f\deg g,$$ 
recovering the usual Bezout's Theorem.  
\end{remark}

\subsection{The algebra $C(A)$} \label{subsec-camapc}

The algebra $C(A)$ defined below is important to study a noncommutative quadric hypersurface $A\in \cA_{n, 1}$.

\begin{definition} 
If $A=S/(f)\in \cA_{n, 1}$, then 
we define $C(A):=A^![(f^!)^{-1}]_0$.  
\end{definition} 

It is unclear from the definition, but the following lemma shows that $C(A)$ is independent of the choice of $S\in \cA_{n, 0}$ and $f\in RN(S)_2$ such that $A=S/(f)$. 

\begin{lemma}[{\cite[Lemma 5.3 (2)]{HMM}}]\label{lem.32} 
For $A, A'\in \cA_{n, 1}$, if $\GrMod A\cong \GrMod A'$, then $C(A)\cong C(A')$ as algebras.  
\end{lemma} 

\begin{lemma} \label{lem.CSD} 
If $A$ is an AS-Gorenstein algebra, 
then 
$\id (\cA)\leq \id (A)-1$. 
\end{lemma} 

\begin{proof} If $\id(A)=d$, then $\Ext_{\cA}^i(\cM, \cA)\cong \Ext^{d-1-i}_{\cA}(\cA, \cM(-\ell))^*=0$ for every $\cM\in \tails A$ and every $i\geq d$ by classical Serre duality \cite[Theorem 2.3]{YZ} where $\ell$ is the Gorenstein parameter of $A$, so $\id (\cA)\leq d-1$.
\end{proof} 

\begin{lemma} \label{lem.36} 
If $A$ is an AS-Gorenstein algebra of $\id (A)=\GKdim A=1$, and $f\in RN(A)_d$,
then 
$A[f^{-1}]_0$ is a self-injective algebra.  
\end{lemma} 

\begin{proof}  
By Lemma \ref{lem-pitails}, $\grmod A\to \mod A[f^{-1}]_0; \;  M\mapsto M[f^{-1}]_0$ induces an isomorphism $\Projn A\to \Specn A[f^{-1}]_0$, that is, an equivalence functor $\tails A\to \mod A[f^{-1}]_0$ sending $\cA$ to $A[f^{-1}]_0$, so $\id A[f^{-1}]_0=\id \cA=\id A-1=0$ by Lemma \ref{lem.CSD}.
\end{proof}

\begin{theorem} \label{thm.37} 
If $A\in \cA_{n, 1}$, then $C(A)$
is a $2^{n-1}$-dimensional self-injective algebra.  
\end{theorem} 

\begin{proof}  
By Lemma \ref{lem.KAS} (2), $A^!$ is an AS-Gorenstein algebra of $\id A^!=\GKdim A^!=1$, so 
$C(A)$ is self-injective by Lemma \ref{lem.36}. Since $H_A(t)=(1-t^2)/(1-t)^n=(1+t)/(1-t)^{n-1}$, $H_{A^!}(t)=(1+t)^n/(1-t)$, so $\dim _kC(A)=\dim _kA^![f^{-1}]_0=2^{n-1}$ by Lemma \ref{lem.09}.
\end{proof} 

\begin{remark}  It was shown in \cite[Propositions 4.5, 4.6, 5.3]{HY} (see also \cite[Lemma 2.6]{H}) that $C(A)$ is self-injective in the case that $f\in Z(S)_2$. 
\end{remark}


\begin{lemma} \label{lem.cwt} 
Let $A\in \cA_{n, 1}$.  If $RN(A^!)_1\neq \emptyset$, 
then $C(A)\cong \sD(A^!)$ as algebras. 
\end{lemma} 

\begin{proof} Since $\GKdim A^!=1$ by Lemma \ref{lem.KAS} (2), $\sD(A^!)$ is well-defined by Lemma \ref{lem.awwd}.  
Since $\Specn C(A)\cong \Projn A^!\cong \Specn \sD(A^!)$
as quasi-schemes, $C(A)\cong \sD(A^!)$ as algebras by Lemma \ref{lem.spnc}. 
\end{proof} 

\begin{lemma} \label{lem.47} For $B, B'\in \cB_{3, 2}$, $B\cong B'$ as graded algebras if and only if $\sD(B) \cong \sD(B') $ as algebras.  
\end{lemma} 

\begin{proof} By Lemma \ref{lem.awwd}, $B\cong B'$ implies $\sD(B) \cong \sD(B')$.  By Lemma \ref{lem.bzb}. $\sD(B)$ is a 4-dimensional commutative Frobenius algebra.  Since $\#(\cB_{3, 2})=6$ (cf. \cite[Remark 3.29]{HM}), and there are exactly 6 isomorphism classes of 4-dimensional commutative Frobenius algebras (cf. \cite[Corollary 4.10]{H}), the result follows. 
\end{proof}

\begin{theorem} \label{thm.43} 
Let $A, A'\in \cA_{3, 1}$ such that $A^!, {A'}^!$ are commutative.  Then $A\cong A'$ as graded algebras if and only if $C(A)\cong C(A')$ as algebras. 
\end{theorem} 

\begin{proof}  
Since $A^!, {A'}^!\in \cB_{3,2}$ by Lemma \ref{lem.48}, $\sD(A^!), \sD({A'}^!)$ are well-defined by Lemma \ref{lem.wbwd}, so $A\cong A'$ as graded algebras if and only if $A^!\cong {A'}^!$ as graded algebras if and only if $\sD(A^!) \cong \sD({A'}^!)$ as algebras by Lemma \ref{lem.47} if and only if $C(A)\cong C(A')$ as algebras by Lemma \ref{lem.cwt}.
\end{proof}

\begin{lemma} \label{lem.nag} 
Let $S$ be a graded algebra and $f=\sum f_i\in S$ a normal element of degree $d$ where $f_i\in S_i$.  If there exists $\nu\in \Aut^{\ZZ}(S)$ such that $gf=f\nu(g)$ for $g\in S$, 
then $gf_i=f_i\nu(g)$ for $g\in S$ so that  
$f_i$ is a normal element
for every $i\in \ZZ$.  In this case, $f^z \in S[z]_d$ is also a normal element.  Moreover, if $f \in Z(S)$, then $f^z\in Z(S[z])$.  
\end{lemma} 

\begin{proof} 
Since $\nu\in \Aut^{\ZZ}(S)$, for every $g_j\in S_j$, $\nu(g_j)\in S_j$, so 
$g_jf_{i}=(g_jf)_{i+j}=(f\nu(g_j))_{i+j}=f_i\nu(g_j)$.  
It follows that $gf_i=f_i\nu(g)$ for $g\in S$, so $f_i$ is a normal element by Lemma \ref{lem.reno} (1). 
Since $f^z=\sum _{i=0}^df_{d-i}z^i$, 
$$
gf^z=g\sum _{i=0}^df_{d-i}z^i=\sum _{i=0}^df_{d-i}\nu (g)z^i=\sum _{i=0}^df_{d-i}z^i\nu(g) = f^z\nu(g)
$$
for $g\in S$. 
For $h=\sum a_iz^i\in S[z]$ where $a_i\in S$, 
$$
h f^z=\left(\sum a_iz^i\right) f^z=\sum a_i f^z z^i=\sum f^z \nu(a_i)z^i = f^z\left(\sum \nu(a_i)z^i\right) = f^z\bar \nu(h)
$$
where $\bar \nu\in \Aut^{\ZZ}S[z]$ is defined by $\bar \nu(g)=\nu(g)$ for $g\in S$ and $\bar \nu(z)=z$, so $f^z \in S[z]$ is a normal element by Lemma \ref{lem.reno} (1). 
If $f\in Z(S)$, then we may take $\nu=\id$ so that $\bar \nu=\id$, so $f^z \in Z(S[z])$.  
\end{proof}

\begin{theorem} \label{thm.an1}  
Let $S = k \langle x_1, \dots, x_n \rangle /I$ be a graded algebra, and $F = (f_1, \dots, f_m)$ a strongly regular normal sequence  
which is either central or  homogeneous.   
\begin{enumerate}
\item{} $F^z$ is normal in $S[z]$. 
\item{}  $(F^z,z)$ is a homogeneous regular normal sequence in $S[z]$.
\item{} If $S\in \cA_{m, 0}$, then $E = S/I_F$ is a self-injective algebra of $\dim_k E=\prod_{i=1}^m \deg f_i$.  
\end{enumerate}
\end{theorem} 

\begin{proof} 
(1) If $F$ is central in $S$, then $F^z$ is central in $S[z]$ by Lemma \ref{lem.nag}, so 
$$
\overline {(f_i)^z}\in Z(S[z]/((f_1)^z, \dots, (f_{i-1})^z).
$$
If $F$ is homogeneous, then the normalizing automorphism of $\overline {f_i}\in S/(f_1, \dots, f_{i-1})$ is graded, 
so 
$$
\overline {(f_i)^z}\in S/(f_1, \dots , f_{i-1})[z]=S[z]/((f_1)^z, \dots, (f_{i-1})^z)
$$ 
is normal by Lemma \ref{lem.nag}. 
In either case, $F^z$ is normal in $S[z]$. 

(2) This follows from (1) and Lemma \ref{lem.612}.

(3) If $S \in \cA_{m, 0}$, 
then $\sH^z(F)$ is a noncommutative complete intersection by (2) and Lemma \ref{lem.Ore}.  
It follows that $\sH^z(F)$ is an AS-Gorenstein algebra of $\id(\sH^z(F))=\GKdim \sH^z(F) =1$ by Lemma \ref{lem.Tak} and Lemma \ref{lem.Ree} (see the proof of Lemma \ref{lem.KAS} (1)), 
so $E \cong \sD_z(\sH^z(F))) \cong \sH^z(F)[z^{-1}]_0$ 
is self-injective by Proposition \ref{prop.05n}, Lemma \ref{lemq.06} and Lemma \ref{lem.36}.  
By Theorem \ref{thm.12},  $\dim_k E=\dim _k \sH^z(F)[z^{-1}]_0=\prod_{i=1}^m \deg f_i$.  
\end{proof}  

\begin{example} 
Let $f=x^2-y, g=xy\in k[x, y]$.  Since $k[x, y]/(f)\cong k[x]$ is a domain, $(f, g)$ is a regular sequence of degree 2, 
but since $f^\vee=x^2, g^\vee=xy$, $(f, g)$ is not a strongly regular sequence, so $E = k[x, y]/(f, g)\not \in \cA_{2, 2}^\vee$.  Although $E \cong k[x]/(x^3)$ is Frobenius, $\dim _k E=3$. 
\end{example}

\subsection{Noncommutative affine pencils of conics} 

For the convenience of the reader, we recall some results in \cite{HMTW}.

\begin{lemma}[{\cite[Corollary 3.23]{HMTW}}] \label{lem-comiffb}
$E \in \cA_{2, 2}^\vee$ is commutative if and only if $E\in \cB_{2, 2}^\vee$.
\end{lemma}

\begin{lemma} [{\cite[Corollary 5.11]{HMTW}}]\label{lem.nst} 
Let $E = S/(f, g), E' = S'/(f', g')\in \cA_{2,2}^\vee\setminus \cB_{2, 2}^\vee$ where $S=k\<x, y\>/(h), S'=k\<x, y\>/(h')\in \cA_{2, 0}$.  Then $E \cong E'$ if and only if $(f, g, h)\sim_{st}(f', g', h')$ in $k\<x, y\>$.  
\end{lemma}

\begin{theorem} [{\cite[Corollary 5.13]{HMTW}}] \label{thm.ANC}  
Every algebra in $\cA^\vee_{2,2}$ is isomorphic to exactly one of the algebras listed in Table \ref{tab-avee22} where 
$k_\l[x,y]/(x^2,y^2)\cong k_{\l'}[x, y]/(x^2. y^2)$ if and only if $\l'=\l^{\pm 1}$.
\begin{center}
\begin{table}[ht]
\begin{threeparttable}
\centering
\caption{Algebras in $\cA^\vee_{2,2}$.} \label{tab-avee22}
\begin{tabular}{|c|}
\hline
{\rm Commutative}  \\ \hline 
$k[x,y]/(x^2-1,y^2-1)$, $ k[x,y]/(x^2-y-1,y^2-1)$, $k[x,y]/(x^2-\frac{2}{\sqrt{3}}y-1,y^2-\frac{2}{\sqrt{3}}x-1)$, \\
$k[x,y]/(x^2,y^2-1)$, $k[x,y]/(x^2,y^2-x)$, $k[x,y]/(x^2,y^2) $ \\ 
\hline \hline
{\rm Not-commutative} \\ \hline  
$k_{-1} [x,y]/(x^2 + 1,y^2 + 1)$, $k_{-1} [x,y]/(x^2,y^2 + 1)$, 
$k_{-1}[x,y] /(x^2+yx, y^2)$, $k_\l[x,y]/(x^2,y^2)$ (where $\l \neq 0, 1$) \\\hline
\end{tabular}
\end{threeparttable}
\end{table}
\end{center}
\end{theorem}  

\begin{corollary} \label{cor.ANC}
For every $E = S/(f, g) \in \cA_{2, 2}^\vee$ where $S=k\<x, y\>/(h)\in \cA_{2, 0}$,  
there exist $S'=k\<x, y\>/(h')\in \cA_{2, 0}$ and either a homogeneous or central strongly regular normal sequence $(f', g')$ of degree 2 in $S'$ (so that $S'/(f', g')\in \cA_{2, 2}^\vee$) such that $(f, g, h)\sim_{st}(f', g', h')$ in $k\<x, y\>$.
\end{corollary} 

\begin{proof} By the above table, if $E$ is not commutative and if $(f, g)$ is not a homogeneous sequence, then $(f, g)=(x^2+1, y^2+1), (x^2, y^2+1)$ which are central sequences by \cite[Lemma 3.15 (4)]{HMTW}.  
\end{proof}

  

\section{Main results}

 
In this section, we give a complete classification of noncommutative central conics up to isomorphism. Let  
\begin{align*}
\cA_{3, 1}^w & :=\{A\in \cA_{3, 1}^c \mid  RN(A^!)_1\neq \emptyset \}/\cong, \\
\cA_{3, 1}^z & :=\{A\in \cA_{3, 1}^c \mid  RZ(A^!)_1\neq \emptyset \}/\cong.
\end{align*} 
In this notation, for $A\in \cA^c_{3, 1}$, $A\in \cA_{3, 1}^w$ if and only if  $\sD(A)$ is well-defined by Lemma \ref{lem.awwd}.  

We will first show that there is a bijection between $\cA^\vee_{2,2}$ and $\cA_{3,1}^z$.  

\begin{theorem} \label{thm.csfg}
Let $F = (f,g), F' = (f',g')$ be strongly regular normal sequences of degree $2$ in $S, S'\in \cA_{2, 0}$ respectively so that $E=S/I_F, E'=S'/I_{F'}\in \cA_{2, 2}^\vee$. 
\begin{enumerate}
\item{} 
$\sH^z(S, F) \in \cA_{3,2}$. 
\item{} 
If $E \cong E'$ as algebras, then $\sH^z(S, F) \cong \sH^z(S', F')$ as graded algebras.
\item{} The map
\begin{equation}\label{equ-mapbv22}
\sH^z: \cA^\vee_{2,2} \to \cA_{3,2}; \, E = S/I_F \mapsto \sH^z(E): = \sH^z(S, F),
\end{equation}
is well-defined which restricts to $\sH^z: \cB^\vee_{2,2} \to \cB_{3,2}$, 
and $\sD_z(\sH^z(E))\cong E$ for $E\in \cA_{2, 2}^\vee$. 
\end{enumerate}
\end{theorem}

\begin{proof}
(1) By Lemma \ref{lem.st} and Corollary \ref{cor.ANC}, we may assume that $F$ is a homogeneous or central sequence.  By Lemma \ref{lem.Ore} and Theorem \ref{thm.an1}, $\sH^z(S, F)=S[z]/I_{F^z} \in \cA_{3,2}$.  

(2) If $E$ is commutative, then $E, E' \in \cB^\vee_{2,2}$ by Lemma \ref{lem-comiffb}, so that $\sH^z(F), \sH^z(F') \in \cB_{3,2}$. Since 
$$
\sD_z(\sH^z(F)) \cong E \cong E' \cong \sD_z(\sH^z(F'))
$$ 
as algebras by Proposition \ref{prop.05n}, $\sH^z(F) \cong \sH^z(F')$ as graded algebras by Lemma \ref{lem.47}.

If $E$ is not commutative, then $E, E' \in \cA^\vee_{2,2} \setminus \cB^\vee_{2,2}$, so $E \cong E'$ if and only if $(f,g,h) \sim_{st} (f',g',h')$ in $k \< x,y \>$ where $S =k \< x,y \> /(h), S' =k \< x,y \> /(h') \in \cA_{2,0}$ by Lemma \ref{lem.nst}. By Lemma \ref{lem.st}, $\sH^z(S, F) \cong \sH^z(S',F')$. 

(3) The map $\sH^z: \cA^\vee_{2,2} \to \cA_{3,2}$ is well-defined by (2), and $\sD_z(\sH^z(E)) \cong E$ by Proposition \ref{prop.05n}. 
\end{proof}

\begin{remark} 
Let $F = (f_1, \dots, f_n)$ be a sequence of degree $\geq 2$ in $k[x_1, \dots, x_n]$, and 
$$
E = k[x_1, \dots, x_n]/I_F.
$$ 
If $F^\vee$ is regular in $k[x_1, \dots, x_n]$ and $\dim_k E = 2^n$, 
then $\prod _{i=1}^{n} \deg f_i = \dim_k E = 2^{n}$ by Theorem \ref{thm.abt}, so $\deg f_i=2$ for every $i=1, \dots, n$. 
In particular, if $k[x,y]/(f,g)$ is a 4-dimensional algebra such that $\deg f, \deg g \geq 2$ and that $\gcd (f^\vee, g^\vee) = 1$,
then $(f^\vee, g^\vee)$ is a regular sequence by Remark \ref{rem.52}, so $\deg f=\deg g=2$ so that $k[x,y]/(f,g) \in \cB^\vee_{2, 2}$ by Remark \ref{rem.92}.
\end{remark}

\begin{theorem} \label{thm.Inc2} 
Let $A=S/(f)\in \cA_{3,1}$.  If $A^!$ is commutative, then $C(A) \in \cB^\vee_{2,2}$, and  
$
\sH^z(C(A))^!\cong A.
$ 
\end{theorem} 

\begin{proof} 
By Lemma \ref{lem.48}, we have ${A}^!\in \cB_{3, 2}$. Since $C(A)\cong \sD(A^!)\in \cB^\vee_{2, 2}$ by Lemma \ref{lem.bcn} and Lemma \ref{lem.cwt}, $\sH^z(C(A))\in \cB_{3, 2}$ by Theorem \ref{thm.csfg}.  
Since $\sD(\sH^z(C(A))) \cong C(A)\cong \sD({A}^!)$ by Proposition \ref{prop.05n}, $\sH^z(C(A)) \cong {A}^!$ by Lemma \ref{lem.47}, 
so $\sH^z(C(A))^!\cong A$.   
\end{proof}

\begin{lemma} \label{lem.sDc} 
Let $A$ be a graded algebra and $z\in RZ(A)_1$.  If $\sD_z(A)$ is commutative, then $A$ is commutative. 
In particular, for $A\in \cA_{3, 1}^z$,  if $C(A)$ is commutative, then $A^!$ is commutative.  
\end{lemma} 

\begin{proof} For every homogeneous elements $f, g\in A$, if  $\sD_z(A)$ is commutative, then
$$fgz^{-\deg f-\deg g}=(fz^{-\deg f})(gz^{-\deg g}) =(gz^{-\deg g})(fz^{-\deg f})= gfz^{-\deg f-\deg g},$$
so $fg=gf$. 

In particular, for $A\in \cA_{3, 1}^z$,  if $C(A)\cong \sD_z(A^!)$ is commutative by Lemma \ref{lem.cwt}, then $A^!$ is commutative.  
\end{proof} 

Let 
$$\cF_2: =\{4\text{-dimensional Frobenius algebras} \} / \cong.$$
For two graded algebras $A, A'$, we denote by $A\sim A'$ if $\GrMod A\cong \GrMod A'$. 
By Lemma \ref{lem.32}, Theorem \ref{thm.37} and Example \ref{ex-4dimsjfro}, the maps
$
C: \cA_{3,1} \to \mathcal{F}_{2}; \ A \mapsto C(A)
$
and 
$
\bar C: \cA_{3,1} / \sim  \to \mathcal{F}_{2}; \ A \mapsto C(A)
$
are well defined which make the following diagram 
$$
\xymatrix{
 & \cF_{2} \\
\cA_{3,1} \ar[ru]^-{C}  \ar[r]_-{\pi} & \cA_{3,1} / \sim \ar[u]_-{\bar C}
}
$$
commute, where $\pi$ is the natural projection. 
The following is one of the main results of \cite{HMTW}.  

\begin{theorem}[{\cite[Theorem 5.10]{HMTW}}] \label{thm.CF} 
$\cA_{2, 2}^\vee = \cF_2$.  
\end{theorem} 

\begin{theorem} \label{thm.c4f}
\begin{itemize}
\item[(1)] For $E \in \cA_{2, 2}^\vee$, $\sH^z(E)^!\in \cA^z_{3, 1}$ and $C(\sH^z(E)^!)\cong E$. 
\item[(2)] The map 
\begin{equation} \label{equ-mapnabla}
\nabla: \cA^\vee_{2,2} \to \cA^z_{3,1}; \ E \mapsto \nabla(E): = \sH^z(E)^!.
\end{equation}
is well-defined which makes the following  diagram
$$
\xymatrix{
\cA^\vee_{2,2} \ar@{<-}[r]^-{=} \ar[d]_-{\nabla} & \cF_2 \\
\cA^z_{3,1} \ar[r]_-{\pi} \ar[ru]^-{C} & \cA^z_{3,1}/\sim \ar[u]_-{\bar C}
}
$$
commute. 
\end{itemize}
\end{theorem} 

\begin{proof}  
(1) By Theorem \ref{thm.an1} and Theorem \ref{thm.csfg}, $\sH^z(E) = S[z]/(f^z, g^z)\in \cA_{3, 2}$ and $(f^z, g^z, z^2) $  is a regular normal sequence of degree $2$ in $S[z]$.  By Lemma \ref{lem.du},  $\sH^z(E)^!\in \cA^z_{3, 1}$, so
\begin{equation*}
C(\sH^z(E)^!)\cong \sD_z(\sH^z(E)) \cong E
\end{equation*}
by  Proposition \ref{prop.05n},  Lemma \ref{lem.cwt}. 

(2) This follows from (1), Theorem \ref{thm.csfg} and Theorem \ref{thm.CF}. 
\end{proof}

\begin{theorem} \label{thm.acb}
\begin{itemize}
\item[(1)] The map 
$
\Delta: \cA_{3,1}^z\to \cA_{2, 2}^\vee; \; A \mapsto \sD_z(A^!)
$
is well-defined. 
\item[(2)] $\nabla: \cA_{2, 2}^\vee \to \cA_{3, 1}^z$ is a bijection with the inverse $\Delta$. 
\end{itemize}
\end{theorem}

\begin{proof}
(1) For $A \in \cA_{3,1}^z$, $\sD_z(A^!) \cong C(A)$ by Lemma \ref{lem.cwt}, so $\Delta$ is well-defined by Theorem \ref{thm.c4f} (2). 

(2) For $E\in \cA_{2, 2}^\vee$, $\Delta (\nabla (E))\cong \sD_z(\sH^z(E))\cong E$ by Theorem \ref{thm.csfg} (3), so it is enough to show $\nabla(\Delta(A)) \cong \sH^z(\sD_z(A^!))^! \cong A$ for $A \in \cA^z_{3,1}$.
If  $\sD_z(A^!) \cong C(A)$ is commutative, then $A^!$ is commutative by Lemma \ref{lem.sDc}, so 
$A^!\cong \sH^z(\sD_z(A^!))$ by Theorem \ref{thm.Inc2}. 
 On the other hand, there are only 5 classes of algebras (4 algebras and 1 family of algebras up to isomorphism) $A\in \cA^z_{3, 1}$ 
such that $C(A) \cong \sD_z(A^!)$ is not commutative by classification (Theorem \ref{thm.Ta2}).  For each such $A$, we can check that $A^!\cong \sH^z(\sD_z(A^!))$ by direct computations (see the next Example). 
\end{proof}

\begin{example}[The algebra $A_2$ in Table \ref{tab-caa}] If $A=k\<x, y, z\>/(yz+zy, zx+xz, xy-yx, x^2+y^2+z^2)$, then $A\in \cA_{3, 1}^c$.  Since $A^!\cong k\<x, y, z\>/(yz-zy, zx-xz, xy+yx, x^2-z^2, y^2-z^2)$, $z\in RZ(A^!)_1$, so $A\in \cA_{3, 1}^z$, and 
$$C(A)\cong \sD_z(A^!)\cong A^!/(z-1)\cong k\<x, y\>/(xy+yx, x^2-1, y^2-1)\cong M_2(k)$$
is not commutative.  We can easily check that $\Delta(A) = \sD_z(A^!)\cong k\<x, y\>/(xy+yx, x^2-1, y^2-1)\in \cA_{2, 2}^\vee$ and that $\sH^z(\sD_z(A^!))\cong A^!$.    
\end{example} 

\begin{theorem} \label{thm.zwA} 
	The composition $\cA_{3, 1}^z\to \cA_{3, 1}^z/\sim \to \cA_{3, 1}^w/\sim\to \cA_{3, 1}^c/\sim$ is a bijection.  
\end{theorem} 

\begin{proof} 
	Bijectivity of the first map: 
	Since the map $\nabla: \cA^\vee_{2, 2}\to \cA^z_{3, 1}$ is bijective by Theorem \ref{thm.acb},  and the composition 
	$
	\xymatrix{
		\cA^\vee_{2, 2} \ar[r]^-{\nabla} & \cA^z_{3, 1} \ar[r]^-{\theta} & \cA^z_{3, 1}/\sim
	}
	$
	is injective by Theorem \ref{thm.c4f}, $\cA^z_{3, 1}=\cA_{3, 1}^z/\sim$.  
	
	Bijectivity of the second map:  This follows from Theorem \ref{thm.awz}.  
	
	Bijectivity of the last map:   There are only 10 classes of algebras $A\in \cA^c_{3, 1}\setminus \cA_{3, 1}^w$ 
	(8 algebras and 2 family of algebras up to isomorphism) by classification (Theorem \ref{thm.Ta2}).  For each $A\in \cA_{3, 1}^c\setminus \cA_{3, 1}^w$, we can check that there exists $A'\in \cA_{3, 1}^w$ such that $A\sim A'$ by direct computations (see the next Example).
\end{proof}

	\begin{example}[The algebra $A_4$ in Table \ref{tab-caa}]
		By classification (Theorem \ref{thm.Ta2}), 
		\begin{align*}
			& A_1= k[x, y, z]/(x^2+y^2+z^2)\in \cA_{3, 1}^w, \\
			& A_4 = k\<x, y, z\>/(xy-yx-y^2, yz-zy-2xy, zx-xz-yz, x^2+yz)\in \cA_{3, 1}^c\setminus \cA_{3, 1}^w.
		\end{align*}  
		It is not so obvious to see $A_4\sim A_1$, so we may need some trials and errors to show it.  
		Let $A'=k[x, y, z]/(x^2+y(x+z))$. 
		If $\sigma=\begin{pmatrix} 1 & -1 & 0 \\ 0 & 1 & 0 \\ 2 & 0 & 1 \end{pmatrix} \in \GL_3(k)$, then 
		$$\s(x^2+y(x+z))=
		(x-y)^2+y((x-y)+(2x+z))=x^2-2xy+y^2+3xy-y^2+yz=x^2+y(x+z),$$
		so $\s\in \Aut^{\ZZ}(A')$.  Since  
		\begin{align*}
			(A')^{\sigma} & \cong k\<x, y, z\>/(\s(x)y-\s(y)x, \s(y)z-\s(z)y, \s(z)x-\s(x)z, \s(x)x+\s(y)(x+z)) \\
			& = k\<x, y, z\>/(xy-yx-y^2, yz-2xy-zy, 2x^2+zx-xz+yz, x^2+yz)=A_4,
		\end{align*} 
		$A_4\sim A'\cong A_1 \in \cA_{3, 1}^w$. 
	\end{example} 

By Theorem \ref{thm.acb} and Theorem \ref{thm.zwA}, we have the following result. 
 
\begin{theorem} \label{thm.main}  
The map $C:\cA_{3,1}\to \cF_2$ induces a bijection $\cA_{3, 1}^c/\sim \to \cF_2$. 
\end{theorem} 

\begin{proof}  The map $C:\cA_{3. 1}\to \cF_2$ induces a map $\cA_{3, 1}^c/\sim \to \cF_2$ by Lemma \ref{lem.32} and  Theorem \ref{thm.37}.  For every $A\in \cA_{3, 1}^c$, there exists $A'\in \cA_{3, 1}^z$ such that $A\sim A'$ by Theorem \ref{thm.zwA}.  Since 
$$C(A)\cong C(A')\cong \sD({A'}^!)=\nabla (A')$$
by Lemma \ref{lem.32} and Lemma \ref{lem.cwt}, 
the result follows from Theorem \ref{thm.CF}, Theorem \ref{thm.acb} and Theorem \ref{thm.zwA}. 
\end{proof}

\begin{theorem} \label{con.main} 
For the following conditions on $A, A'\in \cA_{3, 1}$, 
\begin{enumerate}
\item{} $\GrMod A\cong \GrMod A'$. 
\item{} $\Projn A\cong \Projn A'$. 
\item{} $\GrMod A^!\cong \GrMod {A'}^!$. 
\item{} $\Projn A^!\cong \Projn {A'}^!$. 
\item{} $\uCM^{\ZZ}(A)\cong \uCM^{\ZZ}(A')$.  
\item{} $\cD^b(\mod C(A))\cong \cD^b(\mod C(A'))$. 
\item{} $C(A)\cong C(A')$. 
\end{enumerate} 
$(1) \Leftrightarrow  (2) \Leftrightarrow (3) \Rightarrow (4) \Leftrightarrow  (5)\Leftrightarrow  (6) \Leftrightarrow (7)$.  In particular, if $A, A'\in \cA_{3, 1}^c$, then all of the conditions are equivalent.  
\end{theorem} 

\begin{proof} (1) $\Leftrightarrow$ (2) by \cite[Lemma 3.3]{M2} since $A, A'$ are AS-Gorenstein algebras of dimensions 2 and of Gorenstein parameters 1 by Lemma \ref{lem.KAS} (2).

(1) $\Leftrightarrow$ (3) by \cite[Lemma 4.1]{MU2}.

(3) $\Rightarrow$ (4) by \cite[Theorem 1.4]{Z}.

(4) $\Rightarrow $ (5) by \cite[Lemma 2.5 (2)]{MU}.

(5) $\Leftrightarrow $ (6) by \cite[Lemma 4.13 (4)]{MU}.

(6) $\Leftrightarrow$ (7) by Theorem \ref{thm.37} and \cite[Theorem 4.11]{HMTW}. 

(7) $\Rightarrow $ (4) by \cite[Lemma 4.13 (2)]{MU}.

If $A, A'\in \cA_{3, 1}^c$, then
(7) $\Rightarrow $ (1) follows from Theorem \ref{thm.main}. 
\end{proof}   

\begin{remark} There exist $A, A'\in \cA_{4,1}$ such that 
$C(A)\cong C(A')\cong M_2(k)^{\times 2}$ but $\Projn A\not\cong \Projn A'$ by \cite[Theorem 4.6]{U}. 
\end{remark} 

 The following is the main result of the paper.  
\begin{corollary} \label{cor-main.bij}
    There are bijections among the following sets:
\begin{itemize}
    \item [(i)] the set  $\cF_2$ of isomorphism classes of 4-dimensional Frobenius algebras.
    \item [(ii)] the set  $\cA^\vee_{2,2}$ of isomorphism classes of noncommutative affine pencils of conics.
    \item [(iii)] the set  of isomorphism classes of noncommutative central conics.
\end{itemize} 
\end{corollary}

\begin{proof} 
By Theorem \ref{thm.CF}, Theorem \ref{thm.main}, and Theorem \ref{con.main}, the result holds.  
\end{proof} 

In \cite{HMTW}, Takeda and the authors give complete classifications of 4-dimensional Frobenius algebras and noncommutative affine pencils of conics, so, in this paper, we can give a complete classification of noncommutative central conics by the above bijections.  

\begin{theorem}\label{thm-iso-conics} Every noncommutative central conic is isomorphic to the noncommutative projective scheme associated to exactly one of graded algebras in Table \ref{intro.tab.alg2}, where $\Projn \mathcal{S}_\lambda/(x^2) \cong \Projn \mathcal{S}_{\lambda'}/(x^2)$ if and only if $\l' = \l^{\pm1}$. 

\begin{center}
\begin{table}[ht]
\begin{threeparttable}
\centering
\caption{$A \in A^c_{3,1}/\sim$.} \label{intro.tab.alg2}
\begin{tabular}{|c|}
\hline
{\rm Commutative}  \\ \hline 
$k[x,y,z]/(x^2)$, $k[x,y,z]/(x^2 + y^2)$, $k[x,y,z]/(x^2 + y^2 + z^2)$ \\ \hline \hline
{\rm Not-commutative}  \\ \hline 
$\mathcal{W}/(x^2)$, $\mathcal{S}_\lambda/(x^2)$, $\mathcal{S}/(x^2+y^2)$, $\mathcal{S}/(x^2+y^2+z^2)$, 
$\mathcal{N}/(x^2)$, $\mathcal{N}/(x^2 + y^2 - 4 z^2)$, \\ $\mathcal{N}/(3x^2 + 3y^2 + 4z^2)$ \\\hline
\end{tabular}
\begin{tablenotes}
\linespread{1}
\item\hspace*{-\fontdimen2\font} 
$\mathcal{S} := k \< x,y,z \> /(yz+zy,zx + xz, xy+yx)$, \\
$\mathcal{S}_\l := k \< x,y,z \> /(yz-\l zy,zx - xz, xy - yx)$, $\l \neq 0,1$,\\
$\mathcal{W} := k \< x,y,z \> /(yz-zy-y^2,zx - xz, xy-yx)$, \\
$\mathcal{N} := k \< x,y,z \> /(yz+zy+x^2,zx + xz+y^2, xy+yx)$. 
\end{tablenotes}
\end{threeparttable}
\end{table}
\end{center}
\end{theorem}  

\begin{proof} 
By Theorem \ref{thm.acb} and Theorem \ref{thm.zwA},  a complete list of algebras in $\cA_{3, 1}^c/\sim$ is given by applying the bijection $\nabla :\cA_{2, 2}^\vee\to \cA_{3,1}^z\to \cA_{3, 1}^c/\sim$ to Table \ref{tab-avee22}, which gives the above table.  By Theorem \ref{con.main} (1) $\Leftrightarrow$ (2), the result follows.  
\end{proof} 


\section{Explicit classifications} \label{sec-clafi}

\subsection{Noncommutative central conics}

We would like to find all algebras $A \in \cA^c_{3,1}$. 
The key step is to find all central elements of degree $2$ in $3$-dimensional quantum polynomial algebras $S \in \cA_{3,0}$ which is heavily depends on troublesome computation. 
We show how we calculate $Z(S)_2$ in this subsection. 

For a quadratic algebra $A = T(V)/(R)$, set
$$
\mathcal{V}(R): = \{ (p,q) \in \mathbb{P}(V^*) \times \mathbb{P}(V^*)  \mid f(p,q) = 0 \ \text{for all} \ f \in R \}.
$$

\begin{definition}[{\cite[Definition 4.3]{M2}}] 
We say that a quadratic algebra $A = T(V)/(R)$ satisfies (G1) condition 
if there is a pair $(E,\sigma)$ where $E \subset \mathbb{P}(V^*)$ is a projective scheme 
and $\sigma$ is a $k$-automorphism of $E$ such that 
$$
\mathcal{V}(R) = \{ (p ,\sigma(p)) \in \mathbb{P}(V^*) \times \mathbb{P}(V^*) \mid p \in E \}.
$$
The pair $(E,\sigma)$ is called the geometric pair of $S$, and $E$ is called the point scheme 
of $S$.  (For the scheme structure of $E$, we refer to \cite{ATV}.)
\end{definition}

\begin{theorem}[{\cite[Theorem 1]{ATV}}]
Every $3$-dimensional quantum polynomial algebra $S$ satisfies (G1) condition. 
\end{theorem}

There are following types of $3$-dimensional quantum polynomial algebras $S$ in terms of the point schemes $E$ \cite{IM2,Ma}. Each type can also be further subdivided according to the action of \(\sigma\).
\begin{description}
\item[{\rm Type P}] 
      $E$ is $\mathbb{P}^{2}$. 
\item [{\rm Type S}] 
      $E$ is a triangle. 
\item[{\rm Type S'}] 
      $E$ is a union of a line and a conic meeting at two points. 
\item[{\rm Type T}]
      $E$ is a union of three lines meeting at one point. 
\item[{\rm Type T'}]
      $E$ is a union of a line and a conic meeting at one point. 
\item[{\rm Type NC}]
          $E$ is a nodal cubic curve. 
\item[{\rm Type CC}] $E$ is a cuspidal cubic curve. 
\item[{\rm Type TL}] $E$ is a triple line. 
\item[{\rm Type WL}] $E$ is a union of a double line and a line. 
\item[{\rm Type EC}] $E$ is an elliptic curve. 
\end{description}

\begin{lemma}[{\cite[Lemma 3.6]{HMM}}]\label{lem-censig2}
Let $S$ be a $3$-dimensional quantum polynomial algebra with geometric pair $(E,\sigma)$ such that $E$ is reduced. If $Z(S)_2 \neq \{ 0\}$, then there is some irreducible component $E_0 \subset E$ such that $\sigma^2|_{E_0} = \id$.
\end{lemma}  

By the above theorem, 
to calculate $Z(S)_2$, we may check $\sigma^2|_{E_0}$ for some irreducible component $E_0 \subset E$ first, which can save a lot of computation. We discusses the cases of Type T and Type EC in detail. The rest cases are similar. 

\begin{example}[{\cite{IM2}}]\label{examT1}
Type T: $E = \mathcal{V}(xy(x-y))$ is a union of three lines meeting at one point. Type T is divided into Type T$_1$, T$_2$ and T$_3$.
\begin{enumerate}
\item[(i)] Type T$_1$. In this case, \(\sigma\) stabilizes each component. Specifically, $\sigma|_{\mathcal{V}(x)}(0 \mathpunct{:} b \mathpunct{:} c) = (0 \mathpunct{:} b \mathpunct{:} \alpha b + c)$,
$\sigma|_{\mathcal{V}(y)}(a \mathpunct{:} 0 \mathpunct{:} c) = (a \mathpunct{:} 0 \mathpunct{:} \beta a + c)$ and 
$\sigma|_{\mathcal{V}(x-y)}(a:a:c) = (a \mathpunct{:} a\mathpunct{:} -\gamma a+c)$.
If $Z(S)_2\neq 0$, then $\alpha \beta \gamma = 0$ .
\item[(ii)] Type T$_2$. In this case, \(\sigma\) interchanges two of its components. Specifically,  $\sigma|_{\mathcal{V}(x)}(0\mathpunct{:} b\mathpunct{:} c) =(b \mathpunct{:} 0 \mathpunct{:} \alpha b + c)$, $\sigma|_{\mathcal{V}(y)}(a \mathpunct{:} 0 \mathpunct{:} c) = (0 \mathpunct{:} a \mathpunct{:}\beta a + c)$ and $\sigma|_{\mathcal{V}(x-y)}(a \mathpunct{:} a\mathpunct{:} c) = (a \mathpunct{:} a \mathpunct{:} -\gamma a+c)$.  If $Z(S)_2\neq 0$, then  $\alpha+\beta = 0$ and $\gamma=0$.
\item[(iii)] Type T$_3$. In this case, \(\sigma\) circulates three components, so $Z(S)_2=0$.
\end{enumerate}
\end{example}

\begin{example}[{\cite[Theorem 3.13]{Ma}}]
Type EC:  $E = \mathcal{V}(x^3 + y^3 + z^3 - 3\lambda xyz)$ is an elliptic curve, where $\lambda\in k$ with $\lambda^3\neq 1$. Let $(E,o_E, +)$ be the group structure on $E$ by setting zero element $o_{E} = (1\mathpunct{:}-1\mathpunct{:}0)$.
\begin{itemize}
\item [(i)] EC-A: \(\sigma = \sigma_p\) is a translation by a point $p = (a \mathpunct{:} b\mathpunct{:} c)$ of $E$, where $p$ satisfies \(abc\neq0\) and \((a^3 + b^3 + c^3)^3\neq(3abc)^3\). If $Z(S)_2\neq 0$, then for $q \in E$, $\sigma_p^2(q) = q + 2p = q$.  This implies that $a = b $.
\item [(ii)] EC-B. In this case, $\sigma = \sigma_p\tau$, where $\sigma_p$ is a translation by a point \(p = (1\mathpunct{:}1\mathpunct{:}c)\) satisfying \(c^3 - 3\lambda c + 2 = 0\), and $\tau(a \mathpunct{:} b\mathpunct{:} c) = (b \mathpunct{:} a\mathpunct{:} c)$. 
This implies $(\sigma_p\tau)^2= \id$.
\item [(iii)] EC-E: $\sigma = \sigma_p\tau^2$, or $\sigma_p\tau^4$, where $\sigma_p$ is a translation by a point \(p = (\eta^8\mathpunct{:}\eta^4\mathpunct{:}1)\) with \(\eta\) is a primitive $9$th root of unity, and 
$\tau(a \mathpunct{:} b\mathpunct{:} c) = (b \mathpunct{:} a\mathpunct{:} \varepsilon c)$ with $\varepsilon$ a primitive $3$rd root of unity.
This implies  $|\sigma_p\tau^2|= |\sigma_p\tau^4| =3$, so $Z(S)_2=0$
\item [(iv)] EC-H: $\sigma = \sigma_p\tau, \sigma_p\tau^3$, where \(\sigma_p\) is a translation by the point \(p = (1\mathpunct{:}1\mathpunct{:}\lambda)\), and $\tau(a \mathpunct{:} b\mathpunct{:} c) = (\varepsilon^2 a + \varepsilon b + c\mathpunct{:} \varepsilon a + \varepsilon^2 b + c \mathpunct{:} a + b + c)$ with $\varepsilon$ a primitive $3$rd root of unity. This implies $|\sigma_p\tau|= |\sigma_p\tau^3| =4$, so $Z(S)_2=0$
\end{itemize}
\end{example}

The calculation process is mainly divided into the following two steps, and we will illustrate it with an example. Recall that $S = k\langle x,y,z\rangle/(f_1,f_2,f_3)$ for some homogeneous elements $f_1,f_2,f_3$ of degree 2.

{\bf Step 1:} The key to the calculation of $Z(S)_2$ lies in obtaining a $k$-basis of $S_3$ by using the Diamond Lemma \cite[Theorem 1.15]{Rog}, so we need to know the noncommutative Gr\"obner basis for the ideal $(f_1,f_2,f_3) \subset k\<x, y, z\>$. We use the computer algebra system GAP to calculate noncommutative Gr\"obner basis.

{\bf Step 2:} We represent $gf$ and $fg$ as linear combination of a fixed $k$-basis of $S_3$ for $f\in S_2,g\in S_1$ to check if $f$ is a  central element by comparing thier coefficients of the $k$-basis.

We give a concrete example of Type T$_1$.

\begin{example}[Type T$_1$ algebra in Table \ref{tab-cenele}]
Let $S = k\langle x,y,z \rangle /(f_1,f_2,f_3)$, where 
$$
f_1 = xy-yx, \; f_2 = xz - zx - \beta x^2 + (\beta + \gamma)yx, \; f_3 = yz - zy - \alpha y^2 + (\alpha + \gamma)xy
$$ 
with $\alpha + \beta + \gamma\neq0$. 
If we fix an ordering on the variables: $x<y<z$, then the leading words of $\{f_1,f_2,f_3\}$ are $yx, zx, zy$.
    
By GAP (The program can be refered to \cite[Exercise Set 1.7]{Rog}), we know $\{f_1,f_2,f_3\}$ is a noncommutative Gr\"obner bases of $I=(f_1,f_2,f_3)$ in $k\langle x,y,z \rangle$. By \cite[Theorem 1.15]{Rog}, the images of the reduced words  (i.e. the words does not contain any of the leading words as a subword) in $k\langle x,y,z \rangle/(f_1,f_2,f_3)$ form a $k$-basis, so 
\begin{equation} \label{equ-s3}
    S_3 = \mathrm{span}_k\{x^3,x^2y, x^2z, xy^2, xyz,xz^2,y^3,y^2z,yz^2,z^3 \}.
\end{equation}

Assume $f\in Z(S)_2$ and $f=ax^2+bxy+cxz+dy^2+eyz+gz^2$, where $a,b,c,d,e,g \in k$.  We  represent $xf$ (resp. $yf,zf$) and $fx$ (resp. $fy,fz$) as linear combinations of the $k$-basis of $S_3$ listed in Equation \ref{equ-s3} and  determine the coefficients of $f$ by using the equation $xf = fx$ (resp. $yf = fy, zf = fz$).

(x) We have $xf = ax^3 + bx^2y + cx^2z+dxy^2 + exyz +gxz^2 $ and 
\begin{align*}
    fx =& ax^3+ bxyx + cxzx + dy^2x + eyzx + gz^2x\\
    =&ax^3 + bx^2y + dxy^2 + cx((\beta+\gamma)yx+xz-\beta x^2) + ey ((\beta+\gamma)yx+xz-\beta x^2) \\
    &+ g((\beta+\gamma)(\beta+\gamma -\alpha)xy^2 + 2\beta^2x^3+(\beta + \gamma)(\alpha + \gamma - 3\beta )x^2y -2\beta x^2z + 2(\beta +\gamma)xyz + xz^2)\\
    =& (a-\beta c + 2\beta^2g)x^3 + (b+(\beta +\gamma)c - \beta e + (\beta + \gamma)(\alpha + \gamma - 3\beta )g)x^2y +(c-2\beta g)x^2z\\
    & + (d+ (\beta+\gamma)e  + (\beta+\gamma)(\beta+\gamma -\alpha)g)xy^2 + (e + 2(\beta+\gamma)g)xyz +gxz^2.
\end{align*}

(y) We have $yf = ax^2y + bxy^2 + cxyz + dy^3+ ey^2z + gyz^2$ and
\begin{align*}
    fy  =& ax^2y + bxy^2 + cxzy + dy^3 + eyzy + gz^2y\\
     =& ax^2y + bxy^2 + cx((\alpha+\gamma)xy+yz-\alpha y^2) + dy^3 + ey((\alpha+\gamma)xy+yz-\alpha y^2)\\
    & + g((\alpha+\gamma)(\alpha+\gamma -\beta)x^2y + 2\alpha^2y^3+(\alpha + \gamma)(\beta + \gamma - 3\alpha)xy^2 -2\alpha y^2z + 2(\alpha +\gamma)xyz + yz^2)\\
    =& (a + (\alpha+\gamma)c + (\alpha+\gamma)(\alpha+\gamma -\beta)g)x^2y + (b -\alpha c + (\alpha+\gamma)e + (\alpha + \gamma)(\beta + \gamma - 3\alpha)g)xy^2 + gyz^2\\
    & + (c + 2(\alpha + \gamma)g)xyz + (d - \alpha e + 2\alpha^2g)y^3 + (e-2\alpha g)y^2z.
\end{align*}

 It is easy to see that the results in (x) and (y) are symmetric about $x$ and $y$ ($\alpha$ and $\beta$) since the defining relations are symmetric.

Now, we analyze the coefficients of $f$. First, $\alpha + \beta + \gamma \neq 0$ implies that either $\alpha+\gamma$ or $\beta$ is nonzero, and also either \(\beta+\gamma\) or \(\alpha\) is nonzero.  Since we have $(\alpha +\gamma)g=0=\beta g$ by comparing the coefficients of corresponding terms of $xf$ and $fx$ ($yf$ and $fy$), we have $g=0$. Since $(\alpha+\gamma)c + (\alpha+\gamma)(\alpha+\gamma -\beta)g = 0 = -\beta c + 2\beta^2g$, we have $c=0$ s. Similarly, we have $e = 0$, so $f = ax^2 + bxy + d y^2$. 

(z) We have $f z = ax^2z + bxyz + dy^2z$ and 
\begin{align*}
    z f  =& a zx^2 + b zxy + d zy^2\\
        =& a(2(\beta+\gamma)x^2y + x^2z - 2\beta x^3) + b((\beta + \gamma -\alpha)xy^2 + (\alpha + \gamma -\beta)x^2y +xyz)\\
       & + d(2(\alpha+\gamma)xy^2 + y^2z - 2\alpha y^3 )\\
        =& ax^2z + bxyz + dy^2z + \\
       & + (2(\beta+\gamma)a + (\alpha + \gamma -\beta)b)x^2y + (-2\beta a)x^3 + (2(\alpha+\gamma)d + (\beta + \gamma -\alpha)b)xy^2 +(-2\alpha d)y^3,
\end{align*}
so $(2(\beta+\gamma)a + (\alpha + \gamma -\beta)b) = 0$, $-2\beta a = 0$, $2(\alpha+\gamma)d + (\beta + \gamma -\alpha)b = 0$ and $-2\alpha d = 0$. Using Example \ref{examT1} (i), we can conclude that $Z(S)_2$ is given  as in Table \ref{tab-cenele} for Type T$_1$. 
\end{example}

\begin{theorem}
The list of $Z(S)_2$ is given in Table \ref{tab-cenele} for every $S \in \cA_{3,0}$. The table consists of the following information:   
\begin{enumerate}
\item{} Type of $S=k\<x, y, z\>/(h_1, h_2, h_3)$.  
\item{} Defining relations $h_1, h_2, h_3 \in k\<x, y, z\>_2$.
\item{} $Z(S)_2$ depending on the conditions. 
\end{enumerate} 
In the table, $\alpha, \beta, \gamma, a,b,c,d \in k $.  We omit the defining relations of $S$ if $Z(S)_2 = 0$. 
\end{theorem}

\begin{table}[h]
\scriptsize
\begin{threeparttable}
\caption{$Z(S)_2$ for $S \in \cA_{3,0}$.} \label{tab-cenele}
\begin{tabular}
{{|c|l|c|c|c|c|}} 
\hline
        Type          & \ \ \  Defining relations \ \ \ \     & Conditions & $Z(S)_2$   \\ \hline \hline
\multirow{5}{*}{P$_1$} & \multirow{4}{*}{ \makecell{$\left\{\begin{array}{l} \alpha xy-\beta yx\\\beta yz-\gamma zy\\ \gamma zx-\alpha xz\end{array}\right.$ } } & $\alpha=\beta=\gamma $ & $k[x,y,z]_2$  \\ \cline{3-4}
&  & $(\alpha\mathpunct{:}\beta\mathpunct{:}\gamma)= (1\mathpunct{:}1\mathpunct{:}-1) \in \mathbb{P}^2$ & $ax^2+by^2+cz^2+dxy$   \\ \cline{3-4} 
  &   &  $(\alpha\mathpunct{:}\beta\mathpunct{:}\gamma)= (1\mathpunct{:}-1\mathpunct{:}1) \in \mathbb{P}^2$ &  $ax^2+by^2+cz^2+dxz$   \\ \cline{3-4}
&  &$(\alpha\mathpunct{:}\beta\mathpunct{:}\gamma)= (-1\mathpunct{:}1\mathpunct{:}1) \in \mathbb{P}^2$ & $ax^2+by^2+cz^2+dyz$  \\ \cline{3-4}
& \ \ ($\alpha\beta\gamma\neq 0$)      & Otherwise & 0  \\ \hline
Other P  &    &  &   \\ \hline \hline
\multirow{8}{*}{S$_1$} & \multirow{8}{*}{ \makecell{$\left\{\begin{array}{l}  xy-\gamma  yx\\ yz-\alpha zy\\  zx-\beta xz \end{array}\right.$ } } &   
$(\alpha^2,\beta^2,\gamma^2) = (1,1,1)$ & $ax^2+by^2+cz^2$ \\ \cline{3-4} 
                  &                   &  $(\alpha^2,\beta^2,\gamma^2)=(1,1,n), n \neq 1$ & $z^2$  \\ \cline{3-4} 
                  &                   &  $(\alpha^2,\beta^2,\gamma^2)=(1,n,1), n \neq 1$ & $y^2$  \\ \cline{3-4} 
                  &                   &  $(\alpha^2,\beta^2,\gamma^2)=(n,1,1), n \neq 1$ & $x^2$ \\ \cline{3-4} 
                  &   &  $(\alpha,\beta,\gamma)=(1,m,m)$ & $yz$ \\ \cline{3-4} 
                  &                  & $(\alpha,\beta,\gamma)=(m,1,m)$ & $zx$   \\ \cline{3-4} 
                  &      \ \ ($\alpha\beta\gamma\neq 0,1$)                             & $(\alpha,\beta,\gamma)=(m,m,1)$ & $xy$   \\ \cline{3-4} 
                  &    &   Otherwise  & 0  \\ \hline
\multirow{3}{*}{S$_2$} & \multirow{3}{*}{\makecell{$\left\{\begin{array}{l}  zx-\alpha yz \\ xz-\beta zy \\ x^2+\alpha\beta y^2 \ \  (\alpha\beta\neq 0) \end{array}\right.$} } &   
                   $\alpha=\beta $ & $a(xy+yx)+bz^2$ \\ \cline{3-4}
                  &      & $\alpha=-\beta $ & $y^2$ \\ \cline{3-4}
                  & &$\alpha \neq \pm \beta$ & $0$  \\ \hline
     S$_3$       &     &  &   \\  \hline \hline
{\multirow{4}{*}{S$'_1$}} &{\multirow{3}{*}{$\left\{\begin{array}{l} xy - \beta yx \\ x^2 + yz - \alpha zy \\ zx - \beta xz  \end{array}\right.$}}&
$\beta = \pm 1, \alpha = -1$ & $a x^2 + b y^2 + c z^2 $ \\ \cline{3-4}
& & $\beta = \pm 1, \alpha \neq -1$ & $ x^2 $  \\  \cline{3-4} 
&  & $\beta \neq \pm1, \alpha = 1$ & $x^2 + (1-\beta^2)yz$  \\ \cline{3-4}
& \ \ $(\alpha \beta^2 \neq 0,1)$  & $\beta \neq \pm1, \alpha \neq 1$ &  $0$ \\ \hline
S$'_2$ & $xy - zx,  yx - xz, x^2 + y^2 + z^2$ &  & $ax^2+b(yz+zy) $   \\ \hline \hline
\multirow{5}{*}{T$_1$}      & \multirow{4}{*}{$\left\{\begin{array}{l} xy-yx\\ xz-zx-\beta x^2 +(\beta+\gamma)yx \\ yz-zy-\alpha y^2 +(\alpha+\gamma)xy\end{array}\right.$ }  &  $\alpha=\beta=0$   &  $(x-y)^2$       \\ \cline{3-4} 
&   &        $\alpha=0,  \beta = \gamma$      &     $y^2-xy$    \\ \cline{3-4}
&   &    $\beta=0, \alpha=\gamma$ &      $x^2-xy$  \\ \cline{3-4} 
&    &     $\gamma =0, \alpha=\beta$     &    $xy$  \\ \cline{3-4} 
&   \ \  ($\alpha+\beta+\gamma\neq 0$)                 & Otherwise  & 0                  \\ \hline 
\multirow{3}{*}{T$_2$} & \multirow{3}{*}{$\left\{\begin{array}{l} yz-zx-\alpha yx +(\alpha+\gamma)x^2 \\ xz-zy-\beta xy +(\beta+\gamma)y^2 \\ x^2-y^2 \ \ (\alpha+\beta+\gamma\neq 0) \end{array}\right.$ }  &\multirow{2}{*}{$\alpha+\beta = 0$}    & \multirow{2}{*}{$2x^2-(xy+yx)$}                  \\ 
   & & &  \\ \cline{3-4}
                       &       & $\alpha+\beta \neq 0$       & {$0$}   \\ \hline
T$_3$      &    &  &            \\ \hline \hline
{\multirow{3}{*}{T$'_1$}} &{\multirow{3}{*}{$\left\{\begin{array}{l} \alpha x^2 + \beta (\alpha + \beta)xy - xz + zx -(\alpha + \beta)zy \\2 \beta xy - \beta^2 y^2 + yz -zy  \\ xy - yx - \beta y^2  \ \ (\alpha + 2 \beta \neq 0) \end{array}\right.$}}& $\beta = 0$  & $y^2$   \\ \cline{3-4}
& & $\alpha = 0$  & $x^2-\beta yx-yz$   \\ \cline{3-4}
&  & \text{Otherwise} & 0   \\  \hline \hline
\multirow{2}{*}{NC$_1$} & \multirow{2}{*}{$\left\{\begin{array}{l}  xy-\alpha yx, yz - \alpha zy + x^2 \\ zx - \alpha xz + y^2  \ \ (\alpha^3\neq 0,1) \end{array}\right.$   } & $\alpha=-1$ &  $ax^2+by^2+cz^2$  \\ \cline{3-4}
&  & $\alpha \neq -1$ & $0$  \\ \hline
NC$_2$ & $xz-2yx+zy, zx-2xy+yz, y^2+x^2$ & & $a(xy+yx)+bz^2$   \\ \hline \hline
CC &&& \\ \hline \hline
{\multirow{3}{*}{WL$_1$}} &{\multirow{3}{*}{$\left\{\begin{array}{l} zy - yz + (1+\gamma) y^2 \\ \alpha xz - \gamma yx - zx \\ \alpha xy - yx  \ \ (\alpha \neq 0,1) \end{array}\right.$}}& $\gamma = 0, \alpha = -1$ & $ x^2 $  \\ \cline{3-4}
& & $\gamma = -1, \alpha = -1$ & $ y^2 $  \\  \cline{3-4}
&  &Otherwise & $ 0 $  \\ \hline
{\multirow{4}{*}{WL$_2$}} & {\multirow{4}{*}{$\left\{\begin{array}{l} xy - yx \\  xz - \gamma yx - zx \\ zy - yz + (1+\gamma) y^2 \end{array}\right.$}} &  $\gamma = 0$ & $ x^2 $   \\ \cline{3-4}
& & $\gamma = -1$ & $ y^2 $  \\ \cline{3-4}
& & $\gamma = -1/2$ & $ xy $\\ \cline{3-4}
& & Otherwise & $ 0 $  \\ \hline
WL$_3$ & & & $  $  \\ \hline \hline
\multirow{2}{*}{TL$_1$} & \multirow{2}{*}{$\left\{\begin{array}{l}  \alpha^{-1}zy-\alpha yz+x^2, xz-\alpha^{-1}zx\\  xy-\alpha yx  \ \ (\alpha \neq 0)    \end{array}\right.$ } & $\alpha^2= 1$ & $x^2$           \\ \cline{3-4}
                        &                   & $\alpha^2\neq 1$     & $0$    \\ \hline
Other TL                   &    &    &       \\ \hline \hline
\multirow{3}{*}{EC-A} & \multirow{2}{*}{$\left\{\begin{array}{l}  \alpha xy + \beta yx + \gamma z^2, \alpha yz + \beta zy + \gamma x^2\\ \alpha zx + \beta xz + \gamma y^2
\end{array}\right.$ } & \multirow{2}{*}{$\alpha = \beta $} & \multirow{2}{*}{$ax^2+by^2+cz^2$}     \\  
                        &  &  &   \\ \cline{3-4}
                        & \ $(\alpha^3 + \beta^3 + \gamma^3)^3 \neq (3\alpha\beta \gamma)^3, \alpha \beta \gamma \neq 0$ &$\alpha \neq \beta$ & $0$  \\ \hline
\multirow{2}{*}{EC-B}       &  \multirow{2}{*}{$\left\{\begin{array}{l}  xz + zy + \alpha yx, zx + yz + \alpha xy \\ y^2 + x^2 + \alpha z^2 \ \ (\alpha^3 \neq 0, 1, -8)
\end{array}\right.$ }  &    &   \multirow{2}{*}{$a(xy+yx) + bz^2$}  \\ 
   & & &  \\ \hline
Other EC   & & & \\ \hline 
\end{tabular}
\end{threeparttable}
\end{table}

The reader may compare Table \ref{tab-cenele} to the classification list of $S \in \cA_{3,0}$ in \cite{IM2, Ma}.

\clearpage

\subsection{Geometric pairs}

Geometric pairs are important ingredients in noncommutative algebraic geometry.  In fact, they are essential to classify $3$-dimensional quantum polynomial algebras \cite{ATV}.  By \cite{HMM}, geometric pairs are defined for noncommutative conics $A$ in Calabi-Yau quantum projective planes, and they determine $C(A)$.
For example, we have the following lemma which will be used to prove Theorem \ref{thm.Ta2} in the next subsection.  

\begin{lemma}[{\cite[Theorem 4.14]{HMM}}]
If $S=k\<x, y, z\>/(yz+zy+\a x^2, zx+xz+\a y^2, xy+yx+\a z^2)\in \cA_{3, 0}$ where $\a^3\neq 0, 1, -8$ (Type EC-A), and $0\neq f\in Z(S)_2$ so that $A=S/(f)\in \cA_{3, 1}^c$, then the exactly one of the following occurs: 
\begin{enumerate}
\item{} $\#(E_A)=2$ and $C(A)\cong k[u]/(u^3)\times k$. 
\item{} $\#(E_A)=4$ and $C(A)\cong k[u]/(u^2)\times k^2$. 
\item{} $\#(E_A)=6$ and $C(A)\cong k^4$. 
\end{enumerate}
\end{lemma} 

It is then natural to compute geometric pairs for noncommutative conics without Calabi-Yau condition.  By the classification result in Theorem \ref{con.main} and the following proposition, we are able to classify geometric pairs of noncommutative central conics, however, unlike in the Calabi-Yau case, we will see below that geoemtric pairs do not determine $C(A)$ in general.  

\begin{proposition}[\cite{HWY}]
For $A \in \cA_{3,1}$, $A$ satisfies (G1) condition. 
\end{proposition}

Let $(E_A,\sigma_A), (E_{A'}, \sigma_{A'})$ be geometric pairs of $A , A'\in \cA_{3,1}$ respectively.  We say that two pairs are equivalent, denote by $(E_A,\sigma_A) \sim (E_{A'}, \sigma_{A'})$,  if there is a sequence of automorphisms $\tau = \{ \tau_n \}$ of $\mathbb{P}^2$ for $n \in \mathbb{Z}$, each of which sends $E_A$ isomorphically onto $E_{A'}$ (as varieties) such that the diagram
$$
\xymatrix{
E_A \ar[r]^-{\tau_n} \ar[d]_-{\sigma_A} & E_{A'} \ar[d]^-{\sigma_{A'}}\\
E \ar[r]^-{\tau_{n+1}} & E'
}
$$
commutes for every $n \in \mathbb{Z}$.  

Since for $A , A'\in \cA_{3,1}$, $\GrMod A \cong \GrMod A'$ implies that $(E_A,\sigma_A)\sim (E_{A'} , \sigma_{A'})$ by 
\cite[Theorem 4.7]{M2} (see also \cite[Lemma 2.5]{MU2}), and $\cA_{3,1}/\sim$ is classified (Theorem \ref{con.main}), we would like to calculate the geometric pair for every representative in $\cA^c_{3,1}/\sim$, i.e., the algebras in Table \ref{intro.tab.alg2}. 

For an algebra $A = S/(f) \in \cA_{3,1}$ where $S = k\< x,y,z\>/(f_1,f_2,f_3) \in \cA_{3,0}$, there are two steps to calculate its geometric pair $(E_A, \sigma_A)$.

{\bf Step 1:} Determine the geometric pair $(E,\sigma)$ of $S$ (see \cite{ATV, HMM}). 

{\bf Step 2:} Let $K \in M_{3 \times 4}(k\<x, y, z\>_1)$ such that 
$$
(x,y,z)K = (f_1, f_2, f_3, f).
$$
Then the geometric pair $(E_A, \s_A)$ of $A$ is given by 
$$E_A = \mathcal{V} (\{\text{3-minors of }K\}) \subset \mathbb{P}^2, \; \sigma_A = \sigma|_{E_A}.$$


\begin{example}
If  
$$
\mathcal{S} = k \< x,y,z \> /(yz+zy,zx + xz, xy+yx) \in \cA_{3,0},
$$
then the geometric pair of $\mathcal{S}$ is $(E = \mathcal{V}(xyz), \sigma)$, where $\sigma$ is defined by
$$
\sigma: 
\begin{cases} 
(a\mathpunct{:}b\mathpunct{:}0) \mapsto (a\mathpunct{:}-b\mathpunct{:}0) \\
(a\mathpunct{:}0\mathpunct{:}c) \mapsto (a\mathpunct{:}0\mathpunct{:}-c) \\
(0\mathpunct{:}b\mathpunct{:}c) \mapsto (0\mathpunct{:}b\mathpunct{:}-c).
\end{cases}
$$
If $f = x^2 \in RZ(\mathcal{S})_2$, and $A = \mathcal{S}/(f)$, then 
$$
K = 
\begin{pmatrix} 
0 & z & y & x \\ 
z & 0 & x & 0 \\
y & x & 0 & 0 
\end{pmatrix}, 
$$
so $E_A = \mathcal{V}(2xyz, x^2 z, -x^2y, -x^3) = \mathcal{V}(x)$ (as varieties), and 
$\sigma_A : E_A \to E_A$; $(0\mathpunct{:}b\mathpunct{:}c) \mapsto (0\mathpunct{:}b\mathpunct{:}-c)$. 
\end{example}

\begin{proposition}
For every $A \in \cA^c_{3,1}/\sim$, its geometric pair $(E_A, \sigma_A)$ is listed in the Table \ref{tab-geop}.
Moreover, the geometric paris of $k[x,y,z]/(x^2)$, $\mathcal{S}_\lambda/(x^2)$, $\mathcal{W}/(x^2)$ are all equivalent.
\end{proposition}

\begin{proof}
We show the moreover part.  If
$$
A_1 = k[x,y,z]/(x^2), \ A_\l = \mathcal{S}_\lambda/(x^2), \ A_2 = \mathcal{W}/(x^2), 
$$
and $(E_1, \sigma_1)$, $(E_\l, \sigma_\l)$, $(E_2, \sigma_2)$ is the corresponding geometric pairs, then $E_1 = E_\l = E_2 = \mathcal{V}(x)$.  If we identify $\mathcal{V}(x)$ with $\mathbb{P}^1$, then their automorphism can be defined as follows:
\begin{align*}
\sigma_{1}:& \mathbb{P}^1 \to \mathbb{P}^1; \ (a\mathpunct{:}b) \mapsto (a\mathpunct{:}b), \\
\sigma_{\l}:& \mathbb{P}^1 \to \mathbb{P}^1; \ (a\mathpunct{:}b) \mapsto (a\mathpunct{:}\l b), \\
\sigma_{2}:& \mathbb{P}^1 \to \mathbb{P}^1; \ (a\mathpunct{:}b) \mapsto (a\mathpunct{:}a+b).
\end{align*} 
These $3$ geometric pairs are exactly geometric pairs of $k[x,y]$, $k_{\l}[x,y]$ and $k_J[x,y]$ respectively. Since $2$-dimensional quantum polynomial algebras have equivalent graded module categories, 
their geometric pairs are equivalent. 
\end{proof}


\begin{remark} In the above notation, since 
$$C(A_1)\cong k_{-1}[u, v]/(u^2, v^2), \; C(A_\l)\cong k_{-\l}[u, v]/(u^2, v^2), \; C(A_2)\cong k_{-1}[u, v]/(u^2+uv, v^2) $$
(see Theorem \ref{thm.Ta2}), geoemtric pairs do not determine $C(A)$. 
\end{remark}

\begin{center}
\begin{table}[htbp]
\caption{Geometric pairs $(E_A,\sigma_A)$ of $A \in \cA^c_{3,1} /\sim$.} \label{tab-geop}
\linespread{1.4}\selectfont 
\small
\begin{tabular*}{0.95\linewidth} {@{\extracolsep{\fill}}cccccc}
\hline
$A$ &   $E_A$ & $\sigma_A$  \\ \hline \hline
$k[x,y,z]/(x^2)$ &    a line  &  identity \\ \hline
$k[x,y,z]/(x^2 + y^2)$ &  two crossing lines & identity \\ \hline
$k[x,y,z]/(x^2 + y^2 + z^2)$ &  a smooth conic &  identity \\ \hline
$\mathcal{S}_\l/(x^2)$ &  a line $\mathcal{V}(x)$ &  $(0\mathpunct{:}b\mathpunct{:}c) \mapsto (0\mathpunct{:}b\mathpunct{:}\lambda c)$ \\ \hline
$\mathcal{W}/(x^2)$  &   a line $\mathcal{V}(x)$ &  $(0\mathpunct{:}b\mathpunct{:}c) \mapsto (0\mathpunct{:}b\mathpunct{:}b+c)$ \\ \hline
$\mathcal{N}/(x^2)$  &   one point &  identity \\ \hline
$\mathcal{N}/(3x^2 + 3y^2 + 4z^2)$  &   2 points &  switch two points \\ \hline
$\mathcal{S}/(x^2 + y^2)$  &  \begin{tabular}{c}3 points in \\ general position \end{tabular}&  \begin{tabular}{c}fix one point, switch \\ the other two points\end{tabular} \\ \hline
$\mathcal{N}/(x^2 + y^2 - 4z^2)$  &    \begin{tabular}{c}4 points with \\ 3 points collinear \end{tabular}&  
\begin{tikzpicture}[x=10, y=10, baseline=(current bounding box)]
\draw[white] (0, -1) -- (0, 2.5) ;
\draw[dashed] (-1.3, 0) -- (4, 0) ; 
\fill (-1, 0) circle (1.2pt) ;
\fill (0, 2) circle (1.2pt) ;
\fill (3/2, 0) circle (1.2pt) ;
\fill (3, 0) circle (1.2pt) ;
\draw[<->]
      (-1+0.08, -0.15-0.12)
      to[out=-30, in=-150] (3/2-0.1, -0.25);
      \draw[<->]
      (0.15+0.13, 2)
      to[out=-10, in=130] (3-0.03, 0.13+0.15);
\end{tikzpicture} \\ \hline
$\mathcal{S}/(x^2 + y^2 + z^2)$  &   \begin{tabular}{c}6 points in a quadrilateral \\ configuration\end{tabular} &  
\begin{tikzpicture}[x=11, y=11, baseline=(current bounding box)]
\draw[dashed] (-3.5, 0) -- (0.8, 0) ; 
\draw[dashed] (-2.2, -3.3) -- (1/2, 3/4) ; 
\draw[dashed] (-1.4, 0.6) -- (-2.08, -3.5) ; 
\draw[dashed] (-3.5, 0.375) -- (-0.2, -2.1) ; 
\fill (0, 0) circle (1.2pt) ; 
\fill (-3, 0) circle (1.2pt) ;
\fill (-1.5, 0) circle (1.2pt) ;
\fill (-1, -3/2) circle (1.2pt) ;
\fill (-2, -3) circle (1.2pt) ;
\fill (-5/3, -1) circle (1.2pt) ;
\draw[<->]
      (-0.15, -0.11)
      to[out=-150, in=30] (-5/3+0.16, -0.92);
\draw[<->]
      (-3.08, -0.18)
      to[out=-110, in=150] (-2.17, -2.9);
\draw[<->]
      (-1.4, -0.15)
      to[out=-60, in=100] (-1, -1.32);
\end{tikzpicture} \\ \hline
\end{tabular*}
\end{table}
\end{center}

\clearpage

\subsection{Regular normal (central) elements in degree 1}  \label{appii-tabs}

For $A\in \cA_{3, 1}^c$, it is rather annoying to calculate $C(A):=A^![(f^!)^{-1}]_0$ from the definition.  Recall that $A\in \cA_{3, 1}^w$ if and only if $RN(A^!)_1\neq \emptyset$, and, in this case, $C(A)\cong  \sD_w(A^!) \cong (A^!)^{\nu}/(w-1)$ for $w\in RZ(A^!)_1$ by Lemma \ref{lem.cwt} and Lemma \ref{lem.633} where $\nu$ is the normalizing automorphism of $w$.
Moreover, $A\in \cA_{3, 1}^z$ if and only if $RZ(A^!)_1\neq \emptyset$, and, in this case, $C(A)\cong A^!/(z-1)$ for $z\in RZ(A^!)_1$ by Lemma \ref{lem.zin}, 
so it is important to compute $RN(A^!)_1$ and $RZ(A^!)_1$. 

The following lemma is technical, but very useful to find a regular normal element in $A^!_1$ for $A\in \cA_{3, 1}$ (see Theorem \ref{thm.Ta2}). 
 
\begin{lemma} \label{lem.nore} 
Let $A\in \cA_{n, 1}$ where $n\leq 4$ and $w\in A^!_1$ a normal element. Then $w$ is regular if and only if $(A^!/(w))^!\in \cA_{n-1, 0}$.  
\end{lemma} 

\begin{proof} Since $A$ is Koszul, $H_{A^!}(t)=(1+t)^{n-1}/(1-t)$.  If $w$ is regular, then $A^!/(w)$ is an Koszul AS-Gorenstein algebra \cite[Theorem 1.5]{ST}.  Since $A^!/(w)$ is finite dimensional over $k$, $A^!/(w)$ is a Frobenius Koszul algebra such that $H_{A^!/(w)}(t)=(1-t)H_{A^!}(t)=(1+t)^{n-1}$. Since $n\leq 4$,  $(A^!/(w))^!$ is noetherian by \cite{ATV}, so $(A^!/(w))^!\in \cA_{n-1, 0}$ by Lemma \ref{lem.FAS} and Lemma \ref{lem.smith}. 

Conversely, if $(A^!/(w))^!\in \cA_{n-1, 0}$, then $(A^!/(w))^!$ is Koszul, so $H_{A^!/(w)}(t)=(1+t)^{n-1}=
H_{A^!}(t)(1-t)$, so $w$ is regular by Lemma \ref{lem.Tak}.  
\end{proof}

\begin{example}[The algebra $J_7$ in Table \ref{tab-caj}]  For $A\in \cA_{3, 1}$ and a normal element $w\in A^!_1$,  $w$ is regular if and only if $(A^!/(w))^!\in \cA_{2, 0}$ by the above lemma.  This is particularly useful since $A\in \cA_{2, 0}$ if and only if 
$$
A\cong k\<x, y\>/(ax^2+bxy+cyx+dy^2), \; ad\neq bc
$$  
(see Example \ref{exm-2dimqpa}).  
For example, if 
$$
	{ J_7}:=k\<x, y, z\>/(2xy-zx+yz, 2yx-xz+zy, x^2+y^2, xy+yx+z^2) \in \cA_{3, 1}, 
	$$
then ${ J_7}^!=k\<x, y, z\>/(yz+zx, zy+xz, xy-yx+2zx-2xz, x^2-y^2, z^2-xy-2zx)$ (see Theorem \ref{thm.Ta2}).  It is easy to see that $z\in ({ J_7}^!)_1$ is normal by Lemma \ref{lem.reno}.  Since ${ J_7}^!/(z)\cong k[x, y]/(x^2-y^2, xy)\in \cA_{2, 2}$, $({ J_7}^!/(z))^!\cong k\<x, y\>/(x^2+y^2)\in \cA_{2, 0}$, so 
$z\in ({ J_7}^!)_1$ is regular by Lemma \ref{lem.nore}. 
\end{example}

\begin{theorem} \label{thm.Ta2} For each $E \in \cF_2$, every $A\in \cA_{3, 1}^c$ such that $C(A)\cong E$ is isomorphic to one of the algebras in Table \ref{tab-caa}\,--\,Table\ref{tab-cak}.
The tables consist of the following information:   
\begin{enumerate}
\item{} Name of $A \in \cA^c_{3,1}$.
\item{} Type of $S \in \cA_{3,0}$.  
\item{} Sequences of $F = (f_1, f_2, f_3, f_4)$ in $k\<x, y, z\>$ such that $A \cong k\<x, y, z\>/I_F$.
\item{}  Sequences of $G = (g_1, g_2, g_3, g_4, g_5)$ in $k\<x, y, z\>$ such that $A^!\cong k\<x, y, z\>/I_G$.
\item{} Examples of elements in $RN(A^!)_1$. ($RN(A^!)_1=\emptyset$ or not.) 
\item{} Examples of elements in $RZ(A^!)_1$.  ($RZ(A^!)_1=\emptyset$ or not.) 
\end{enumerate} 
In these tables, $\alpha, \beta, \gamma, a,b,c \in k $.
\end{theorem} 

\begin{remark} 
\begin{enumerate}
\item{} For $X, Y\in \cA_{3, 1}^c$ in the tables below (the names of $A\in \cA_{3, 1}^c$), $X\neq Y$ implies $X_i\not \cong Y_j$ for any $i, j$, but $i\neq j$ does not imply $X_i\not \cong X_j$. 
\item{}  For $A\in \cA_{3, 1}^c$, if $A^!$ is commutative, then $A^!\in \cB_{3, 2}$ by Lemma \ref{lem.48}.
Since $RZ(A^!)_1\neq \emptyset$ by Lemma \ref{lem.wbwd}, $A\in \cA_{3, 1}^z$. 
\item{} For $A, A'\in \cA_{3, 1}^c$ such that $A^!, {A'}^!$ are both commutative, $A\cong A'$ if and only if $C(A)\cong C(A')$ by Theorem \ref{thm.43}.
\item{} In Table \ref{tab-cai}\,--\,Table \ref{tab-cak}, we write the symbol $*$ to indicate that, for $X \in \cA_{3, 1}^c$, the set $RN(X^!)_1$ or $RZ(X^!)_1$ is not empty, but we fail to find a concrete element in the set. 
\end{enumerate}
\end{remark} 

\begin{center}
	\begin{table}[ht]
		\begin{threeparttable}
			\small
			\centering
			\caption{$A \in \cA^c_{3,1}$ s.t. $C(A) \cong M_2(k)$.} \label{tab-caa}
			\begin{tabular*}{0.95\linewidth} {@{\extracolsep{\fill}}cccccc}
				\hline
				$A$ & $S$ & $F$ & $G$ & $RN(A^!)_1$ & $RZ(A^!)_1$  \\ \hline \hline
				$A_1$ & P$_1$ &  
				$\begin{array}{c}xy-yx\\yz-zy\\zx-xz\\x^2+y^2+z^2\end{array}$ & 
				$\begin{array}{c}xy+yx\\yz+zy\\zx+xz\\x^2-y^2\\x^2-z^2\end{array}$&  
				$\begin{array}{c}{x} \end{array}$ &  $\emptyset$  \\ \hline
				$A_2$ & P$_1$ &  
				$\begin{array}{c}xy-yx\\yz+zy\\zx+xz\\x^2+y^2+z^2\end{array}$ & 
				$\begin{array}{c}xy+yx\\yz-zy\\zx-xz\\x^2-y^2\\x^2-z^2\end{array}$&  
				${z}$ &  $z$ \\ \hline
				$\begin{array}{c}
					A_3(\alpha)\\
					(\alpha\neq0,\pm1)
				\end{array}$ & S$'_1$ & 
				$\begin{array}{c}xy-\alpha yx\\yz-zy+(1-\alpha^2)x^2\\zx-\alpha xz\\x^2+yz\end{array}$ & 
				$\begin{array}{c}\alpha xy+yx\\yz+\alpha^2 zy-x^2\\\alpha zx+xz\\y^2\\z^2\end{array}$&  
				$x$ & $\emptyset$ \\ \hline 
				$A_4$ & T$'_1$ &  
				$\begin{array}{c}xy-yx-y^2\\yz-zy-2xy\\zx-xz-yz\\x^2+yz\end{array}$ & 
				$\begin{array}{c}xy+yx-2zy\\yz+zy+zx-x^2\\zx+xz\\y^2-yx\\z^2\end{array}$&  
				$ \emptyset $ &$\emptyset$ \\ \hline
			\end{tabular*}
		\end{threeparttable}
	\end{table}
\end{center}

\begin{table}[htp]
	\begin{threeparttable}
		\small
		\centering
		\caption{$A \in \cA^c_{3,1}$ s.t. $C(A) \cong k_{-1} [u,v]/(u^2-1,v^2)$.}  \label{tab-cab}
		\begin{tabular*}{0.95\linewidth} {@{\extracolsep{\fill}}cccccc}
			\hline
			$A$ & $S$ & $F$ & $G$ & $RN(A^!)_1$ & $RZ(A^!)_1$ \\ \hline \hline
			$B_1$ & P$_1$ &  
			$\begin{array}{c}xy-yx\\yz-zy\\zx-xz\\x^2+y^2\end{array}$ & 
			$\begin{array}{c}xy+yx\\yz+zy\\zx+xz\\x^2-y^2\\z^2\end{array}$ &  
			$x$ & $\emptyset$ \\ \hline
			$B_2$ & P$_1$ &  
			$\begin{array}{c}xy-yx\\yz+zy\\zx+xz\\x^2+y^2\end{array}$ & 
			$\begin{array}{c}xy+yx\\yz-zy\\zx-xz\\x^2-y^2\\z^2\end{array}$ &  
			$x$ &  $\emptyset$ \\ \hline
			$B_3$ & P$_1$ &  
			$\begin{array}{c}xy+yx\\yz+zy\\zx-xz\\x^2+y^2\end{array}$ & 
			$\begin{array}{c}xy-yx\\yz-zy\\zx+xz\\x^2-y^2\\z^2\end{array}$ &  
			$y$ &  $y$ \\ \hline
			$\begin{array}{c}
				B_4(\alpha)\\
				(\alpha\neq0,\pm1)
			\end{array}$ & S$_1$ & 
			$\begin{array}{c}xy-yx\\yz-\alpha zy\\zx-\alpha xz\\xy\end{array}$ & 
			$\begin{array}{c}\alpha yz+zy\\\alpha zx+xz\\x^2\\y^2\\z^2\end{array}$ &  
			$\emptyset$ & $\emptyset$ \\ \hline 
			$B_5$ & S$_2$ & 
			$\begin{array}{c}x^2-y^2\\yz+zx\\zy-xz\\x^2\end{array}$ & 
			$\begin{array}{c}xy\\yx\\yz-zx\\zy+xz\\z^2\end{array}$ &  
			$x -\sqrt{-1}y$ & $\emptyset$ \\ \hline
			$B_6$ & T$_1$ & 
			$\begin{array}{c}xy-yx\\yz-zy-y^2+xy\\zx-xz+x^2-yx\\xy\end{array}$ & 
			$\begin{array}{c}yz+zy\\zx+xz\\x^2-zx\\y^2+yz\\z^2\end{array}$ &  
			$\begin{array}{c}x- y  \end{array}$ & $\emptyset$ \\ \hline
			$B_7$ & WL$_2$ & 
			$\begin{array}{c}xy-yx\\yz-zy-y^2\\zx-xz-yx\\xy\end{array}$ & 
			$\begin{array}{c}yz+zy\\zx+xz\\x^2\\y^2+yz\\z^2\end{array}$ &  
			$\emptyset$ & $\emptyset$ \\ \hline
		\end{tabular*}
	\end{threeparttable}
\end{table}

\begin{table}[htp]
	\begin{threeparttable}
		\small
		\centering
		\caption{$A \in \cA^c_{3,1}$ s.t. $C(A) \cong k_{-1}[u,v]/(u^2+uv,v^2)$.}  \label{tab-cac}
		\begin{tabular*}{0.95\linewidth} {@{\extracolsep{\fill}}cccccc}
			\hline
			$A$ & $S$ & $F$ & $G$ & $RN(A^!)_1$ & $RZ(A^!)_1$ \\ \hline \hline
			$C_1$ & T$_1$ &  
			$\begin{array}{c}xy-yx\\yz-zy+y^2+xy\\zx-xz\\x^2\end{array}$ & 
			$\begin{array}{c}xy+yx-yz\\yz+zy\\zx+xz\\y^2-yz\\z^2\end{array}$ &  
			$\begin{array}{c}-2x+y\end{array}$ & $\emptyset$ \\ \hline
			$\begin{array}{c}
				C_2(\alpha)\\
				(\alpha\in k)
			\end{array}$ & T$_2$ &  
			$\begin{array}{c}xy+yx\\yz-zy-x^2-y^2+\alpha xy\\zx+xz-\alpha x^2\\x^2\end{array}$ & 
			$\begin{array}{c}xy-yx-\alpha yz\\yz+zy\\zx-xz\\y^2+yz\\z^2\end{array}$ &  
			$x+\dfrac{\alpha}{2}z$ &  $x+\dfrac{\alpha}{2}z$ \\ \hline
			$C_3$ & T$'_1$ & 
			$\begin{array}{c}xy-yx\\yz-zy-y^2+zx\\zx-xz\\x^2\end{array}$ & 
			$\begin{array}{c}xy+yx\\yz+zy\\zx+xz-yz\\y^2+yz\\z^2\end{array}$ &  
			$\emptyset$ & $\emptyset$ \\ \hline 
			$C_4$ & WL$_1$ & 
			$\begin{array}{c}xy+yx\\yz-zy-y^2\\zx+xz\\x^2\end{array}$ & 
			$\begin{array}{c}xy-yx\\yz+zy\\zx-xz\\y^2+yz\\z^2\end{array}$ &  
			$x$ & $x$ \\ \hline
			$C_5$ & WL$_2$ & 
			$\begin{array}{c}xy-yx\\yz-zy-y^2\\zx-xz\\x^2\end{array}$ & 
			$\begin{array}{c}xy+yx\\yz+zy\\zx+xz\\y^2+yz\\z^2\end{array}$ &  
			$\begin{array}{c}x\end{array}$ & $\emptyset$ \\ \hline
		\end{tabular*}
		\begin{tablenotes}
			\linespread{1}
			\item\hspace*{-\fontdimen2\font} $C_2(\alpha) \cong C_4$. 
		\end{tablenotes}
	\end{threeparttable}
\end{table}

\begin{table}[htp]
	\begin{threeparttable}
		\small
		\centering
		\caption{$A \in \cA^c_{3,1}$ s.t. $C(A) \cong k_{-1}[u,v]/(u^2,v^2)$.}  \label{tab-cad}
		\begin{tabular*}{0.95\linewidth} {@{\extracolsep{\fill}}cccccc}
			\hline
			$A$ & $S$ & $F$ & $G$ & $RN(A^!)_1$ & $RZ(A^!)_1$ \\ \hline \hline
			$D_1$ & P$_1$ &  
			$\begin{array}{c}xy-yx\\yz-zy\\zx-xz\\x^2\end{array}$ & 
			$\begin{array}{c}xy+yx\\yz+zy\\zx+xz\\y^2\\z^2\end{array}$ &  
			$\begin{array}{c}x \end{array}$ & $\emptyset$ \\ \hline
			$D_2$ & P$_1$ &  
			$\begin{array}{c}xy-yx\\yz+zy\\zx+xz\\x^2\end{array}$ & 
			$\begin{array}{c}xy+yx\\yz-zy\\zx-xz\\y^2\\z^2\end{array}$ &  
			$\begin{array}{c}x  \end{array}$ & $\emptyset$  \\ \hline
			$D_3$ & P$_1$ & 
			$\begin{array}{c}xy+yx\\yz-zy\\zx+xz\\x^2\end{array}$ & 
			$\begin{array}{c}xy-yx\\yz+zy\\zx-xz\\y^2\\z^2\end{array}$ &   
			$x$ &  $x$ \\ \hline
			$D_4$ & WL$_1$ & 
			$\begin{array}{c}xy+yx\\yz+zy-xy\\zx-xz\\x^2\end{array}$ & 
			$\begin{array}{c}xy-yx+yz\\yz-zy\\zx+xz\\y^2\\z^2\end{array}$ &  
			$\emptyset$ &$\emptyset$ \\ \hline 
			$D_5$ & WL$_2$ & 
			$\begin{array}{c}xy-yx\\yz-zy+xy\\zx-xz\\x^2\end{array}$ & 
			$\begin{array}{c}xy+yx-yz\\yz+zy\\zx+xz\\y^2\\z^2\end{array}$ &  
			$\emptyset$ & $\emptyset$\\ \hline
			$D_6$ & TL$_1$ & 
			$\begin{array}{c}xy-yx\\yz-zy-x^2\\zx-xz\\x^2\end{array}$ & 
			$\begin{array}{c}xy+yx\\yz+zy\\zx+xz\\y^2\\z^2\end{array}$ &  
			$\begin{array}{c}x \end{array}$ & $\emptyset$ \\ \hline
			$D_7$ & TL$_1$ & 
			$\begin{array}{c}xy+yx\\yz-zy+x^2\\zx+xz\\x^2\end{array}$ & 
			$\begin{array}{c}xy-yx\\yz+zy\\zx-xz\\y^2\\z^2\end{array}$ &  
			$x$ & $x$\\ \hline
		\end{tabular*}
		\begin{tablenotes}
			\item\hspace*{-\fontdimen2\font} $D_1 = D_6$, $D_3 = D_7$.
		\end{tablenotes}
	\end{threeparttable}
\end{table}

\begin{table}[ht]
	\begin{threeparttable}
		\small
		\centering
		\caption{$A \in \cA^c_{3,1}$ s.t. $C(A) \cong k_{-\lambda}[u,v]/(u^2,v^2)$ where $\lambda\neq0,\pm1$.}  \label{tab-cae}
		\begin{tabular*}{0.95\linewidth} {@{\extracolsep{\fill}}cccccc}
			\hline
			$A$ & $S$ & $F$ & $G$ & $RN(A^!)_1$ & $RZ(A^!)_1$ \\ \hline \hline
			$E_1(\lambda)$ & S$_1$ & 
			$\begin{array}{c}xy-yx\\yz-\lambda zy\\zx-xz\\x^2\end{array}$ & 
			$\begin{array}{c}xy+yx\\\lambda yz+zy\\zx+xz\\y^2\\z^2\end{array}$ &  
			$x$ & $\emptyset$  \\ \hline
			$E_2(\lambda)$ & S$_1$ &  
			$\begin{array}{c}xy-yx\\yz+\lambda zy\\zx+xz\\x^2\end{array}$ & 
			$\begin{array}{c}xy+yx\\\lambda yz-zy\\zx-xz\\y^2\\z^2\end{array}$ &  
			$x$ &  $\emptyset$ \\ \hline
			$E_3(\lambda)$ & S$_1$ & 
			$\begin{array}{c}xy+yx\\yz-\lambda zy\\zx+xz\\x^2\end{array}$ & 
			$\begin{array}{c}xy-yx\\\lambda yz+zy\\zx-xz\\y^2\\z^2\end{array}$ &  
			$x$ &  $x$ \\ \hline
			$E_4(\lambda)$ & S$_1$ & 
			$\begin{array}{c}xy+yx\\yz+\lambda zy\\zx-xz\\x^2\end{array}$ & 
			$\begin{array}{c}xy-yx\\\lambda yz-zy\\zx+xz\\y^2\\z^2\end{array}$ &  
			$x$ & $\emptyset$\\ \hline 
			$E_5(\lambda)$ & S$'_1$ & 
			$\begin{array}{c}xy+yx\\yz-\lambda zy+x^2\\zx+xz\\x^2\end{array}$ & 
			$\begin{array}{c}xy-yx\\\lambda yz+zy\\zx-xz\\y^2\\z^2\end{array}$ &  
			$x$ & $x$ \\ \hline
			$E_6(\lambda)$ & S$'_1$ & 
			$\begin{array}{c}xy-yx\\yz-\lambda zy+x^2\\zx-xz\\x^2\end{array}$ & 
			$\begin{array}{c}xy+yx\\\lambda yz+zy\\zx+xz\\y^2\\z^2\end{array}$ &  
			$x$ &$\emptyset$ \\ \hline
		\end{tabular*}
		\begin{tablenotes}
			\item\hspace*{-\fontdimen2\font} $E_1(\lambda) = E_6(\lambda)$, $E_3(\lambda) = E_5(\lambda)$,  $E_1(\lambda) \cong E_1(\lambda^{-1})$, $E_3(\lambda) \cong E_3(\lambda^{-1})$, $E_2(\lambda) \cong E_4(\lambda^{-1})$. 
		\end{tablenotes}
	\end{threeparttable}
\end{table}

\begin{table}[ht]
	\begin{threeparttable}
		\small
		\centering
		\caption{$A \in \cA^c_{3,1}$ s.t. $C(A) \cong k[u,v] /(u^2,v^2)$.}  \label{tab-caf}
		\begin{tabular*}{0.95\linewidth} {@{\extracolsep{\fill}}cccccc}
			\hline
			$A$ & $S$ & $F$ & $G$ & $RN(A^!)_1$ & $RZ(A^!)_1$ \\ \hline \hline
			$F_1$ & S$_1$ &
			$\begin{array}{c}xy+yx\\yz+zy\\zx+xz\\x^2\end{array}$ & 
			$\begin{array}{c}xy-yx\\yz-zy\\zx-xz\\y^2\\z^2\end{array}$ &  
			$x$& $x$ \\ \hline
			$F_2$ & S$_1$ &
			$\begin{array}{c}xy+yx\\yz-zy\\zx-xz\\x^2\end{array}$ & 
			$\begin{array}{c}xy-yx\\yz+zy\\zx+xz\\y^2\\z^2\end{array}$ &  
			$x$ &  $\emptyset$  \\ \hline
			$F_3$ & S$_1$ &
			$\begin{array}{c}xy-yx\\yz+zy\\zx-xz\\x^2\end{array}$ & 
			$\begin{array}{c}xy+yx\\yz-zy\\zx+xz\\y^2\\z^2\end{array}$ &  
			$x$ & $\emptyset$  \\ \hline
			$F_4$ & S$_2$ & 
			$\begin{array}{c}xy-zx\\yx-xz\\y^2+z^2\\x^2\end{array}$ & 
			$\begin{array}{c}xy+zx\\yx+xz\\y^2-z^2\\yz\\zy\end{array}$ &  
			$x$ &$\emptyset$ \\ \hline 
			$F_5$ & S$'_1$ & 
			$\begin{array}{c}xy+yx\\yz+zy+x^2\\zx+xz\\x^2\end{array}$ & 
			$\begin{array}{c}xy-yx\\yz-zy\\zx-xz\\y^2\\z^2\end{array}$ &  
			$x$& $x$ \\ \hline
			$F_6$ & S$'_1$ & 
			$\begin{array}{c}xy-yx\\yz+zy+x^2\\zx-xz\\x^2\end{array}$ & 
			$\begin{array}{c}xy+yx\\yz-zy\\zx+xz\\y^2\\z^2\end{array}$ &  
			$x$ & $\emptyset$\\ \hline
			$F_7$ & S$'_2$ & 
			$\begin{array}{c}xy-zx\\yx-xz\\x^2+y^2+z^2\\x^2\end{array}$ & 
			$\begin{array}{c}xy+zx\\yx+xz\\y^2-z^2\\yz\\zy\end{array}$ &  
			$x$ & $\emptyset$\\ \hline
		\end{tabular*}
		\begin{tablenotes}
			\item\hspace*{-\fontdimen2\font} $F_1 = F_5$, $ F_3= F_6$, $F_4 = F_7$.
		\end{tablenotes}
	\end{threeparttable}
\end{table}

\begin{table}[ht]
	\begin{threeparttable}
		\small
		\centering
		\caption{$A \in \cA^c_{3,1}$ s.t. $C(A) \cong (k[u]/(u^2))^2$.}  \label{tab-cag}
		\begin{tabular*}{0.95\linewidth} {@{\extracolsep{\fill}}cccccc}
			\hline
			$A$ & $S$ & $F$ & $G$ & $RN(A^!)_1$ & $RZ(A^!)_1$ \\ \hline \hline
			$G_1$ & S$_1$ &
			$\begin{array}{c}xy+yx\\yz+zy\\zx+xz\\x^2+y^2\end{array}$ & 
			$\begin{array}{c}xy-yx\\yz-zy\\zx-xz\\x^2-y^2\\z^2\end{array}$ &  
			$x$ & $x$ \\ \hline
			$G_2$ & S$_1$ &
			$\begin{array}{c}xy+yx\\yz-zy\\zx-xz\\x^2+y^2\end{array}$ & 
			$\begin{array}{c}xy-yx\\yz+zy\\zx+xz\\x^2-y^2\\z^2\end{array}$ &  
			$\begin{array}{c}x \end{array}$ &  $\emptyset$  \\ \hline
			$G_3$ & S$_1$ &
			$\begin{array}{c}xy-yx\\yz+zy\\zx-xz\\x^2+y^2\end{array}$ & 
			$\begin{array}{c}xy+yx\\yz-zy\\zx+xz\\x^2-y^2\\z^2\end{array}$ &  
			$\begin{array}{c}x \end{array}$ &  $\emptyset$ \\ \hline
			$G_4$ & S$_2$ & 
			$\begin{array}{c}x^2+y^2\\zx-yz\\xz-zy\\xy+yx\end{array}$ & 
			$\begin{array}{c}x^2-y^2\\zx+yz\\xz+zy\\xy-yx\\z^2\end{array}$ &  
			$\emptyset$ &$\emptyset$ \\ \hline 
			$G_5$ & S$'_1$ & 
			$\begin{array}{c}xy+yx\\yz+zy+x^2\\zx+xz\\x^2+y^2\end{array}$ & 
			$\begin{array}{c}xy-yx\\yz-zy\\zx-xz\\x^2-y^2-yz\\z^2\end{array}$ &  
			$y$ & $y$ \\ \hline
			$G_6$ & S$'_1$ & 
			$\begin{array}{c}xy-yx\\yz+zy+x^2\\zx-xz\\x^2+y^2\end{array}$ & 
			$\begin{array}{c}xy+yx\\yz-zy\\zx+xz\\x^2-y^2-yz\\z^2\end{array}$ &  
			$\begin{array}{c}y\end{array}$ &$\emptyset$ \\ \hline
			$\begin{array}{c}
				G_7(\alpha) \\
				( \alpha \neq 0)
			\end{array}$ & NC$_1$ & 
				$\begin{array}{c}xy+yx\\yz+zy+ x^2\\zx+xz+\alpha y^2\\x^2+y^2\end{array}$ & 
			$\begin{array}{c}xy-yx\\yz-zy\\zx-xz\\x^2-y^2-yz+\alpha zx\\z^2\end{array}$ &  
			$x$ & $x$\\ \hline
			$G_8$ & ${\rm NC_2}$ & 
			$\begin{array}{c}xy-yz-zx\\yx-xz-zy\\x^2+y^2\\xy+yx\end{array}$ & 
			$\begin{array}{c}xy-yx+zx-xz\\yz-zx\\zy-xz\\x^2-y^2\\z^2\end{array}$ &  
			$\emptyset$ &$\emptyset$ \\ \hline
		\end{tabular*}
		\begin{tablenotes}
			\item\hspace*{-\fontdimen2\font} $G_1 \cong G_5 \cong G_7(\alpha)$.
		\end{tablenotes}
	\end{threeparttable}
\end{table}

\begin{table}[ht]
	\begin{threeparttable}
		\small
		\centering
		\caption{$A \in \cA^c_{3,1}$ s.t. $C(A) \cong k[u]/(u^4)$.}  \label{tab-cah}
		\begin{tabular*}{0.95\linewidth} {@{\extracolsep{\fill}}cccccc}
			\hline
			$A$ & $S$ & $F$ & $G$ & $RN(A^!)_1$ & $RZ(A^!)_1$ \\ \hline \hline
			$H_1$ & S$'_1$ &
			$\begin{array}{c}xy+yx\\yz+zy\\zx+xz+y^2\\x^2\end{array}$ & 
			$\begin{array}{c}xy-yx\\yz-zy\\zx-xz\\y^2-zx\\z^2\end{array}$ &  
			$x$ & $x$ \\ \hline
			$H_2$ & S$'_1$ &
			$\begin{array}{c}xy-yx\\yz-zy\\zx+xz+y^2\\x^2\end{array}$ & 
			$\begin{array}{c}xy+yx\\yz+zy\\zx-xz\\y^2-zx\\z^2\end{array}$ &  
			$x$ & $\emptyset$  \\ \hline
			$H_3$ & NC$_1$ & 
			$\begin{array}{c}xy+yx\\yz+zy+x^2\\zx+xz+y^2\\x^2\end{array}$ & 
			$\begin{array}{c}xy-yx\\yz-zy\\zx-xz\\y^2-zx\\z^2\end{array}$ &  
			$x$ & $x$\\ \hline 
		\end{tabular*}
		\begin{tablenotes}
			\linespread{1}
			\item\hspace*{-\fontdimen2\font} $H_1 = H_3$.
		\end{tablenotes}
	\end{threeparttable}
\end{table}

\begin{table}[ht]
	\begin{threeparttable}
		\small 
		\centering
		\caption{$A \in \cA^c_{3,1}$ such that $C(A) \cong k^4$.}  \label{tab-cai}
		\begin{tabular*}{0.95\linewidth}{@{\extracolsep{\fill}}cccccc}
			\hline
			$A$ & $S$ & $F$ & $G$ & $RN(A^!)_1$ & $RZ(A^!)_1$ \\ \hline \hline
			$I_1$ & S$_1$ &
			$\begin{array}{c}xy+yx\\yz+zy\\zx+xz\\x^2+y^2+z^2\end{array}$ & 
			$\begin{array}{c}xy-yx\\yz-zy\\zx-xz\\x^2-y^2\\x^2-z^2\end{array}$ & 
			$x$ & $x$ \\ \hline
			$I_2$ & S$_1$ &
			$\begin{array}{c}xy+yx\\yz-zy\\zx-xz\\x^2+y^2+z^2\end{array}$ & 
			$\begin{array}{c}xy-yx\\yz+zy\\zx+xz\\x^2-y^2\\x^2-z^2\end{array}$ & 
			$x$ & $\emptyset$   \\ \hline
			$I_3$ & S$_2$ &
			$\begin{array}{c}x^2+y^2\\yz-zx\\xz-zy\\xy+yx+z^2\end{array}$ & 
			$\begin{array}{c}xy-yx\\yz+zx\\xz+zy\\x^2-y^2\\z^2-xy\end{array}$ & 
			$z$ &$\emptyset$ \\ \hline 
			$I_4$ & S$'_1$ & 
			$\begin{array}{c}xy+yx+z^2\\yz+zy\\zx+xz\\x^2+y^2\end{array}$ & 
			$\begin{array}{c}xy-yx\\yz-zy\\xz-zx\\x^2-y^2\\z^2-xy\end{array}$&  $y$ & $y$ \\ \hline
			$I_5$ & S$'_1$ & 
			$\begin{array}{c}xy+yx+z^2\\yz-zy\\zx-xz\\x^2+y^2\end{array}$ & 
			$\begin{array}{c}xy-yx\\yz+zy\\xz+zx\\x^2-y^2\\z^2-xy\end{array}$  & 
			$x$ &$\emptyset$ \\ \hline
			$\begin{array}{c}
					I_6(\alpha)\\
					(\alpha\neq0,\pm1)
			\end{array}$ & S$'_1$ & 
			$\begin{array}{c}xy+yx\\yz+zy+\alpha x^2\\zx+xz\\x^2+y^2+z^2\end{array}$ & 
			$\begin{array}{c}xy-yx\\yz-zy\\zx-xz\\x^2-z^2-  \alpha yz\\y^2-z^2\end{array}$ & 
			$z$ & $z$  \\ \hline
			$\begin{array}{c}
					I_7(\alpha)\\
					(\alpha\neq0,\pm1)
			\end{array}$ & S$'_1$ & 
			$\begin{array}{c}xy-yx\\yz+zy+ \alpha x^2\\zx-xz\\x^2+y^2+z^2\end{array}$ & 
			$\begin{array}{c}xy+yx\\yz-zy\\zx+xz\\x^2-z^2- \alpha yz\\y^2-z^2\end{array}$ & 
			$y$ & $\emptyset$ \\ \hline
			$\begin{array}{c}
					I_8(\alpha)\\
					(\alpha\neq \pm1)
			\end{array}$ & S$'_2$ & 
			$\begin{array}{c}xy-zx\\y^2+z^2+x^2\\yx-xz\\ \alpha x^2+yz+zy\end{array}$ & 
			$\begin{array}{c}xy+zx\\yz-zy\\yx+xz\\x^2-z^2- \alpha yz\\y^2-z^2\end{array}$ & 
			$x$ & $\emptyset$ \\ \hline
			
		\end{tabular*}
	\end{threeparttable}
\end{table}
\begin{table}[ht]
	\begin{threeparttable}
		\small
		\centering
		\caption*{} 
		\begin{tabular*}{0.95\linewidth}{@{\extracolsep{\fill}}cccccc}
			\hline
			$A$ & $S$ & $F$ & $G$ & $RN(A^!)_1$ & $RZ(A^!)_1$ \\ \hline \hline
			$I_9$ & NC$_1$ & 
			$\begin{array}{c}xy+yx+z^2\\yz+zy\\zx+xz+y^2\\x^2\end{array}$ & 
			$\begin{array}{c}xy-yx\\yz-zy\\xz-zx\\y^2-zx\\z^2-xy\end{array}$ & 
			$x$ & $x$ \\ \hline
			$\begin{array}{c}
					I_{10}(\alpha)\\
					(\alpha\neq 0, \pm\frac{16\sqrt{-3}}{9})
			\end{array}$ & NC$_1$ & 
			$\begin{array}{c}xy+yx+\alpha z^2\\yz+zy\\zx+xz+y^2\\x^2+y^2\end{array}$ & 
			$\begin{array}{c}xy-yx\\yz-zy\\xz-zx\\x^2-y^2+zx\\z^2-\alpha xy\end{array}$& 
			$x$ & $x$ \\ \hline
			$\begin{array}{c}
					I_{11}(\alpha, \beta)\\
					(\delta(\alpha,\beta)\neq0)
			\end{array}$ & NC$_1$ & 
			$\begin{array}{c}xy+yx\\yz+zy+\tfrac{\sqrt{\beta}}{\alpha}x^2\\zx+xz+\tfrac{\sqrt{\alpha}}{\beta}y^2\\x^2+y^2+z^2\end{array}$ & 
			$\begin{array}{c}xy-yx\\yz-zy\\zx-xz\\x^2-z^2- \tfrac{\sqrt{\beta}}{\alpha} yz\\y^2-z^2-\tfrac{\sqrt{\alpha}}{\beta}zx\end{array}$ & 
			$z$& $z$\\ \hline
			$I_{12}$ & NC$_2$ & 
			$\begin{array}{c}xy+ zx- yz\\yx+ xz-zy\\y^2+z^2\\x^2\end{array}$ & 
			$\begin{array}{c}xy-zx\\yz+zx\\zy+xz\\yx+zy\\y^2-z^2\end{array}$& 
			$x$ & $\emptyset$\\ \hline
			$\begin{array}{c}
					I_{13}(\alpha)\\
					(\alpha^2\neq0,-1,3)
			\end{array}$ & NC$_2$ & 
			$\begin{array}{c}2xy-\alpha zx-\alpha yz\\2yx-\alpha xz-\alpha zy\\x^2+y^2\\xy+yx+z^2\end{array}$ & 
			$\begin{array}{c}  \alpha xy-\alpha yx+2zx-2xz\\yz-zx\\zy-xz\\x^2-y^2\\ \alpha z^2-\alpha xy-2zx\end{array}$ & 
			$z$ & $\emptyset$ \\ \hline
			$\begin{array}{c}
					I_{14}(\alpha,a,b)\\
					(\alpha^3\neq0,1,-8, \\
					\#(E_A)=6)
			\end{array}$ & EC-A & 
			$\begin{array}{c}xy+yx+\alpha z^2\\yz+zy+\alpha x^2\\zx+xz+\alpha y^2\\ax^2+by^2+cz^2\end{array}$ & 
			& *
			& * \\ \hline
			$\begin{array}{c}
					I_{15}(\alpha,a,b)\\
					(\alpha^3\neq0,1,-8, \\
					\#(E_A)=6)
			\end{array}$ & EC-B & 
			$\begin{array}{c}xz+zy+\alpha yx\\zx+yz+\alpha xy\\y^2+x^2+\alpha z^2\\a(xy+yx)+bz^2 
			\end{array}$ & 
			& 
			 $\begin{array}{c}
			 	x+y \\ (b=-a\alpha ), \\
			 	z \\ ({\rm otherwise}) \\
			 \end{array}$ & $\emptyset$\\ \hline
		\end{tabular*}
		\begin{tablenotes}
			\item\hspace*{-\fontdimen2\font} $\delta(\alpha, \beta)=-27-256\alpha^3+288\alpha \beta+256\alpha^2\beta^2-256\beta^3$.
			\item\hspace*{-\fontdimen2\font} $I_1\cong I_4\cong I_6(\alpha)\cong I_9\cong I_{10}(\alpha)\cong I_{11}(\alpha, \beta)\cong I_{14}(\alpha,a,b)$, so $RZ((I_{14}(\alpha,a,b))^!)_1 \neq \emptyset$. 
		\end{tablenotes}
	\end{threeparttable}
\end{table}

\begin{table}[ht]
	\begin{threeparttable}
		\small
		\centering
		\caption{$A \in \cA^c_{3,1}$ s.t. $C(A) \cong k[u]/(u^2)\times k^2$.}  \label{tab-caj}
		\begin{tabular*}{0.95\linewidth} {@{\extracolsep{\fill}}cccccc}
			\hline
			$A$ & $S$ & $F$ & $G$ & $RN(A^!)_1$ & $RZ(A^!)_1$ \\ \hline \hline
			$J_1$ & S$'_1$ & 
			$\begin{array}{c}xy+yx\\yz+zy+x^2\\zx+xz\\x^2+y^2+z^2\end{array}$ & 
			$\begin{array}{c}xy-yx\\yz-zy\\zx-xz\\x^2-z^2-yz\\y^2-z^2\end{array}$ & 
			$z$& $z$ \\ \hline
			$ J_2$ & S$'_1$ &
			$\begin{array}{c}xy-yx\\yz+zy+x^2\\zx-xz\\x^2+y^2+z^2\end{array}$ & 
			$\begin{array}{c}xy+yx\\yz-zy\\zx+xz\\x^2-z^2-yz\\y^2-z^2\end{array}$ & 
			$y$ &$\emptyset$ \\ \hline
			$J_3$ & S$'_2$ &
			$\begin{array}{c}xy-zx\\x^2+y^2+z^2\\yx-xz\\x^2+yz+zy\end{array}$ & 
			$\begin{array}{c} xy+zx\\yz-zy\\yx+xz\\x^2-z^2-yz\\y^2-z^2\end{array}$ & 
			$x+y-z$ &$\emptyset$ \\ \hline
					$J_4$ & S$'_2$ &
			$\begin{array}{c}xy-zx\\x^2+y^2+z^2\\yx-xz\\x^2-yz-zy\end{array}$ & 
			$\begin{array}{c} xy+zx\\yz-zy\\yx+xz\\x^2-z^2+yz\\y^2-z^2\end{array}$ & 
			$\emptyset$ &$\emptyset$ \\ \hline
			${\begin{array}{c}
					J_5(\alpha)\\
					(\alpha=\pm\frac{16\sqrt{-3}}{9})
			\end{array}}$ & NC$_1$ &
			{$\begin{array}{c}xy+yx+\alpha z^2\\yz+zy\\zx+xz+y^2\\x^2+y^2\end{array}$} & 
			$\begin{array}{c}xy-yx\\yz-zy\\xz-zx\\x^2-y^2+zx\\z^2-\alpha xy\end{array}$& 
			$x$ & $x$ \\ \hline
			${\begin{array}{c}
					J_6(\alpha,\beta)\\
					(\textnormal{condition }\diamond)
			\end{array}} $ & NC$_1$ &
			$\begin{array}{c}xy+yx\\yz+zy+\tfrac{\sqrt{\beta}}{\alpha} x^2\\zx+xz+\tfrac{\sqrt{\alpha}}{\beta} y^2\\x^2+y^2+z^2\end{array}$ & 
			${\begin{array}{c}xy-yx\\yz-zy\\zx-xz\\ x^2-z^2-\tfrac{\sqrt{\beta}}{\alpha}yz\\y^2-z^2-\tfrac{\sqrt{\alpha}}{\beta}zx\end{array}}$ & 
			$z$ &$z$ \\ \hline
			$J_7$ & NC$_2$ & 
			${\begin{array}{c}2xy-zx+yz\\2yx-xz+zy\\x^2+y^2\\xy+yx+z^2\end{array}}$ & 
			{$\begin{array}{c}xy-yx+2zx-2xz\\yz+zx\\zy+xz\\x^2-y^2\\z^2-xy-2zx\end{array}$} & 
			$z$ & $\emptyset$\\ \hline
			${\begin{array}{c}
					J_8(\alpha,a,b)\\
					(\alpha^3\neq0,1,-8, \\
					\#(E_A)=4)
			\end{array}}$ & EC-A &
			$\begin{array}{c}xy+yx+\alpha z^2\\yz+zy+\alpha x^2\\zx+xz+\alpha y^2\\ax^2+by^2+cz^2\end{array}$ & 
			& *
			& * \\ \hline
			${\begin{array}{c}
					J_9(\alpha,a,b)\\
					(\alpha^3\neq0,1,-8, \\
					\#(E_A)=4)
			\end{array}}$ & EC-B & 
			$\begin{array}{c}xz+zy+\alpha yx\\zx+yz+\alpha xy\\y^2+x^2+\alpha z^2\\a(xy+yx)+bz^2\end{array}$ &
				&  	{$\begin{aligned}
					z\ &(b^2 \neq a^2\alpha^2) \\
					x+y\ &(b=-a\alpha)\\
					\emptyset\ &(b=a\alpha) 
				\end{aligned}$}
			& $\emptyset$\\ \hline
		\end{tabular*}
		\begin{tablenotes}
			\item\hspace*{-\fontdimen2\font} Condition $\diamond$: $\delta(\alpha,\beta) = -27-256\alpha^3+288\alpha \beta+256\alpha^2\beta^2-256\beta^3 = 0$, and $\alpha\neq\frac{3\mu}{4},\beta\neq\frac{3\mu^2}{4}$ where $\mu$ is a 3rd root of unit.
			\item\hspace*{-\fontdimen2\font} $J_1\cong J_5(\alpha) \cong  J_6(\alpha, \beta)\cong J_8(\alpha,a,b)$, so $RZ((J_8(\alpha,a,b))^!)_1 \neq \emptyset$. 
		\end{tablenotes}
	\end{threeparttable}
\end{table}

\begin{table}[ht]
	\begin{threeparttable}
		\small
		\centering
		\caption{$A \in \cA^c_{3,1}$ s.t. $C(A) \cong k[u]/(u^3)\times k$.}  \label{tab-cak}
		\begin{tabular*}{0.95\linewidth} {@{\extracolsep{\fill}}cccccc}
			\hline
			$A$ & $S$ & $F$ & $G$ & $RN(A^!)_1$ & $RZ(A^!)_1$ \\ \hline \hline
			$K_1$ & NC$_1$ & 
			$\begin{array}{c}xy+yx\\yz+zy+\sqrt{\tfrac{4}{3}}x^2\\zx+xz+\sqrt{\tfrac{4}{3}}y^2\\x^2+y^2+z^2\end{array}$ & 
			${ \begin{array}{c}xy-yx\\yz-zy\\zx-xz\\x^2-z^2-\sqrt{\tfrac{4}{3}}yz\\y^2-z^2-\sqrt{\tfrac{4}{3}}zx\end{array}}$ & 
			$z$ &$z$ \\ \hline
			$K_2$ & NC$_2$ & 
			${ \begin{array}{c}-\sqrt{\tfrac{4}{3}}xy+zx+yz\\-\sqrt{\tfrac{4}{3}}yx+xz+zy\\x^2+y^2\\xy+yx+z^2\end{array}}$ & 
			${ \begin{array}{c}xy-yx+\sqrt{\tfrac{4}{3}}zx-\sqrt{\tfrac{4}{3}}xz\\yz-zx\\zy-xz\\x^2-y^2\\z^2-xy-\sqrt{\tfrac{4}{3}}zx\end{array}}$ & 
			$z$ & $\emptyset$\\ \hline
			${ \begin{array}{c}
					K_3(\alpha,a,b)\\
					(\alpha^3\neq0,1,-8, \\
					\#(E_A)=2)
			\end{array}}$ & EC-A&
			$\begin{array}{c}xy+yx+\alpha z^2\\yz+zy+\alpha x^2\\zx+xz+\alpha y^2\\ax^2+by^2+cz^2\end{array}$ & 
			& *
			& * \\ \hline
			${ \begin{array}{c}
					K_4(\alpha,a,b)\\
					(\alpha^3\neq0,1,-8, \\
					\#(E_A)=2)
			\end{array}}$ & EC-B & 
			$\begin{array}{c}xz+zy+\alpha yx\\zx+yz+\alpha xy\\x^2+y^2+\alpha z^2\\a(xy+yx)+bz^2\end{array}$ & 
			& {$z$} & $\emptyset$ \\ \hline
		\end{tabular*}
		\begin{tablenotes}
			\item\hspace*{-\fontdimen2\font} $K_1\cong K_3(\alpha,a,b)$, so $ RZ((K_3(\alpha,a,b))^!)_1 \neq \emptyset$.
		\end{tablenotes}
	\end{threeparttable}
\end{table}



\end{document}